\documentclass[12pt]{article}

\usepackage{amsmath, amsthm, amssymb}
\usepackage{enumerate}
\usepackage{ifpdf}
\ifpdf
\usepackage[pdftex]{graphicx}
\else
\usepackage[dvips]{graphicx}
\fi
\usepackage{tikz}
 	 \usetikzlibrary{arrows,backgrounds}
\usepackage[all]{xy}

\input xy
\xyoption{all}

\usepackage[pdftex,plainpages=false,hypertexnames=false,pdfpagelabels]{hyperref}
\newcommand{\arxiv}[1]{\href{http://arxiv.org/abs/#1}{\tt arXiv:\nolinkurl{#1}}}
\newcommand{\arXiv}[1]{\href{http://arxiv.org/abs/#1}{\tt arXiv:\nolinkurl{#1}}}
\newcommand{\doi}[1]{\href{http://dx.doi.org/#1}{{\tt DOI:#1}}}

\newcommand{\googlebooks}[1]{(preview at \href{http://books.google.com/books?id=#1}{google books})}

\usepackage{xcolor}
\definecolor{dark-red}{rgb}{0.7,0.25,0.25}
\definecolor{dark-blue}{rgb}{0.15,0.15,0.55}
\definecolor{medium-blue}{rgb}{0,0,.8}
\definecolor{DarkGreen}{RGB}{0,150,0}
\definecolor{rho}{named}{red}
\definecolor{theta}{named}{blue}
\hypersetup{
   colorlinks, linkcolor={purple},
   citecolor={medium-blue}, urlcolor={medium-blue}
}

\usepackage{longtable}
\usepackage{fullpage}
\theoremstyle{plain}
\newtheorem{thm}{Theorem}[section]
\newtheorem*{thm*}{Theorem}
\newtheorem{thmalpha}{Theorem}

\newtheorem{cor}[thm]{Corollary}

\newtheorem*{cor*}{Corollary}
\newtheorem{conj}[thm]{Conjecture}
\newtheorem{conjalpha}[thmalpha]{Conjecture}
\newtheorem*{conj*}{Conjecture}
\newtheorem{lem}[thm]{Lemma}

\newtheorem{facts}[thm]{Facts}
\newtheorem{prop}[thm]{Proposition}

\newtheorem*{quest*}{Question}
\newtheorem*{claim*}{Claim}

\theoremstyle{definition}
\newtheorem{defn}[thm]{Definition}
\newtheorem{assumption}[thm]{Assumption}

\newtheorem{ex}[thm]{Example}
\newtheorem{rem}[thm]{Remark}
\newtheorem{remark}[thm]{Remark}
\newtheorem{rems}[thm]{Remarks}

\DeclareMathOperator{\coeff}{coeff}

\DeclareMathOperator{\Hom}{Hom}
\DeclareMathOperator{\spann}{span}
\DeclareMathOperator{\id}{id}

\DeclareMathOperator{\Tr}{Tr}

\newcommand{\D}{\displaystyle}
\newcommand{\comment}[1]{}

\newcommand{\be}{\begin{enumerate}[(1)]}
\newcommand{\ee}{\end{enumerate}}

\newcommand{\itt}[1]{\item[\underline{\text{#1}:}]}

\newcommand{\set}[2]{\left\{#1 \middle| #2\right\}}

\def\semicolon{;}
\def\applytolist#1{
    \expandafter\def\csname multi#1\endcsname##1{
        \def\multiack{##1}\ifx\multiack\semicolon
            \def\next{\relax}
        \else
            \csname #1\endcsname{##1}
            \def\next{\csname multi#1\endcsname}
        \fi
        \next}
    \csname multi#1\endcsname}

\def\calc#1{\expandafter\def\csname c#1\endcsname{{\mathcal #1}}}
\applytolist{calc}QWERTYUIOPLKJHGFDSAZXCVBNM;
\def\bbc#1{\expandafter\def\csname bb#1\endcsname{{\mathbb #1}}}
\applytolist{bbc}QWERTYUIOPLKJHGFDSAZXCVBNM;
\def\bfc#1{\expandafter\def\csname bf#1\endcsname{{\mathbf #1}}}
\applytolist{bfc}QWERTYUIOPLKJHGFDSAZXCVBNM;
\def\sfc#1{\expandafter\def\csname s#1\endcsname{{\sf #1}}}
\applytolist{sfc}QWERTYUIOPLKJHGFDSAZXCVBNM;

\newcommand{\Pro}{{\sf Pro}}

\renewcommand{\Vec}{{\sf Vec}}
\newcommand{\jw}[1]{f^{(#1)}}

\newcommand{\noshow}[1]{}
\newcommand{\MR}[1]{}
\newcommand{\TL}{\cT\hspace{-.08cm}\cL}
\newcommand{\rhoE}{\textcolor{rho}{e_1}}
\newcommand{\rhoJW}{\textcolor{rho}{\jw{2}}}

\usetikzlibrary{shapes}
\usetikzlibrary{backgrounds}
\usetikzlibrary{decorations,decorations.pathreplacing,decorations.markings}
\usetikzlibrary{fit,calc,through}
\usetikzlibrary{external}
\tikzset{
	super thick/.style={line width=3pt}
}
\tikzstyle{shaded}=[fill=red!10!blue!20!gray!30!white]
\tikzstyle{unshaded}=[fill=white]
\tikzstyle{empty box}=[circle, draw, thick, fill=white, opaque, inner sep=2mm]
\tikzstyle{annular}=[scale=.7, inner sep=1mm, baseline]
\tikzstyle{rectangular}=[scale=.75, inner sep=1mm, baseline=-.1cm]
\tikzstyle{mid>}=[decoration={markings, mark=at position 0.5 with {\arrow{>}}}, postaction={decorate}]
\tikzstyle{mid<}=[decoration={markings, mark=at position 0.5 with {\arrow{<}}}, postaction={decorate}]

\newcommand{\nbox}[6]{
	\draw[thick, #1] ($#2+(-#3,-#3)+(-#4,0)$) rectangle ($#2+(#3,#3)+(#5,0)$);
	\coordinate (ZZa) at ($#2+(-#4,0)$);
	\coordinate (ZZb) at ($#2+(#5,0)$);
	\node at ($1/2*(ZZa)+1/2*(ZZb)$) {#6};
}

\newcommand{\ncircle}[5]{
	\draw[thick, #1] #2 circle (#3);
	\node at #2 {#5};
	\node at ($#2+(#4:.15cm)+(#4:#3cm)$) {$\star$};
}

\begin{document}
\title{The generator conjecture for $3^G$ subfactor planar algebras}
\author{Zhengwei Liu and David Penneys}
\date{\today}
\maketitle
\begin{abstract}
We state a conjecture for the formulas of the depth 4 low-weight rotational eigenvectors and their corresponding eigenvalues for the $3^G$ subfactor planar algebras.
We prove the conjecture in the case when $|G|$ is odd. 
To do so, we find an action of $G$ on the reduced subfactor planar algebra at $\jw{2}$, which is obtained from shading the planar algebra of the even half.
We also show that this reduced subfactor planar algebra is a Yang-Baxter planar algebra.
%This is the submitted version of \arxiv{}
\end{abstract}
%\tableofcontents

Dedicated to the 60th birthday of Vaughan F. R. Jones.
%%%%%%%%%%%%%%%%%%%%%%%%%%%%%%%%%%%%%%%%%%%%%%%%%%
%%%%%%%%%%%%%%%%%%%%%%%%%%%%%%%%%%%%%%%%%%%%%%%%%%
%%%%%%%%%%%%%%%%%%%%%%%%%%%%%%%%%%%%%%%%%%%%%%%%%%
\section{Introduction}

Haagerup initiated the classification of subfactor principal graphs with index a little greater than 4, and he gave a classification of all possible graph pairs in the index range $(4,3+\sqrt{2})$ \cite{MR1317352}.
In doing so, he discovered a so-called `exotic' subfactor \cite{MR1686551} with index $\frac{5+\sqrt{13}}{2}$ and principal graph the 3-spoke
$$
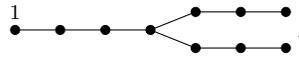
\begin{tikzpicture}[baseline=-.1cm,scale=.6]
	\filldraw (0,0) circle (.1cm) node [above] {{\scriptsize{$1$}}};
	\filldraw (1,0) circle (.1cm);
	\filldraw (2,0) circle (.1cm);
	\filldraw (3,0) circle (.1cm);
	\filldraw (4,.4) circle (.1cm);
	\filldraw (4,-.4) circle (.1cm);
	\filldraw (5,.4) circle (.1cm);
	\filldraw (5,-.4) circle (.1cm);
	\filldraw (6,.4) circle (.1cm);
	\filldraw (6,-.4) circle (.1cm);
	\draw (0,0)--(3,0);
	\draw (6,-.4)--(4,-.4)--(3,0)--(4,.4)--(6,.4);
\end{tikzpicture}\,.
$$
The $\bbZ/3\bbZ$-symmetry of this graph means that the dimension one vertices at the ends of the spokes form the group $\bbZ/3\bbZ$ under the fusion operation of the corresponding bimodules.

In \cite{MR1832764}, Izumi gave a generalized construction to other abelian groups using Cuntz algebras, and constructed an example when $G=\bbZ/5\bbZ$.
Such a subfactor is called a $3^G$ subfactor and has principal graph
$$
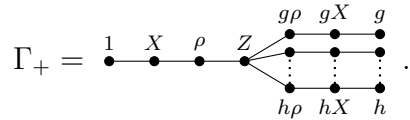
\Gamma_+=
\begin{tikzpicture}[baseline=-.1cm,scale=.6]
	\filldraw (0,0) circle (.1cm) node [above] {{\scriptsize{$1$}}};
	\filldraw (1,0) circle (.1cm) node [above] {{\scriptsize{$X$}}};
	\filldraw (2,0) circle (.1cm) node [above] {{\scriptsize{$\rho$}}};
	\filldraw (3,0) circle (.1cm) node [above] {{\scriptsize{$Z$}}};
	\filldraw (4,.6) circle (.1cm) node [above] {{\scriptsize{$g\rho$}}};
	\filldraw (4,.2) circle (.1cm);
	\filldraw (4,-.6) circle (.1cm) node [below] {{\scriptsize{$h\rho$}}};
	\filldraw (5,.6) circle (.1cm) node [above] {{\scriptsize{$gX$}}};
	\filldraw (5,.2) circle (.1cm);
	\filldraw (5,-.6) circle (.1cm) node [below] {{\scriptsize{$hX$}}};
	\filldraw (6,.6) circle (.1cm) node [above] {{\scriptsize{$g$}}};
	\filldraw (6,.2) circle (.1cm);
	\filldraw (6,-.6) circle (.1cm) node [below] {{\scriptsize{$h$}}};
	\draw[] (0,0)--(3,0);
	\draw (6,.6)--(4,.6)--(3,0)--(4,.2)--(6,.2);
	\draw (6,-.6)--(4,-.6)--(3,0);
	\draw[thick, dotted]  (4,.2)--(4,-.6);
	\draw[thick, dotted]  (5,.2)--(5,-.6);
	\draw[thick, dotted]  (6,.2)--(6,-.6);
\end{tikzpicture}\,.
$$
Two recent articles of Evans-Gannon \cite{MR2837122,MR3167494} have successfully used Izumi's equations to construct a myriad of new examples of $3^G$ and related $2^G1$ subfactors with principal graphs
$$
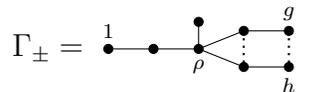
\Gamma_\pm=
\begin{tikzpicture}[baseline=-.1cm,scale=.6]
	\filldraw (1,0) circle (.1cm) node [above] {{\scriptsize{$1$}}};
	\filldraw (2,0) circle (.1cm);
	\filldraw (3,0) circle (.1cm) node [below] {{\scriptsize{$\rho$}}};
	\filldraw (3,.6) circle (.1cm);
	\filldraw (4,.4) circle (.1cm);
	\filldraw (4,-.4) circle (.1cm);
	\filldraw (5,.4) circle (.1cm) node [above] {{\scriptsize{$g$}}};
	\filldraw (5,-.4) circle (.1cm) node [below] {{\scriptsize{$h$}}};
	\draw[] (1,0)--(3,0);
	\draw (3,.6)--(3,0)--(4,.4)--(5,.4);
	\draw (5,-.4)--(4,-.4)--(3,0);
	\draw[thick, dotted]  (4,.4)--(4,-.4);
	\draw[thick, dotted]  (5,.4)--(5,-.4);
\end{tikzpicture}\,.
$$
They also give simple formulas for the quantum double and its modular data, leading them to conjecture that there should be associated rational conformal field theories.

The even halves of the $3^G$ and $2^G1$ subfactors are examples of \underline{quadratic} unitary fusion categories, which have a group $G$ of invertible objects and one other orbit $G\rho$ of simple objects, together with a relation for fusion on the right by $g$, and a quadratic fusion relation for $\rho$. 
For example, by unpublished work of Izumi, the even half of a $3^G$ subfactor for $|G|$ odd satisfies 
$$
\rho g=g^{-1}\rho \text{ and }\rho^2\cong 1\oplus \bigoplus_{g\in G} g\rho.
$$
The even halves of $2^G1$ subfactors are unitary near group fusion categories \cite{MR3167494}, which are generalizations of Tambara-Yamagami  categories \cite{MR1659954}.

Izumi observed in \cite{MR1832764} that when realizing a unitary fusion category as a category of sectors of some infinite factor $M$, Cuntz algebras naturally appear as the C*-algebras generated by the orthonormal bases of intertwiner spaces in $M$.
Cuntz algebras are particularly useful in constructing quadratic categories because we usually only need to analyze one Cuntz algebra, in which the quadratic relation allows us to write down polynomial equations in the generators to define an endomorphism of the C*-algebra.
One then extends the endomorphisms to the von Neumann completion using the unique KMS state \cite{MR500150}, which is again an infinite factor (e.g., see \cite{MR1441540}).
When the category is not quadratic, we obtain multiple Cuntz algebras together with relations between them, and the situation is much more complicated.

Currently planar algebra techniques are not as effective as Cuntz algebras for constructing quadratic categories.
The recent articles \cite{MR3157990,1308.5197} suggest a uniform skein theory for the $3^G$'s using 2-strand jellyfish relations.
A general formula for the generators in the graph planar algebras remains elusive, as the valence and size of the $2^n1$ and $3^n$ graphs gets quite large.
We expect that the Cuntz algebra and planar algebra techniques can be reconciled, which will be explored in future work.
For example, Izumi has shown how to draw planar diagrams for the actions of his Cuntz algebra endomorphisms.

Based on \cite{MR3314808,1308.5197} we conjecture specific formulas for the $3^G$ low-weight rotational eigenvectors in the 4-box spaces in terms of minimal projections in the $3^G$ subfactor planar algebras.
This is the first step in the Jones-Peters graph planar algebra embedding program \cite{MR1929335,MR2679382,MR2812459,MR2972458} toward a uniform planar algebraic approach to the $3^G$ subfactors.

Let $\cP_\bullet$ be a $3^G$ subfactor planar algebra.
For $g\in G\setminus\{1\}$, let $p_g$ be the projection in $\cP_{4,+}$ corresponding to $g\rho$.
We make the following conjecture about the low-weight rotational eigenvectors for $\cP_\bullet$ which agrees with the Haagerup $3^{\bbZ/3\bbZ}$ and Izumi $3^{\bbZ/2\bbZ\times \bbZ/2\bbZ}$ and $3^{\bbZ/4\bbZ}$ subfactor planar algebras by \cite{MR2679382,MR2972458,MR3314808,1308.5197}.

\begin{conjalpha}\label{conj:Main}
Suppose $g,h,k,\ell \in G\setminus\{1\}$ are distinct elements.
\begin{enumerate}[(1)]
\item
If $g=g^{-1}$ and $h=h^{-1}$, then $p_g-p_h$ is a low-weight rotational eigenvector with eigenvalue 1.
\item
If $g=h^{-1}$, then $p_g-p_h$ is a low-weight rotational eigenvector with eigenvalue -1.
\item
If $g=g^{-1}$ and $h=k^{-1}$, then $2p_g-(p_h+p_k)$ is a low-weight rotational eigenvector with eigenvalue 1.
\item
If $g=h^{-1}$ and $k=\ell^{-1}$, then $(p_g+p_h)-(p_k+p_\ell)$ is a low-weight rotational eigenvector with eigenvalue 1.
\end{enumerate}
\end{conjalpha}

Our main theorem in this article is as follows.

\begin{thmalpha}\label{thm:Main}
Conjecture \ref{conj:Main} is true when $|G|$ is odd.
\end{thmalpha}

To prove this theorem, we construct a $G$-action\footnote{\label{footnote:SymmetricSelfDuality}There is actually a technicality involving the shading and the symmetric self-duality \cite{MR3314808,MPAffineAandD} of $\cR_\bullet$ which we address in Sections \ref{sec:ActionOnR} and \ref{sec:ActionOnS}, but we omit the shading in the introduction to give the spirit of the argument.}
on the reduced unshaded planar algebra $\cR_\bullet$ of $\cP_\bullet$ at $\rho=\jw{2}$ with principal graphs
$$
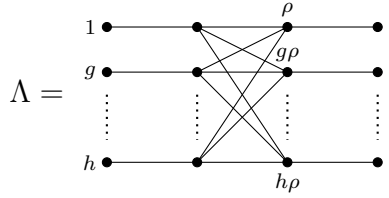
\Lambda=
\begin{tikzpicture}[baseline=-1cm,scale=.6]
	\filldraw (0,0) circle (.1cm) node [left] {{\scriptsize{$1$}}};
	\filldraw (0,-1) circle (.1cm) node [left] {{\scriptsize{$g$}}};
	\filldraw (0,-3) circle (.1cm) node [left] {{\scriptsize{$h$}}};
	\filldraw (2,0) circle (.1cm);
	\filldraw (2,-1) circle (.1cm);
	\filldraw (2,-3) circle (.1cm);
	\filldraw (4,0) circle (.1cm) node [above] {{\scriptsize{$\rho$}}};
	\filldraw (4,-1) circle (.1cm) node [above] {{\scriptsize{$g\rho$}}};
	\filldraw (4,-3) circle (.1cm) node [below] {{\scriptsize{$h\rho$}}};
	\filldraw (6,0) circle (.1cm);
	\filldraw (6,-1) circle (.1cm);
	\filldraw (6,-3) circle (.1cm);
	\draw (0,0) -- (6,0);
	\draw (0,-1) -- (6,-1);
	\draw (0,-3) -- (6,-3);
	\draw (2,0) -- (4,-1);
	\draw (2,0) -- (4,-3);
	\draw (2,-1) -- (4,0);
	\draw (2,-1) -- (4,-3);
	\draw (2,-3) -- (4,0);
	\draw (2,-3) -- (4,-1);
	\draw[thick, dotted]  (0,.-1.5)--(0,-2.5);
	\draw[thick, dotted]  (2,-1.5)--(2,-2.5);
	\draw[thick, dotted]  (4,-1.5)--(4,-2.5);
	\draw[thick, dotted]  (6,-1.5)--(6,-2.5);
\end{tikzpicture}\,.
$$
For each $g\in G$, we pick a distinguished isomorphism $V_g: \rho g^{-1}\to g\rho$.
Denoting $\rho\in \cR_1$ by a red strand and the group elements $g\in \cP_{6,+}$ by black labelled, oriented strands, the action is given by
$$
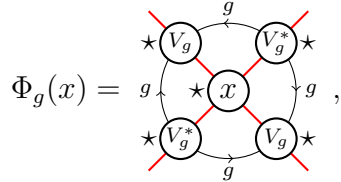
\Phi_g(x) = 
\begin{tikzpicture}[baseline = -.1cm]
	\pgfmathsetmacro{\innerRadius}{.9}
	\pgfmathsetmacro{\middleRadius}{1.1}
	\pgfmathsetmacro{\outerRadius}{1.5}
	\draw[mid>] (225:\innerRadius) arc (225:135:\innerRadius);
	\draw[mid<] (135:\innerRadius) arc (135:45:\innerRadius);
	\draw[mid>] (45:\innerRadius) arc (45:-45:\innerRadius);
	\draw[mid<] (315:\innerRadius) arc (315:225:\innerRadius);
	\draw[thick, rho] (0,0) -- (45:\outerRadius);
	\draw[thick, rho] (0,0) -- (135:\outerRadius);
	\draw[thick, rho] (0,0) -- (225:\outerRadius);
	\draw[thick, rho] (0,0) -- (315:\outerRadius);
	\node at (0:\middleRadius) {\scriptsize{$g$}};
	\node at (90:\middleRadius) {\scriptsize{$g$}};
	\node at (180:\middleRadius) {\scriptsize{$g$}};
	\node at (270:\middleRadius) {\scriptsize{$g$}};
	\ncircle{unshaded}{(45:\innerRadius)}{.27}{0}{{\scriptsize{$V_{g}^*$}}}
	\ncircle{unshaded}{(135:\innerRadius)}{.27}{180}{{\scriptsize{$V_g$}}}
	\ncircle{unshaded}{(225:\innerRadius)}{.27}{180}{{\scriptsize{$V_{g}^*$}}}
	\ncircle{unshaded}{(315:\innerRadius)}{.27}{0}{{\scriptsize{$V_g$}}}
	\ncircle{unshaded}{(0,0)}{.27}{180}{$x$}
\end{tikzpicture}
\,,
$$
which is similar to diagrams arising from looking at connections \cite{MR996454,math/9909027,MR3254427,1308.5656}. 
Moreover, we have $\Phi_g\circ \Phi_h = \Phi_{gh}$, giving a $G$-action on $\cR_\bullet$\textsuperscript{\ref{footnote:SymmetricSelfDuality}}.
We anticipate this new technique will have new applications to subfactor planar algebras beyond the proof of our theorem.

When there is an $h\in G$ such that $h^2=g$, we apply the action of $\Phi_h$ to the relation
$$
\begin{tikzpicture}[baseline = -.1cm]
	\filldraw[rho] (.2,0) circle (.05cm);
	\filldraw[rho] (-.2,0) circle (.05cm);
	\draw[thick, rho] (-.2,-.4) -- ( -.2, .4);	
	\draw[thick, rho] (.2,-.4) -- (.2, .4);	
	\draw[thick, rho] (-.2,0) -- (.2,0);
	\nbox{}{(0,0)}{.4}{0}{0}{}
\end{tikzpicture}
=
\begin{tikzpicture}[baseline = -.1cm]
	\filldraw[rho] (0,.2) circle (.05cm);
	\filldraw[rho] (0,-.2) circle (.05cm);
	\draw[thick, rho] (-.2,.4) -- (0,.2) --  ( .2, .4);	
	\draw[thick, rho] (-.2,-.4) -- (0,-.2) --  ( .2, -.4);	
	\draw[thick, rho] (0,-.2) -- (0, .2);
	\nbox{}{(0,0)}{.4}{0}{0}{}
\end{tikzpicture}
- \frac{1}{[3]-1} \left(\,
\begin{tikzpicture}[baseline = -.1cm]
	\nbox{unshaded}{(0,0)}{.4}{0}{0}{}
	\draw[thick, rho] (-.2,.4) arc (-180:0:.2);
	\draw[thick, rho] (-.2,-.4) arc (180:0:.2);	
\end{tikzpicture}
-
\begin{tikzpicture}[baseline = -.1cm]
	\draw[thick, rho] (-.2,-.4) -- ( -.2, .4);	
	\draw[thick, rho] (.2,-.4) -- (.2, .4);
	\nbox{}{(0,0)}{.4}{0}{0}{}
\end{tikzpicture}
\,\right),
$$
where the trivalent vertex is a suitably normalized map $\rho\otimes \rho\to \rho$.
We then obtain the formula
$$
\cF_{\cR_\bullet}(p_g)
= 
p_{g^{-1}} - \frac{1}{[3]-1} \left(\,
\begin{tikzpicture}[baseline = -.1cm]
	\nbox{unshaded}{(0,0)}{.4}{0}{0}{}
	\draw[thick, rho] (-.2,.4) arc (-180:0:.2);
	\draw[thick, rho] (-.2,-.4) arc (180:0:.2);	
\end{tikzpicture}
-
\begin{tikzpicture}[baseline = -.1cm]
	\draw[thick, rho] (-.2,-.4) -- ( -.2, .4);	
	\draw[thick, rho] (.2,-.4) -- (.2, .4);
	\nbox{}{(0,0)}{.4}{0}{0}{}
\end{tikzpicture}
\,\right).
$$
By Corollary \ref{cor:EquivalentConjecture} below, Conjecture \ref{conj:Main} is equivalent to this formula holding for every $g\in G$.
When $|G|$ is odd, every $g\in G$ has a square root, but this is no longer the case when $|G|$ is even.

We note that the principal graph $\Lambda$ above bears a strong resemblance to the principal graph of the reduced subfactor of $A_7$ at $\jw{2}$, given by
$$
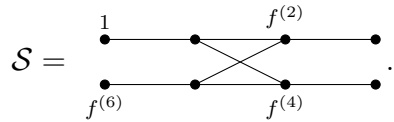
\cS=
\begin{tikzpicture}[baseline=-.4cm,scale=.6]
	\filldraw (0,0) circle (.1cm) node [above] {{\scriptsize{$1$}}};
	\filldraw (0,-1) circle (.1cm) node [below] {{\scriptsize{$\jw{6}$}}};
	\filldraw (2,0) circle (.1cm);
	\filldraw (2,-1) circle (.1cm);
	\filldraw (4,0) circle (.1cm) node [above] {{\scriptsize{$\jw{2}$}}};
	\filldraw (4,-1) circle (.1cm) node [below] {{\scriptsize{$\jw{4}$}}};
	\filldraw (6,0) circle (.1cm);
	\filldraw (6,-1) circle (.1cm);
	\draw (0,0) -- (6,0);
	\draw (0,-1) -- (6,-1);
	\draw (2,0) -- (4,-1);
	\draw (2,-1) -- (4,0);
\end{tikzpicture}\,.
$$
It was recently shown in \cite{MR3306607} that the planar algebra corresponding to this reduced subfactor is a \underline{Yang-Baxter planar algebra}, which is a planar algebra generated by 2-boxes with a relation which writes one type of triangle in terms of the other triangle and lower order terms (below, $\cB_{2,+}$ is a basis of $\cP_{2,+}$ including the diagrammatic basis of $\cT\cL_{2,+}$):
$$
\begin{tikzpicture}[baseline=-.1cm]
	\pgfmathsetmacro{\width}{.8};
	\pgfmathsetmacro{\height}{.8};
	\coordinate (a) at ($ 3/4*(0,-\height) $);
	\coordinate (b) at ($ 3/4*(0,\height) $);
	\coordinate (c) at (\width,0);

	\fill[shaded] ($(a)+3/4*(0,-\height)$) -- ($(b)+3/4*(0,\height)$) -- ($(b)+3/4*(0,\height) + 1/2*(\width,0)$) -- (b) -- (c) -- (a) -- ($(a)+3/4*(0,-\height) + 1/2*(\width,0)$);
	\fill[shaded] ($(a)+3/4*(0,-\height)+(\width,0)$) -- ($(b)+3/4*(0,\height)+(\width,0)$) -- ($(b)+3/4*(0,\height)+3/2*(\width,0)$) -- ($(a)+3/4*(0,-\height)+3/2*(\width,0)$);

	\draw (b) -- (c) -- (a);
	\draw ($(a)+3/4*(0,-\height)$) -- ($(b)+3/4*(0,\height)$);
	\draw ($(a)+3/4*(0,-\height)+(\width,0)$) -- ($(b)+3/4*(0,\height)+(\width,0)$);
	\draw (a) -- ($(a)+3/4*(0,-\height) + 1/2*(\width,0)$);
	\draw (b) -- ($(b)+3/4*(0,\height) + 1/2*(\width,0)$);
	\ncircle{unshaded}{(b)}{.3}{180}{$a$}
	\ncircle{unshaded}{(c)}{.3}{115}{$b$}
	\ncircle{unshaded}{(a)}{.3}{180}{$c$}
\end{tikzpicture}
=
\sum_{x,y,z\in \cB_{2,+}}
\lambda_{x,y,z}
\begin{tikzpicture}[xscale=-1, baseline=-.1cm]
	\pgfmathsetmacro{\width}{.8};
	\pgfmathsetmacro{\height}{.8};
	\coordinate (a) at ($ 3/4*(0,-\height) $);
	\coordinate (b) at ($ 3/4*(0,\height) $);
	\coordinate (c) at (\width,0);

	\fill[shaded] ($(a)+3/4*(0,-\height)+ (\width,0)$) -- ($(b)+3/4*(0,\height)+ (\width,0)$) -- ($(b)+3/4*(0,\height) + 1/2*(\width,0)$) -- (b) -- (c) -- (a) -- ($(a)+3/4*(0,-\height) + 1/2*(\width,0)$);
	\fill[shaded] ($(a)+3/4*(0,-\height)$) -- ($(b)+3/4*(0,\height)$) -- ($(b)+3/4*(0,\height)+1/2*(-\width,0)$) -- ($(a)+3/4*(0,-\height)+1/2*(-\width,0)$);

	\draw (b) -- (c) -- (a);
	\draw ($(a)+3/4*(0,-\height)$) -- ($(b)+3/4*(0,\height)$);
	\draw ($(a)+3/4*(0,-\height)+(\width,0)$) -- ($(b)+3/4*(0,\height)+(\width,0)$);
	\draw (a) -- ($(a)+3/4*(0,-\height) + 1/2*(\width,0)$);
	\draw (b) -- ($(b)+3/4*(0,\height) + 1/2*(\width,0)$);
	\ncircle{unshaded}{(b)}{.3}{73}{$x$}
	\ncircle{unshaded}{(c)}{.3}{0}{$y$}
	\ncircle{unshaded}{(a)}{.3}{65}{$z$}
\end{tikzpicture}\,.
$$
(There is also be a relation swapping the above types of triangles, which is necessary to be able to evaluate all closed diagrams.)
In his recent classification of singly generated Yang-Baxter planar algebras \cite{YBPA}, Liu discovered that the subfactor for $\cS$ belongs to an infinite family of subfactors arising from the $E_{N+2}$ quantum subgroup of $SU(N)$.

We further conjecture a new skein theoretic approach to constructing the reduced subfactor planar algebra $\cR_\bullet$ of a $3^G$ subfactor planar algebra, and we prove it in the case $|G|$ is odd.

\begin{conjalpha}\label{conj:YB}
$\cR_\bullet$ is a Yang-Baxter planar algebra with $|G|-1$ generators.
\end{conjalpha}

\begin{thmalpha}\label{thm:YB}
Conjecture \ref{conj:YB} is true when $|G|$ is odd.
\end{thmalpha}

%%%%%%%%%%%%%%%%%%%%%%%%%%%%%%%%%%%%%%%%%%%%%%%%%%%%%
\subsection{Acknowledgements}

The authors would like to thank Masaki Izumi and Emily Peters for helpful conversations; the Banff International Research Station for hosting the 2014 workshop on Subfactors and Fusion Categories, where this work started; and the 2014 Maui conference on Subfactor Theory in Mathematics and Physics, which was funded by NSF grant DMS-1400275 and DOD-DARPA grant HR0011-12-1-0009.
Zhengwei Liu was supported by NSF grant DMS-1001560.
David Penneys was partially supported by an AMS Simons travel grant and NSF grant DMS-1500387.
Both authors were supported by DOD-DARPA grant HR0011-12-1-0009.

%%%%%%%%%%%%%%%%%%%%%%%%%%%%%%%%%%%%%%%%%%%%%%%%%%%%%%%%%%%%%%%%%%
%%%%%%%%%%%%%%%%%%%%%%%%%%%%%%%%%%%%%%%%%%%%%%%%%%%%%%%%%%%%%%%%%%
%%%%%%%%%%%%%%%%%%%%%%%%%%%%%%%%%%%%%%%%%%%%%%%%%%%%%%%%%%%%%%%%%%
\section{$3^G$ subfactors}

\begin{defn}
Let $G$ be a non-trivial finite group.
A \underline{$3^G$ subfactor planar algebra} is a subfactor planar algebra whose principal graph $\Gamma_+$ is a $3^{|G|}$ spoke graph
$$
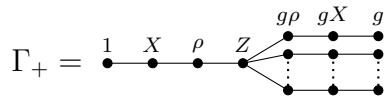
\Gamma_+=
\begin{tikzpicture}[baseline=-.1cm,scale=.6]
	\filldraw (0,0) circle (.1cm) node [above] {{\scriptsize{$1$}}};
	\filldraw (1,0) circle (.1cm) node [above] {{\scriptsize{$X$}}};
	\filldraw (2,0) circle (.1cm) node [above] {{\scriptsize{$\rho$}}};
	\filldraw (3,0) circle (.1cm) node [above] {{\scriptsize{$Z$}}};
	\filldraw (4,.6) circle (.1cm) node [above] {{\scriptsize{$g\rho$}}};
	\filldraw (4,.2) circle (.1cm);
	\filldraw (4,-.6) circle (.1cm);
	\filldraw (5,.6) circle (.1cm) node [above] {{\scriptsize{$gX$}}};
	\filldraw (5,.2) circle (.1cm);
	\filldraw (5,-.6) circle (.1cm);
	\filldraw (6,.6) circle (.1cm) node [above] {{\scriptsize{$g$}}};
	\filldraw (6,.2) circle (.1cm);
	\filldraw (6,-.6) circle (.1cm);
	\draw[] (0,0)--(3,0);
	\draw (6,.6)--(4,.6)--(3,0)--(4,.2)--(6,.2);
	\draw (6,-.6)--(4,-.6)--(3,0);
	\draw[thick, dotted]  (4,.2)--(4,-.6);
	\draw[thick, dotted]  (5,.2)--(5,-.6);
	\draw[thick, dotted]  (6,.2)--(6,-.6);
\end{tikzpicture}
$$
where the even bimodules generate a $G$-quadratic category, denoted $\frac{1}{2}\cP_+$, whose fusion rules are
\be
\item
$g\otimes h =  gh$, i.e., we may identify the dimension 1 bimodules with $G$,

\item
$g \otimes \rho = g\rho$, so $\set{g \rho}{g\in G}$ is a left $G$-set the obvious way: $g(h\rho)=(gh)\rho$,

\item
$\rho\otimes g = \theta(g)\rho$
for some automorphism $\theta$ of $G$, since $\rho g$ is irreducible by Frobenius reciprocity, and

\item
$\rho \otimes \rho = 1\oplus \bigoplus_{g\in G} g\rho$.
\ee
\end{defn}

\begin{conj}[Izumi]\label{conj:Izumi}
If a $3^G$ subfactor planar algebra exists, then $G$ is abelian, and $\theta(g)=g^{-1}$ for all $g\in G$.
\end{conj}

\begin{remark}
In unpublished work, Izumi has proven Conjecture \ref{conj:Izumi} for the case $|G|$ odd.
\end{remark}

\begin{assumption}
We will assume $G$ is abelian and $\theta(g)=g^{-1}$ for all $g\in G$.
\end{assumption}

\begin{cor}\label{cor:SelfDual}
Every bimodule at depth 4 is self-dual.
\end{cor}
\begin{proof}
For all $g\in G$, 
$\overline{g\rho}=\rho g^{-1}=\theta(g^{-1})\rho = g\rho$.
\end{proof}

An argument of Izumi gives the structure of the dual principal graph.
We provide a proof for the reader's convenience.
We begin with a helpful lemma generalizing \cite[Lemma 3.6]{MR2914056}.
For a pair of principal graphs $(\Gamma_+,\Gamma_-)$ of a subfactor planar algebra $\cP_\bullet$, let $\Gamma_\pm(n)$ denote the truncation to depth $n$.
Denote the one-click rotation by $\cF$.

\begin{lem}\label{lem:SimplyLacedTruncation}
Suppose $\Gamma_\pm$ is exactly $(n-1)$ supertransitive for an even $n\geq 2$.
If $\Gamma_+(n)$ is simply-laced and only has self-dual vertices at depth $n$, then $\Gamma_-(n)$ is simply-laced and only has self-dual vertices at depth $n$. 
\end{lem}
\begin{proof}
We analyze the rotation by $\pi$ given by $\cF^n$ on $\cP_{n,\pm}$.
Since $\Gamma_+(n)$ is simply-laced and only has self-dual vertices, $\cF^n$ is the identity on $\cP_{n,+}$.
This means $\cF^n$ is also the identity on $\cP_{n,-}$.
Observe that for $x\in \cP_{n,-}$, 
$$
\cF^n(x)=\cF^n(\cF(\cF^{-1}(x))) = \cF(\cF^n(\cF^{-1}(x)))=\cF(\cF^{-1}(x))=x.
$$
But $\cF^n$ is also an anti-isomorphism of the algebra $\cP_{n,-}$.
Thus $\cP_{n,-}$ must be abelian, which shows $\Gamma_-(n)$ is simply laced.
The rest follows from \cite[Lemma 3.6]{MR2914056}.
\end{proof}

\begin{thm}[Izumi]
For every self-dual tail of $\Gamma_+$, there is self-dual tail of $\Gamma_-$.
$$
\begin{tikzpicture}[baseline=-.1cm,scale=.6]
	\filldraw (0,0) circle (.1cm);
	\draw[very thick, red] (0,0)--(0,.2);
	\filldraw (1,0) circle (.1cm);
	\filldraw (2,0) circle (.1cm);
	\draw[very thick, red] (2,0)--(2,.2);
	\filldraw (4,0) circle (.1cm);
	\draw[very thick, red] (4,0)--(4,.2);
	\filldraw (5,0) circle (.1cm);
	\filldraw (6,0) circle (.1cm);
	\draw[very thick, red] (6,0)--(6,.2);
	\draw[] (0,0)--(6,0);
	\draw[dotted]  (3,0)--(4,-.25);
	\draw[dotted]  (3,0)--(4,-.5);
	\filldraw[unshaded] (3,0) circle (.1cm);
\end{tikzpicture}
\hspace{.5cm}
\longleftrightarrow
\hspace{.5cm}
\begin{tikzpicture}[baseline=-.1cm,scale=.6]
	\filldraw (0,0) circle (.1cm);
	\draw[very thick, red] (0,0)--(0,.2);
	\filldraw (1,0) circle (.1cm);
	\filldraw (2,0) circle (.1cm);
	\draw[very thick, red] (2,0)--(2,.2);
	\filldraw (4,0) circle (.1cm);
	\draw[very thick, red] (4,0)--(4,.2);
	\filldraw (5,0) circle (.1cm);
	\filldraw (6,0) circle (.1cm);
	\draw[very thick, red] (6,0)--(6,.2);
	\draw[] (0,0)--(6,0);
	\draw[dotted]  (3,0)--(4,-.25);
	\draw[dotted]  (3,0)--(4,-.5);
	\filldraw[unshaded] (3,0) circle (.1cm);
\end{tikzpicture}
$$
For every Haagerup tail of $\Gamma_+$, there is a dual Haagerup tail of $\Gamma_-$.
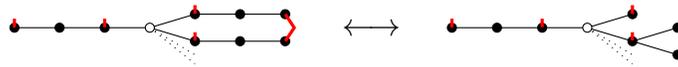
$$
\begin{tikzpicture}[baseline=-.1cm,scale=.6]
	\filldraw (0,0) circle (.1cm);
	\draw[very thick, red] (0,0)--(0,.2);
	\filldraw (1,0) circle (.1cm);
	\filldraw (2,0) circle (.1cm);
	\draw[very thick, red] (2,0)--(2,.2);
	\filldraw (4,-.3) circle (.1cm);
	\filldraw (4,.3) circle (.1cm);
	\draw[very thick, red] (4,.3)--(4,.5);
	\draw[very thick, red] (4,-.3)--(4,-.1);
	\filldraw (5,.3) circle (.1cm);
	\filldraw (5,-.3) circle (.1cm);
	\filldraw (6,.3) circle (.1cm);
	\filldraw (6,-.3) circle (.1cm);
	\draw[very thick, red] (6,-.3)--(6.2,0)--(6,.3);
	\draw[] (0,0)--(3,0);
	\draw (6,.3)--(4,.3)--(3,0)--(4,-.3)--(6,-.3);
	\draw[dotted]  (3,0)--(4,-.6);
	\draw[dotted]  (3,0)--(4,-.8);
	\filldraw[unshaded] (3,0) circle (.1cm);
\end{tikzpicture}
\hspace{.5cm}
\longleftrightarrow
\hspace{.5cm}
\begin{tikzpicture}[baseline=-.1cm,scale=.6]
	\filldraw (0,0) circle (.1cm);
	\draw[very thick, red] (0,0)--(0,.2);
	\filldraw (1,0) circle (.1cm);
	\filldraw (2,0) circle (.1cm);
	\draw[very thick, red] (2,0)--(2,.2);
	\filldraw (4,-.3) circle (.1cm);
	\filldraw (4,.3) circle (.1cm);
	\draw[very thick, red] (4,.3)--(4,.5);
	\draw[very thick, red] (4,-.3)--(4,-.1);
	\filldraw (5,0) circle (.1cm);
	\filldraw (5,-.6) circle (.1cm);
	\draw[] (0,0)--(3,0);
	\draw (4,-.3)--(3,0)--(4,.3);
	\draw (5,0)--(4,-.3)--(5,-.6);
	\draw[dotted]  (3,0)--(4,-.6);
	\draw[dotted]  (3,0)--(4,-.8);
	\filldraw[unshaded] (3,0) circle (.1cm);
\end{tikzpicture}
$$
\end{thm}
\begin{proof}
We know the vertices of $\Gamma_-$ at depth 3 and 5, so it remains to determine the vertices at depth 4 and 6 and the edges.
First, we consider the vertices $\overline{X}g$ at depth 5.
By Frobenius reciprocity,
\begin{align*}
\langle \overline{X} g X, \overline{X} g X\rangle 
= \langle X \overline{X} g, g  X\overline{X}\rangle 
= \langle (1+\rho)g,g(1+\rho)\rangle 
= 1+\langle \rho g, g\rho\rangle
=
\begin{cases}
2 & \text{if }g=g^{-1}\\
1 & \text{if }g\neq g^{-1}.
\end{cases}
\end{align*}
Hence $\overline{X}gX$ is simple precisely when $g\neq g^{-1}$.

\begin{itemize}
\item\underline{Case 1:}
Suppose $g\neq g^{-1}$. Then $\overline{X}gX$ and $\overline{X}g^{-1}X$ are simple, and moreover,
$$
\langle \overline{X} g X, \overline{X} g^{-1} X\rangle 
= \langle X \overline{X} g, g^{-1}  X\overline{X}\rangle 
= \langle (1+\rho)g,g^{-1}(1+\rho)\rangle 
=1,
$$
so they are equal. 
Hence the distinct depth 5 vertices $\overline{X}g^{-1}$ and $\overline{X}g$ of $\Gamma_-$ are univalent and connect to a single self-dual vertex $\overline{X}gX = \overline{X}g^{-1}X$ at depth 4.
Since 
\begin{equation}\label{eq:ThreeSimples}
\overline{X}gX \overline{X} 
= \overline{X}g(1+\rho) 
= \overline{X}g+\overline{X}g\rho 
= \overline{X}g + \overline{X} \rho g^{-1} 
= \overline{X}g +\overline{X}g^{-1} + \overline{Z}g^{-1},
\end{equation}
each of which is simple, $\overline{X}gX$ connects by a single edge to the branch point at depth 3 of $\Gamma_-$.

\item\underline{Case 2:}
Suppose $g=g^{-1}$. 
Then $\overline{X}g$ connects to two even vertices of $\Gamma_-$, at least one of which must connect to the branch point at depth 3.
Since $\overline{X}gX\overline{X}$ splits into exactly three simples by Equation \eqref{eq:ThreeSimples}, $\overline{X}g$ must connect to one bivalent vertex at depth 4 and one univalent vertex at depth 6.
\end{itemize}

Now in both cases, $\overline{X}gX$ is self-dual, so the new vertices we have found so far at depths 4 and 6 are all self-dual, as the dual of a vertex must occur at the same depth.
It remains to show there is a single self-dual vertex connected to $\overline{Z}$ for each (unordered) set of elements $\{g,g^{-1}\}$ with $g\neq g^{-1}$.

Analyzing the Ocneanu 4-partite graph, we see there are $|G|$ paths from $Z$ to $\overline{Z}$ through $A-A$ bimodules, so there must be $|G|$ paths through $B-B$ bimodules.
If $G=N\cup S\cup \{1\}$ where $N$ is the set of non self-inverse elements of $G$ and $S$ is the set of non-trivial self-inverse elements of $G$, then currently we can account for $|S|+1+|N|/2$ paths through $B-B$ vertices.

Hence the remaining paths through $B-B$ bimodules must come from vertices at depth 4 which do not continue to depth 5.
By Lemma \ref{lem:SimplyLacedTruncation}, we know $\Gamma_-(4)$ is simply laced with only self-dual vertices, so there must be exactly $|N|/2$ self-dual univalent vertices at depth 4 of $\Gamma_-$.
\end{proof}

For $g\in G\setminus \{1\}$, let $p_g$ be the projection in $\cP_{4,+}$ corresponding to $g\rho$. 
Let $\langle E_i\rangle$ denote the algebra generated by $E_1,\dots, E_{n-1}$ in $TL_{n,+}$, and note that $\langle E_i\rangle$ is perpendicular to all projections on the principal graph and to the Jones-Wenzl idempotents, and $\spann(\set{p_g}{g\neq 1})=\cP_{4,+}\ominus \langle E_i\rangle$.

\begin{lem}
If $\cQ_\bullet$ is an $n-1$ supertransitive subfactor planar algebra, then any non-zero element in $\cQ_{n,+}\ominus \langle E_i\rangle$ with zero trace is uncappable.
\end{lem}
\begin{proof}
Follows easily from $q E_i=0$ for all $i<n$ and all minimal projections $q\in \cQ_{n,+}\ominus \langle E_i\rangle$.
\end{proof}

\begin{cor}
For all $g,h\in G\setminus\{1\}$ with $g\neq h$, $p_g-p_h$ is uncappable.
\end{cor}

\begin{prop}
The new low-weight vectors at depth 4 have eigenvalue $\pm 1$.
\end{prop}
\begin{proof}
By Corollary \ref{cor:SelfDual}, $\cF^{4}$ is the identity on $\spann(\set{p_g}{g\in G}\})$.
\end{proof}

%%%%%%%%%%%%%%%%%%%%%%%%%%%%%%%%%%%%%%%%%%%%%%%%%%%%%%%%%%%%%%%%
\subsection{The low weight rotational eigenvector conjecture}

\begin{conj*}[Conjecture \ref{conj:Main}]
Suppose $g,h,k,\ell \in G\setminus\{1\}$ are distinct elements.
\begin{enumerate}[(1)]
\item
If $g=g^{-1}$ and $h=h^{-1}$, then $p_g-p_h$ is a low-weight rotational eigenvector with eigenvalue 1.
\item
If $g=h^{-1}$, then $p_g-p_h$ is a low-weight rotational eigenvector with eigenvalue -1.
\item
If $g=g^{-1}$ and $h=k^{-1}$, then $2p_g-(p_h+p_k)$ is a low-weight rotational eigenvector with eigenvalue 1.
\item
If $g=h^{-1}$ and $k=\ell^{-1}$, then $(p_g+p_h)-(p_k+p_\ell)$ is a low-weight rotational eigenvector with eigenvalue 1.
\end{enumerate}
\end{conj*}

\begin{remark}
Conjecture \ref{conj:Main} agrees with  the Haagerup $3^{\bbZ/3\bbZ}$ and Izumi $3^{\bbZ/2\bbZ\times \bbZ/2\bbZ}$ and $3^{\bbZ/4\bbZ}$ subfactor planar algebras by \cite{MR2679382,MR2972458,MR3314808,1308.5197}.
See also Remark \ref{rem:Conjecture}.
\end{remark}

\begin{lem}
The set $B=\set{p_g}{g\in G}\cup \{\cF^{2}(f^{(4)})\})$ is linearly independent. 
\end{lem}
\begin{proof}
First, note that $f^{(4)}=\sum_{g\in G\setminus\{1\}} p_g$ and $\cF^{2}(f^{(4)})$ are linearly independent, since capping $\cF^{2}(f^{(4)})$ on the bottom does not give zero. Suppose
$$
0=\lambda_f \cF^{2}(f^{(4)})+\sum_{g\in G\setminus\{1\}} \lambda_g p_g.
$$
Taking inner products with $p_g$ gives 
$$
0=
\lambda_g \Tr(p_g)+ \lambda_f \Tr(\cF^{2}(f^{(4)}p_g)
=
\lambda_g \Tr(p_g)+ \lambda_f
\Tr(p)\left( \coeff_{\in f^{(4)}} \cF^{2}(\id)\right),
$$
so $\lambda_g$ is independent of $g$. Call this constant $\lambda$. Then we have
$$
0=\lambda_f \cF^{2}(f^{(4)})+\lambda \sum_{g\in G\setminus\{1\}} p_g
=\lambda_f \cF^{2}(f^{(4)})+\lambda f^{(4)},
$$
so $\lambda_f=\lambda = 0$, and $B$ is linearly independent.
\end{proof}

\begin{prop}\label{prop:2ClickRotation}
Conjecture \ref{conj:Main} holds if and only if for all $g\in G$,
$$
\cF^{2}(p_g)=\frac{1}{|G|-1}\left(\cF^{2}(f^{(4)})-\jw{4}\right)+   p_{g^{-1}}.
$$
\end{prop}
\begin{proof}
If $\cF^2$ is given by the above formula, a straightforward calculation shows that Conjecture \ref{conj:Main} holds.
We now prove the other direction.

Divide $G\setminus\{1\}$ into the two subsets: the non-trivial self-inverse elements $S$, and the non-self-inverse elements $N$.
Let $N^+\subset N$ so that for each $g\in N$, exactly one of $g,g^{-1}\in N^+$.
Let $B_{1} = \set{p_g - p_{g^{-1}}}{g\in N^+}$, and note $|B_1|=|N|/2$.

\itt{Case 1}
Suppose $S\neq \emptyset$, so $|G|$ is even. 
Fix $s_0\in S$.
Let $B_{2}=\set{2p_{s_0} - (p_g+p_{g^{-1}})}{g\in N^+}$, and note $|B_{2}|=|N|/2$.
Finally, let $B_3=\set{p_{s_0}-p_{s}}{s\in S\setminus \{s_0\}}$, and note $|B_3|=|S|-2$.
Observe $B'=B_1\cup B_2\cup B_3$ has size $|G|-2$.

\begin{claim*}
$D=B'\cup \{\jw{4},\cF^{2}(\jw{4}\}$ is a basis for $\spann(B)$.
\end{claim*}
\begin{proof}[Proof of Claim]
It suffices to show $D$ is linearly independent.
Note that by taking linear combinations, we can obtain $p_g-p_h$ for all $g,h\in G\setminus\{1\}$.
The result now follows since $\jw{4}=\sum_{g\neq 1} p_g$.
\end{proof}

Now by Conjecture \ref{conj:Main},
$$
[\cF^{2}]_{D} = 
\begin{pmatrix}
-I_{B_1} & 0 &0& 0 & 0\\
0 & I_{B_2} & 0 & 0 & 0\\
0 & 0 & I_{B_3}& 0 & 0\\
0 & 0 & 0 & 0 & 1 \\
0 & 0 & 0 & 1 & 0
\end{pmatrix}\,,
$$
and the change of basis matrix from $D$ to $B$ is given by
$$
Q_D^B=
\left(
\begin{array}{c|ccccccccc}
& B_1=N^+-N^- & B_2=2s_0-N^+ & B_3=s_0-S\setminus\{s_0\} & \jw{4} & \cF^{2}(\jw{4})
\\\hline
N^+ & I & -I & 0 & {\bf 1} & 0\\
N^-  & -I & -I & 0 & {\bf 1} & 0\\
S\setminus\{s_0\} & 0 & 0 & -I & {\bf 1} &  0\\
s_0 & 0 & {\bf 2} & {\bf 1} &  1 & 0 \\
\cF^{2}(\jw{4}) & 0 & 0 & 0 & 0 & 1
\end{array}
\right)
$$
where $I$ is an identity matrix, and ${\mathbf{k}}$ is the matrix with all $k$'s.

Setting $n=|N|+|S|=|G|-1$, it is straightforward to check
$$
(Q_{D}^{B})^{-1}=
\frac{1}{2n}
\left(
\begin{array}{c|ccccccccc}
& N^+ & N^- & S\setminus\{s_0\} & s_0 & \cF^{2}(\jw{4})
\\\hline
B_1  & nI & -nI & 0 & 0& 0\\
B_2 & {\bf 2}-nI & {\bf 2} -nI & {\bf 2} & {\bf 2} & 0\\
B_3 & {\bf 2} & {\bf 2} & {\bf 2}-2nI & {\bf 2} &  0\\
\jw{4} & {\bf 2} & {\bf 2} & {\bf 2} & 2 & 0 \\
\cF^{2}(\jw{4}) & 0 & 0 & 0 & 0 & 2n
\end{array}
\right)
$$
Hence 
$$
[\cF^{2}]_B=Q_{D}^{B}[\cF^{2}]_{D}(Q_{D}^{B})^{-1}
=
\frac{1}{n}
\left(
\begin{array}{c|ccccccc}
& N^+ & N^- & S  & \cF^{2}(\jw{4})
\\\hline
N^+  & -{\bf 1} & nI-{\bf 1} & -{\bf 1} & {\bf n}\\
N^- & nI-{\bf 1} & -{\bf 1} & -{\bf 1} & {\bf n}\\
S & {\bf -1} & {\bf -1} & nI-{\bf 1} & {\bf n}\\
\cF^{2}(\jw{4}) & {\bf 1} & {\bf 1} & {\bf 1} & 0\\
\end{array}
\right).
$$ 

\itt{Case 2}
Suppose $S=\emptyset$, so $|G|$ is odd.
Fix $g_0\in N^+$.
Let 
$$
B_2=\set{(p_g+p_{g^{-1}})-(p_{g_0}+p_{g_0^{-1}})}{g\in N^+\setminus\{g\}},
$$ 
and note $|B_2|=|N^+|-1$.
Hence if $D=B_1\cup B_2$, we have $|D|=|G|-2$.
Similar to before, $D\cup \{\jw{4},\cF^{2}(\jw{4}\}$ is a basis for $\spann\{B\}$.
Now by Conjecture \ref{conj:Main},
$$
[\cF^{2}]_{D} = 
\begin{pmatrix}
-I_{B_1} & 0 & 0 & 0\\
0 & I_{B_2}  & 0 & 0\\
0 & 0 & 0 & 1 \\
0 & 0 & 1 & 0
\end{pmatrix}\,,
$$
and the change of basis matrix from $D$ to $B$ is given by
$$
Q_D^B=
\left(
\begin{array}{c|ccccccccc}
& B_1=N^+\setminus\{g_0\}-N^-\setminus\{g_0^{-1}\} & g_0-g_0^{-1} & B_2 & \jw{4} & \cF^{2}(\jw{4})
\\\hline
N^+\setminus\{g_0\} & I & 0 & I & {\bf 1} & 0\\
g_0 & 0 &  1 & -{\bf 1} & 1 & 0 \\
N^-\setminus\{g_0^{-1}\}  & -I & 0 & I & {\bf 1} & 0\\
g_0^{-1} & 0 & -1 & -{\bf 1}  & 1 & 0 \\
\cF^{2}(\jw{4}) & 0 & 0 & 0 & 0 & 1
\end{array}
\right)
$$
Setting $n=|G|-1$, we have
$$
(Q_D^B)^{-1}=
\frac{1}{2n}
\left(
\begin{array}{c|ccccccccc}
& N^+\setminus\{g_0\}  & g_0 & N^-\setminus\{g_0^{-1}\}  & g_0^{-1} & \cF^{2}(\jw{4})
\\\hline
B_1 & nI & 0 & -nI & 0 & 0\\
g_0-g_0^{-1} & 0 & n & 0 & -n & 0\\
B_2 & nI-{\bf 2} & -{\bf 2} & nI-{\bf 2} & -{\bf 2} & 0\\
\jw{4} & {\bf 2} & 2 & {\bf 2} & 2 & 0\\
\cF^{2}(\jw{4}) & 0 & 0 & 0 & 0 & 2n
\end{array}
\right)
$$
Hence 
\begin{align*}
[\cF^{2}]_B=Q_{D}^{B}[\cF^{2}]_{D}(Q_{D}^{B})^{-1}
&=
\frac{1}{n}
\left(
\begin{array}{c|ccccccc}
& N^+ & N^-  & \cF^{2}(\jw{4})
\\\hline
N^+  & -{\bf 1} & nI-{\bf 1}  & {\bf n}\\
N^- & nI-{\bf 1} & -{\bf 1}  & {\bf n}\\
\cF^{2}(\jw{4}) & {\bf 1} & {\bf 1}  & 0\\
\end{array}
\right).
\qedhere
\end{align*}
\end{proof}

\begin{rem}\label{rem:Conjecture}
Parts (1)-(3) of Conjecture \ref{conj:Main} were chosen because they agree with $3^G$ for $G\in \{\bbZ/2,\bbZ/4,\bbZ/2\times \bbZ/2\}$.
There are two reasons +1 was chosen as the eigenvalue in part (4).
First, if $S\neq \emptyset$ and $g,h\in N^+$ are distinct, then we have
$$
(p_g+p_{g^{-1}})-(p_h+p_h^{-1}) = 2p_{s_0}-(p_h+p_{h^{-1}}) - (2p_{s_0}-(p_g+p_g^{-1})),
$$
which has eigenvalue 1 by part (3).
Second, if $S=\emptyset$, switching the eigenvalue to $-1$ gives the formula
$$
\cF^{2}(p_g) = \frac{1}{n}\cF^{2}(f^{(4)})+  \frac{n-1}{n} p_{g} - \frac{1}{n}\sum_{h\neq g}p_h,
$$
which is not the same formula as when $S\neq \emptyset$.
\end{rem}

%%%%%%%%%%%%%%%%%%%%%%%%%%%%%%%%%%%%%%%%%%%%%%%%%%%%%%%%%%%%%%%%%%
%%%%%%%%%%%%%%%%%%%%%%%%%%%%%%%%%%%%%%%%%%%%%%%%%%%%%%%%%%%%%%%%%%
%%%%%%%%%%%%%%%%%%%%%%%%%%%%%%%%%%%%%%%%%%%%%%%%%%%%%%%%%%%%%%%%%%
\section{The reduced subfactor at $\rho=\jw{2}$}

The even half $\cE_\bullet$ of $\cP_\bullet$ is a factor planar algebra \cite{1208.5505} with principal graph
$$
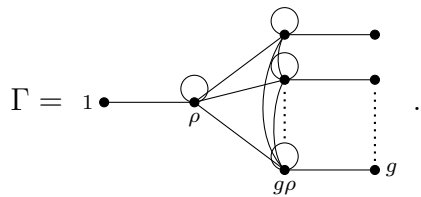
\Gamma
=
\begin{tikzpicture}[baseline=-.1cm,scale=.6]
	\filldraw (0,0) circle (.1cm) node [left] {{\scriptsize{$1$}}};
	\filldraw (2,0) circle (.1cm) node [below] {{\scriptsize{$\rho$}}};
	\draw (2,.3) circle (.3cm);
	\filldraw (4,1.5) circle (.1cm);
	\draw (4,1.8) circle (.3cm);
	\filldraw (4,.5) circle (.1cm);
	\draw (4,.8) circle (.3cm);
	\filldraw (4,-1.5) circle (.1cm) node [below] {{\scriptsize{$g\rho$}}};
	\draw (4,-1.2) circle (.3cm);
	\filldraw (6,1.5) circle (.1cm);
	\filldraw (6,.5) circle (.1cm);
	\filldraw (6,-1.5) circle (.1cm) node [right] {{\scriptsize{$g$}}};
	\draw (0,0) -- (2,0);
	\draw (2,0) -- (4,1.5) -- (6,1.5);
	\draw (2,0) -- (4,.5) -- (6,.5);
	\draw (2,0) -- (4,-1.5) -- (6,-1.5);
	\draw (4,1.5) .. controls ++(235:.3cm) and ++(135:.3cm) .. (4,.5);
	\draw (4,1.5) .. controls ++(235:1cm) and ++(135:1cm) .. (4,-1.5);
	\draw (4,.5) .. controls ++(235:.5cm) and ++(135:.5cm) .. (4,-1.5);
	\draw[thick, dotted]  (4,-.8)--(4,.3);
	\draw[thick, dotted]  (6,-1.3)--(6,.3);
\end{tikzpicture}
\,.
$$
We denote $\rho=\jw{2}$ by a red strand, and we write a trivalent vertex for the intertwiner $\jw{2}\otimes \jw{2} \to \jw{2}$ given by
\begin{align}
\begin{tikzpicture}[baseline = -.1cm]
	\filldraw[rho] (0,0) circle (.05cm);
	\draw[thick, rho] (0,0) -- (-.2,-.4);
	\draw[thick, rho] (0,0) -- (0,.4);
	\draw[thick, rho] (0,0) -- (.2,-.4);
	\nbox{}{(0,0)}{.4}{0}{0}{}
\end{tikzpicture}
=
\left(
\frac{[2]}{[3]-1}
\right)^{1/2}
\begin{tikzpicture}[baseline = -.6cm, scale=.6]
	\clip (-1.2,-2.4) -- (-1.2,.8) -- (1.2,.8) -- (1.2,-2.4);
	\draw[shaded] (-.4,-2.5) -- (-.4,-1.2) arc (180:0:.4cm) -- (.4,-2.5) -- (.8,-2.5) -- (.8,-1.2) -- (.2,-.4) -- (.2,.9) -- (-.2,.9) -- (-.2,-.4) -- (-.8,-1.2) -- (-.8,-2.5);
	\nbox{unshaded}{(0,0)}{.4}{0}{0}{2}
	\nbox{unshaded}{(-.6,-1.6)}{.4}{0}{0}{2}
	\nbox{unshaded}{(.6,-1.6)}{.4}{0}{0}{2}
\end{tikzpicture}
\label{eq:V}
\tag{V}
\end{align}
where we just write $2$ for $\jw{2}$.

\begin{rem}\label{rem:GSize}
It is a straightforward calculation that $\D |G|=\frac{[3]^2-1}{[3]}$.
\end{rem}

\begin{prop}\label{prop:RhoSkeinRelations}
We have the following skein relations:
\begin{align}
\begin{tikzpicture}[baseline = -.1cm]
	\filldraw[rho] (0,.2) circle (.05cm);
	\filldraw[rho] (0,-.2) circle (.05cm);
	\draw[thick, rho] (0,0) circle (.2cm);	
	\draw[thick, rho] (0,-.2) -- (0, -.4);
	\draw[thick, rho] (0,.2) -- (0, .4);
	\nbox{}{(0,0)}{.4}{0}{0}{}
\end{tikzpicture}
&=
\begin{tikzpicture}[baseline = -.1cm]	
	\draw[thick, rho] (0,-.4) -- (0, .4);
	\nbox{}{(0,0)}{.4}{0}{0}{}
\end{tikzpicture}
\notag
\\
\begin{tikzpicture}[baseline = -.1cm]
	\draw[thick, rho] (0,0) circle (.3cm);
	\nbox{}{(0,0)}{.4}{0}{0}{}
\end{tikzpicture}
&=
\begin{tikzpicture}[baseline = -.1cm]
	\filldraw[rho] (-.3,0) circle (.05cm);
	\filldraw[rho] (.3,0) circle (.05cm);
	\draw[thick, rho] (0,0) circle (.3cm);
	\draw[thick, rho] (-.3,0)--(.3,0);
	\nbox{}{(0,0)}{.4}{0}{0}{}
\end{tikzpicture}
=
[3]
\notag
\\
\begin{tikzpicture}[baseline = -.1cm]
	\filldraw[rho] (0,-.2) circle (.05cm);
	\draw[thick, rho] (0,0) circle (.2cm);	
	\draw[thick, rho] (0,-.2) -- (0, -.4);
	\nbox{}{(0,0)}{.4}{0}{0}{}
\end{tikzpicture}
&=
0
\notag
\\
\begin{tikzpicture}[baseline = -.1cm]
	\filldraw[rho] (0,0) circle (.05cm);
	\draw[thick, rho] (0,0) arc (0:-180:.15cm) -- (-.3,.4);
	\draw[thick, rho] (0,0) -- (0,.4);
	\draw[thick, rho] (0,0) -- (.2,-.4);
	\nbox{}{(0,0)}{.4}{.1}{0}{}
\end{tikzpicture}
&=
\begin{tikzpicture}[baseline = -.1cm]
	\filldraw[rho] (0,0) circle (.05cm);
	\draw[thick, rho] (0,0) -- (-.2,-.4);
	\draw[thick, rho] (0,0) -- (0,.4);
	\draw[thick, rho] (0,0) -- (.2,-.4);
	\nbox{}{(0,0)}{.4}{0}{0}{}
\end{tikzpicture}^{\,*}
=
\begin{tikzpicture}[baseline = -.1cm]
	\filldraw[rho] (0,0) circle (.05cm);
	\draw[thick, rho] (0,0) -- (-.2,.4);
	\draw[thick, rho] (0,0) -- (0,-.4);
	\draw[thick, rho] (0,0) -- (.2,.4);
	\nbox{}{(0,0)}{.4}{0}{0}{}
\end{tikzpicture}
\notag
\\
\begin{tikzpicture}[baseline = -.1cm]
	\draw[thick, rho] (-.2,-.4) -- ( -.2, .4);	
	\draw[thick, rho] (.2,-.4) -- (.2, .4);
	\nbox{}{(0,0)}{.4}{0}{0}{}
\end{tikzpicture}
-
\begin{tikzpicture}[baseline = -.1cm]
	\draw[thick, rho] (-.2,-.4)arc (180:0:.2cm);	
	\draw[thick, rho] (-.2,.4)arc (-180:0:.2cm);	
	\nbox{}{(0,0)}{.4}{0}{0}{}
\end{tikzpicture}
&=
\left([3]-1\right)
\left(\,
\begin{tikzpicture}[baseline = -.1cm]
	\filldraw[rho] (0,.2) circle (.05cm);
	\filldraw[rho] (0,-.2) circle (.05cm);
	\draw[thick, rho] (-.2,.4) -- (0,.2) --  ( .2, .4);	
	\draw[thick, rho] (-.2,-.4) -- (0,-.2) --  ( .2, -.4);	
	\draw[thick, rho] (0,-.2) -- (0, .2);
	\nbox{}{(0,0)}{.4}{0}{0}{}
\end{tikzpicture}
-
\begin{tikzpicture}[baseline = -.1cm]
	\filldraw[rho] (.2,0) circle (.05cm);
	\filldraw[rho] (-.2,0) circle (.05cm);
	\draw[thick, rho] (-.2,-.4) -- ( -.2, .4);	
	\draw[thick, rho] (.2,-.4) -- (.2, .4);	
	\draw[thick, rho] (-.2,0) -- (.2,0);
	\nbox{}{(0,0)}{.4}{0}{0}{}
\end{tikzpicture}
\,\right)
\label{rel:IH}
\tag{I=H}\\
\begin{tikzpicture}[baseline = -.1cm]
	\draw[thick, rho] (-.2,-.4) -- ( -.2, .4);	
	\draw[thick, rho] (.2,-.4) -- (.2, .4);
	\nbox{}{(0,0)}{.4}{0}{0}{}
\end{tikzpicture}
&=
\frac{1}{[3]}\,
\begin{tikzpicture}[baseline = -.1cm]
	\draw[thick, rho] (-.2,-.4)arc (180:0:.2cm);	
	\draw[thick, rho] (-.2,.4)arc (-180:0:.2cm);	
	\nbox{}{(0,0)}{.4}{0}{0}{}
\end{tikzpicture}\,
+
\begin{tikzpicture}[baseline = -.1cm]
	\filldraw[rho] (0,.2) circle (.05cm);
	\filldraw[rho] (0,-.2) circle (.05cm);
	\draw[thick, rho] (-.2,.4) -- (0,.2) --  ( .2, .4);	
	\draw[thick, rho] (-.2,-.4) -- (0,-.2) --  ( .2, -.4);	
	\draw[thick, rho] (0,-.2) -- (0, .2);
	\nbox{}{(0,0)}{.4}{0}{0}{}
\end{tikzpicture}\,
+
\sum_{g\neq 1}
\begin{tikzpicture}[baseline = -.1cm]
	\draw[thick, rho] (-.1,.5) -- ( -.1, -.5);	
	\draw[thick, rho] (.1,.5) -- ( .1, -.5);	
	\nbox{unshaded}{(0,0)}{.25}{0}{0}{$p_g$}
\end{tikzpicture}\,
\label{rel:E}
\tag{E}
\end{align}
\end{prop}
\begin{proof}
The proof is similar to \cite[Proposition 3.1]{1308.5723}.
We prove Relations \eqref{rel:IH} and \eqref{rel:E}.
To prove Relation \eqref{rel:IH}, note $\dim(\Hom(\rho^{\otimes 2}, \rho^{\otimes 2}))=3$, so there is some linear relation amongst the four diagrams which appear in the relation.
By rotational symmetry, we must have a relation of the form
$$
\begin{tikzpicture}[baseline = -.1cm]
	\draw[thick, rho] (-.2,-.4) -- ( -.2, .4);	
	\draw[thick, rho] (.2,-.4) -- (.2, .4);
	\nbox{}{(0,0)}{.4}{0}{0}{}
\end{tikzpicture}
\pm
\begin{tikzpicture}[baseline = -.1cm]
	\draw[thick, rho] (-.2,-.4)arc (180:0:.2cm);	
	\draw[thick, rho] (-.2,.4)arc (-180:0:.2cm);	
	\nbox{}{(0,0)}{.4}{0}{0}{}
\end{tikzpicture}
=
\lambda
\left(\,
\begin{tikzpicture}[baseline = -.1cm]
	\filldraw[rho] (0,.2) circle (.05cm);
	\filldraw[rho] (0,-.2) circle (.05cm);
	\draw[thick, rho] (-.2,.4) -- (0,.2) --  ( .2, .4);	
	\draw[thick, rho] (-.2,-.4) -- (0,-.2) --  ( .2, -.4);	
	\draw[thick, rho] (0,-.2) -- (0, .2);
	\nbox{}{(0,0)}{.4}{0}{0}{}
\end{tikzpicture}
\pm
\begin{tikzpicture}[baseline = -.1cm]
	\filldraw[rho] (.2,0) circle (.05cm);
	\filldraw[rho] (-.2,0) circle (.05cm);
	\draw[thick, rho] (-.2,-.4) -- ( -.2, .4);	
	\draw[thick, rho] (.2,-.4) -- (.2, .4);	
	\draw[thick, rho] (-.2,0) -- (.2,0);
	\nbox{}{(0,0)}{.4}{0}{0}{}
\end{tikzpicture}
\,\right).
$$
We can determine $\lambda$ by capping off the right hand side of all the diagrams. 
Finally, to determine that the sign is a minus sign, we note that the Temperley-Lieb diagram 
$$
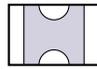
\begin{tikzpicture}[baseline = -.1cm]
	\fill[shaded] (-.4,-.4) -- (-.4,.4) -- (.4,.4) -- (.4,-.4);
	\draw[] (-.4,-.4) -- ( -.4, .4);	
	\draw[] (.4,-.4) -- ( .4, .4);	
	\filldraw[unshaded] (-.2,.4) arc (-180:0:.2cm);
	\filldraw[unshaded] (-.2,-.4) arc (180:0:.2cm);
	\nbox{}{(0,0)}{.4}{.2}{.2}{}
\end{tikzpicture}
$$
has the same non-zero coefficient in both $\rho$-diagrams on the right when we expand the $\rho$-vertex and $\rho$-strands, but it does not appear on the left.

To prove Equation \eqref{rel:E}, we note that since $[2]>2$, $\jw{2}\otimes \jw{2}\cong \jw{0}\oplus \jw{2}\oplus \jw{4}$.
Since $\sum_{g\neq 1} p_g = \jw{4}$, we are finished.
\end{proof}

\begin{cor}\label{cor:RotateH}
Since $\D |G|-1=\frac{[3]^2-[3]-1}{[3]}$, similar to Proposition \ref{prop:2ClickRotation}, we have
$$
\begin{tikzpicture}[baseline = -.1cm]
	\filldraw[rho] (.2,0) circle (.05cm);
	\filldraw[rho] (-.2,0) circle (.05cm);
	\draw[thick, rho] (-.2,-.4) -- ( -.2, .4);	
	\draw[thick, rho] (.2,-.4) -- (.2, .4);	
	\draw[thick, rho] (-.2,0) -- (.2,0);
	\nbox{}{(0,0)}{.4}{0}{0}{}
\end{tikzpicture}
=
\frac{1}{|G|-1}\left( \cF^2(\jw{4})-\jw{4}\right) + 
\begin{tikzpicture}[baseline = -.1cm]
	\filldraw[rho] (0,.2) circle (.05cm);
	\filldraw[rho] (0,-.2) circle (.05cm);
	\draw[thick, rho] (-.2,.4) -- (0,.2) --  ( .2, .4);	
	\draw[thick, rho] (-.2,-.4) -- (0,-.2) --  ( .2, -.4);	
	\draw[thick, rho] (0,-.2) -- (0, .2);
	\nbox{}{(0,0)}{.4}{0}{0}{}
\end{tikzpicture}\,.
$$
\end{cor}

\begin{defn}
We define 
$
\begin{tikzpicture}[baseline = -.1cm]
	\draw[thick, rho] (-.15,.6) --  ( -.15, -.6);	
	\draw[thick, rho] (.15,.6) --  ( .15, -.6);	
	\nbox{unshaded}{(0,0)}{.3}{0}{0}{$p_1$}
\end{tikzpicture}
:=
\begin{tikzpicture}[baseline = -.1cm]
	\draw[thick, rho] (-.2,.6) --  ( -.2, -.6);	
	\draw[thick, rho] (.2,.6) --  ( .2, -.6);	
	\nbox{unshaded}{(0,0)}{.4}{0}{0}{}
	\filldraw[rho] (0,.2) circle (.05cm);
	\filldraw[rho] (0,-.2) circle (.05cm);
	\draw[thick, rho] (-.2,.4) arc (-180:0:.2);
	\draw[thick, rho] (-.2,-.4) arc (180:0:.2);	
	\draw[thick, rho] (0,-.2) -- (0, .2);
\end{tikzpicture}
$\,.
\end{defn}

\begin{defn}\label{defn:ReducedPlanarAlgebra}
Let $\cS_\bullet$ be the reduced subfactor planar algebra of $\cP_\bullet$ at $\rho=\jw{2}$.
Since $\rho$ is symmetrically self-dual in the sense of \cite{MR3314808,MPAffineAandD}, we may lift the shading on $\cS_\bullet$ to get a factor planar algebra $\cR_\bullet$.
We use the convention that $\cR_n = \Hom( \rho^{\otimes n},\rho^{\otimes n})$, which is usually denoted $\cR_{2n}$ in \cite{1208.5505}.

By Relation \eqref{rel:E}, the unitary fusion category associated to $\cR_\bullet$ is $\frac{1}{2}\cP_+\boxtimes \Vec(\bbZ/2\bbZ)$, and the principal graph $\Lambda$ of $\cR_\bullet$ is given by
$$
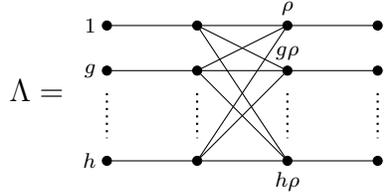
\Lambda = 
\begin{tikzpicture}[baseline=-1cm,scale=.6]
	\filldraw (0,0) circle (.1cm) node [left] {{\scriptsize{$1$}}};
	\filldraw (0,-1) circle (.1cm) node [left] {{\scriptsize{$g$}}};
	\filldraw (0,-3) circle (.1cm) node [left] {{\scriptsize{$h$}}};
	\filldraw (2,0) circle (.1cm);
	\filldraw (2,-1) circle (.1cm);
	\filldraw (2,-3) circle (.1cm);
	\filldraw (4,0) circle (.1cm) node [above] {{\scriptsize{$\rho$}}};
	\filldraw (4,-1) circle (.1cm) node [above] {{\scriptsize{$g\rho$}}};
	\filldraw (4,-3) circle (.1cm) node [below] {{\scriptsize{$h\rho$}}};
	\filldraw (6,0) circle (.1cm);
	\filldraw (6,-1) circle (.1cm);
	\filldraw (6,-3) circle (.1cm);
	\draw (0,0) -- (6,0);
	\draw (0,-1) -- (6,-1);
	\draw (0,-3) -- (6,-3);
	\draw (2,0) -- (4,-1);
	\draw (2,0) -- (4,-3);
	\draw (2,-1) -- (4,0);
	\draw (2,-1) -- (4,-3);
	\draw (2,-3) -- (4,0);
	\draw (2,-3) -- (4,-1);
	\draw[thick, dotted]  (0,.-1.5)--(0,-2.5);
	\draw[thick, dotted]  (2,-1.5)--(2,-2.5);
	\draw[thick, dotted]  (4,-1.5)--(4,-2.5);
	\draw[thick, dotted]  (6,-1.5)--(6,-2.5);
\end{tikzpicture}\,.
$$
The principal graphs of $\cS_\bullet$ are $(\Lambda,\Lambda)$, and $\frac{1}{2}\cS_+=\frac{1}{2}\cP_+$.
\end{defn}

\begin{rems}
\mbox{}
\be
\item
We may also obtain $\cS_\bullet$ by imposing a shading on $\cE_\bullet$
\item
We may naturally identify $\cR_n\cong \cS_{n,\pm}$ as a subspace of $\cP_{2n,+}$.
\ee
\end{rems}

\begin{prop}\label{prop:RectangularLowWeight}
Suppose that $A\in \cR_{2}$ is rectangularly uncappable in $\cP_{4,+}$, i.e., capping $A$ on the top or bottom is zero when we write $A$ with 4 strands up and 4 strands down.
Then we may identify the two click rotation $\cF_{\cP_\bullet}^{2}(A)\in \cP_{4,+}$ with the 1-click rotation $\cF_{\cR_\bullet}(A)\in \cR_{2}$.
\end{prop}
\begin{proof}
$
\cF_{\cP_\bullet}^{2}(A) = 
\begin{tikzpicture}[baseline = -.1cm]
	\draw (-.3,.3) -- (-.3,1.1);
	\draw (-.1,.3) -- (-.1,1.1);
	\draw (.3,-.3) -- (.3,-1.1);
	\draw (.1,-.3) -- (.1,-1.1);
	\draw (-.3,-.3) arc (0:-180:.2cm) -- (-.7,1.1);
	\draw (-.1,-.3) arc (0:-180:.4cm) -- (-.9,1.1);
	\draw (.3,.3) arc (180:0:.2cm) -- (.7,-1.1);
	\draw (.1,.3) arc (180:0:.4cm) -- (.9,-1.1);
	\nbox{unshaded}{(0,0)}{.3}{.1}{.1}{$A$}
\end{tikzpicture}
=
\begin{tikzpicture}[baseline = -.1cm]
	\draw (-.3,.3) -- (-.3,1.1);
	\draw (-.1,.3) -- (-.1,1.1);
	\draw (.3,-.3) -- (.3,-1.1);
	\draw (.1,-.3) -- (.1,-1.1);
	\draw (-.3,-.3) arc (0:-180:.2cm) -- (-.7,1.1);
	\draw (-.1,-.3) arc (0:-180:.4cm) -- (-.9,1.1);
	\draw (.3,.3) arc (180:0:.2cm) -- (.7,-1.1);
	\draw (.1,.3) arc (180:0:.4cm) -- (.9,-1.1);
	\nbox{unshaded}{(0,0)}{.3}{.1}{.1}{$A$}
	\nbox{unshaded}{(-.8,0)}{.2}{0}{0}{$2$}
	\nbox{unshaded}{(.8,0)}{.2}{0}{0}{$2$}
	\nbox{unshaded}{(-.2,.7)}{.2}{0}{0}{$2$}
	\nbox{unshaded}{(.2,-.7)}{.2}{0}{0}{$2$}
\end{tikzpicture}
=
\begin{tikzpicture}[baseline = -.1cm]
	\draw[thick, rho] (-.2,.3) -- (-.2,1.1);
	\draw[thick, rho] (.2,-.3) -- (.2,-1.1);
	\draw[thick, rho] (-.2,-.3) arc (0:-180:.3cm) -- (-.8,1.1);
	\draw[thick, rho] (.2,.3) arc (180:0:.3cm) -- (.8,-1.1);
	\nbox{unshaded}{(0,0)}{.3}{.1}{.1}{$A$}
\end{tikzpicture}
=
\cF_{\cR_\bullet}(A).
$
\end{proof}

Consider the orthogonal complement of $\TL_{4,+}\subset \cP_{4,+}$, which is spanned by $|G|-2$ low weight vectors $\set{A_j}{j=1,\dots, |G|-2}\subset \spann\set{p_g}{g\neq 1}$, which are also eigenvectors for the 2-click rotation $\cF^{2}$ on $\cP_\bullet$ corresponding to rotational eigenvectors $\omega_{A_j}$.
By Proposition \ref{prop:RectangularLowWeight}, we get the following corollary:

\begin{cor}
Each $A_j$ is also a low-weight rotational eigenvector for the $1$-click rotation in $\cR_\bullet$.
\end{cor}

Since $\sum_{g\neq 1} p_g = \jw{4}$, we know that another low weight rotational eigenvector in $\cR_{2,+}$ comes from $\TL_{4,+}\subset \cP_{4,+}$:
$$
B=(|G|-1) p_1 - \jw{4} = (|G|-1)p_1 - \sum_{g\neq 1} p_g,
$$
where the rotational eigenvalue $\omega_B=1$ by Corollary \ref{cor:RotateH}.
Note that $B$ is orthogonal to each $A_j$.

We now compute the 1-click rotation for all the non-trivial minimal projections in $\cR_{2}$ in terms of the $p_g$ for $g\in G$ and 
$$
\rhoE
=
\frac{1}{[3]}
\begin{tikzpicture}[baseline = -.1cm]
	\nbox{unshaded}{(0,0)}{.4}{0}{0}{}
	\draw[thick, rho] (-.2,.4) arc (-180:0:.2);
	\draw[thick, rho] (-.2,-.4) arc (180:0:.2);	
\end{tikzpicture}
\text{ and }
\rhoJW
=
\begin{tikzpicture}[baseline = -.1cm]
	\draw[thick, rho] (-.2,-.4) -- ( -.2, .4);	
	\draw[thick, rho] (.2,-.4) -- (.2, .4);
	\nbox{}{(0,0)}{.4}{0}{0}{}
\end{tikzpicture}
-\rhoE
=
\begin{tikzpicture}[baseline = -.1cm]
	\filldraw[rho] (0,.2) circle (.05cm);
	\filldraw[rho] (0,-.2) circle (.05cm);
	\draw[thick, rho] (-.2,.4) -- (0,.2) --  ( .2, .4);	
	\draw[thick, rho] (-.2,-.4) -- (0,-.2) --  ( .2, -.4);	
	\draw[thick, rho] (0,-.2) -- (0, .2);
	\nbox{}{(0,0)}{.4}{0}{0}{}
\end{tikzpicture}\,
+
\sum_{g\neq 1}
\begin{tikzpicture}[baseline = -.1cm]
	\draw[thick, rho] (-.1,.5) -- ( -.1, -.5);	
	\draw[thick, rho] (.1,.5) -- ( .1, -.5);	
	\nbox{unshaded}{(0,0)}{.25}{0}{0}{$p_g$}
\end{tikzpicture}.
$$
Note that by Relations \eqref{rel:IH} and \eqref{rel:E},
\begin{align*}
\cF^2(\jw{4}) 
&= \frac{1}{[3]([3]-1)}\jw{4} + \frac{[3]^2-[3]-1}{[3]}\rhoE - \frac{[3]^2-[3]-1}{[3]([3]-1)}p_1\\
&= \frac{1}{[3]([3]-1)}\jw{4}+ (|G|-1)\rhoE - \frac{|G|-1}{[3]-1} p_1.
\end{align*}

\begin{cor}\label{cor:EquivalentConjecture}
Conjecture \ref{conj:Main} holds if and only if for each $g\in G$,
\begin{align*}
\cF_{\cR_\bullet}(p_g) 
&= \frac{1}{|G|-1} (\cF^2(\jw{4})-\jw{4})+p_{g^{-1}}\\
&= p_{g^{-1}} - \frac{1}{[3]-1}\jw{4} + \rhoE- \frac{1}{[3]-1}p_1\\
&= p_{g^{-1}} + \rhoE -\frac{1}{[3]-1}\sum_{g\in G} p_g\\
&= p_{g^{-1}} + \rhoE -\frac{1}{[3]-1}\rhoJW\\
&= p_{g^{-1}} - \frac{1}{[3]-1} \left(\,
\begin{tikzpicture}[baseline = -.1cm]
	\nbox{unshaded}{(0,0)}{.4}{0}{0}{}
	\draw[thick, rho] (-.2,.4) arc (-180:0:.2);
	\draw[thick, rho] (-.2,-.4) arc (180:0:.2);	
\end{tikzpicture}
-
\begin{tikzpicture}[baseline = -.1cm]
	\draw[thick, rho] (-.2,-.4) -- ( -.2, .4);	
	\draw[thick, rho] (.2,-.4) -- (.2, .4);
	\nbox{}{(0,0)}{.4}{0}{0}{}
\end{tikzpicture}\,
\right).
\end{align*}
\end{cor}

%%%%%%%%%%%%%%%%%%%%%%%%%%%%%%%%%%%%%%%%%%%%%%%%%%%%%%%%%%%%%%%%%%
%%%%%%%%%%%%%%%%%%%%%%%%%%%%%%%%%%%%%%%%%%%%%%%%%%%%%%%%%%%%%%%%%%
%%%%%%%%%%%%%%%%%%%%%%%%%%%%%%%%%%%%%%%%%%%%%%%%%%%%%%%%%%%%%%%%%%
\section{An `almost' $G$-action on $\cR_\bullet$}\label{sec:ActionOnR}

We now construct an `almost' $G$-action on $\cR_\bullet$.
This corresponds to an action of $G$ on $\frac{1}{2}\cP_+$.
We define for each $g\in G$ a map $\Phi_g$ on the unshaded factor planar algebra $\cR_\bullet$.

\begin{defn}
Recall that the group $G$ can be seen in $\frac{1}{2}\cP_+$ inside $\cP_{6,+}$ as the minimal projections at depth $6$ of $\Gamma_+$, together with the image of the trivial object in $\cP_{6,+}$ given by
$$
e=\frac{1}{[4]}\,
\begin{tikzpicture}[baseline=-.1cm]
	\fill[shaded] (-.2,-.4) arc (180:0:.2cm);
	\fill[shaded] (-.2,.4) arc (-180:0:.2cm);
	\draw[thick, theta] (-.2,-.4) arc (180:0:.2cm);
	\draw[thick, theta] (-.2,.4) arc (-180:0:.2cm);
	\nbox{}{(0,0)}{.4}{0}{0}{}
\end{tikzpicture}
\text{ where }
\begin{tikzpicture}[baseline=-.1cm]
	\fill[shaded] (0,-.4) rectangle (.4,.4);
	\draw[thick, theta] (0,-.4) -- (0,.4);
	\nbox{}{(0,0)}{.4}{0}{0}{}
\end{tikzpicture}
=\jw{3}\in\cP_{3,+}.
$$
Here, we switch the convention of the unit of the group to $e$ instead of $1$ to not confuse the empty diagram $1$ with the projection $e\in \cP_{6,+}$.
(While the empty diagram is the identity of $\cP_{0,+}$, $e$ is not the identity of $\cP_{6,+}$!)
\end{defn}

The group multiplication is given by a multiple of the coproduct.

\begin{lem}\label{lem:GCoproduct}
In $\cP_{6,+}$, the coproduct 
$g * h = 
\begin{tikzpicture}[baseline = -.1cm]
	\fill[shaded] (-.5,-.7) rectangle (.5,.7);
	\fill[white](-.3,-.3) arc (-180:0:.3cm) -- (.3,.3) arc (0:180:.3cm);
	\draw[thick, theta] (-.5,-.7) -- (-.5,.7);
	\draw[thick, theta] (.5,-.7) -- (.5,.7);
	\nbox{unshaded}{(-.4,0)}{.3}{0}{0}{$g$}
	\nbox{unshaded}{(.4,0)}{.3}{0}{0}{$h$}
	\draw[thick, theta] (-.3,.3) arc (180:0:.3cm);
	\draw[thick, theta] (-.3,-.3) arc (-180:0:.3cm);
\end{tikzpicture}
=
[4]^{-1}\,
\begin{tikzpicture}[baseline = -.1cm]
	\fill[shaded] (-.2,-.7) rectangle (.2,.7);
	\draw[thick, theta] (-.2,-.7) -- (-.2,.7);
	\draw[thick, theta] (.2,-.7) -- (.2,.7);
	\nbox{unshaded}{(0,0)}{.3}{0}{0}{$gh$}
\end{tikzpicture}
$.
\end{lem}
\begin{proof}
We know $g\otimes h\cong gh$ and $g*h$ is self adjoint, so $g*h=\lambda gh$ for some $\lambda\in \bbR$.
Taking traces shows that $\lambda=[4]^{-1}$.
\end{proof}

Recall that $\rho g \cong g^{-1}\rho$ for all $g\in G$.
This means there is exactly one non-zero morphism up to scaling between the two.

\begin{defn}
For $g\in G$, we define the following element of $\cP_{6,+}$:
$$
V_g = 
[3]
\begin{tikzpicture}[baseline = -.1cm]
	\draw[thick, theta] (.1,1.1) -- (.1,1.4);
	\draw[thick, theta] (.3,1.1) -- (.3,1.4);
	\draw[thick, theta] (-.3,-1.1) -- (-.3,-1.4);
	\draw[thick, theta] (-.1,-1.1) -- (-.1,-1.4);
	\draw[thick, rho] (-.2,-.3) -- (0,.3);
	\draw[thick, rho] (0,-.3) -- (.2,.3);
	\draw[thick, rho] (-.4,-.3) -- (-.4,1.4);
	\draw[thick, rho] (.4,.3) -- (.4,-1.4);
	\nbox{unshaded}{(-.2,-.7)}{.4}{0}{0}{$g$}
	\nbox{unshaded}{(.2,.7)}{.4}{0}{0}{$g$}
	\node at (-.75,-.7) {$\star$};	
	\node at (.75,.7) {$\star$};	
\end{tikzpicture}
=
[3]
\begin{tikzpicture}[baseline = -.1cm]
	\draw[thick, rho] (0,1.1) -- (0,1.4);
	\draw[thick, rho] (.2,1.1) -- (.2,1.4);
	\draw[thick, rho] (.4,1.1) -- (.4,1.4);
	\draw[thick, rho] (-.4,-1.1) -- (-.4,-1.4);
	\draw[thick, rho] (-.2,-1.1) -- (-.2,-1.4);
	\draw[thick, rho] (0,-1.1) -- (0,-1.4);
	\draw[thick, rho] (-.2,-.3) -- (0,.3);
	\draw[thick, rho] (0,-.3) -- (.2,.3);
	\draw[thick, rho] (-.4,-.3) -- (-.4,1.4);
	\draw[thick, rho] (.4,.3) -- (.4,-1.4);
	\nbox{unshaded}{(-.2,-.7)}{.4}{0}{0}{$g$}
	\nbox{unshaded}{(.2,.7)}{.4}{0}{0}{$g$}
	\node at (-.75,-.7) {$\star$};	
	\node at (.75,.7) {$\star$};	
\end{tikzpicture}
=
\begin{tikzpicture}[baseline = -.1cm]
	\draw[thick, rho] (.4,-.8) -- (-.4,.8);
	\draw[mid>] (-.4,-.8) -- (-.15,-.3);
	\node at (.1,.7) {\scriptsize{$g$}};
	\draw[mid<] (.15,.3) -- (.4,.8);
	\node at (-.5,-.6) {\scriptsize{$g$}};
	\ncircle{unshaded}{(0,0)}{.4}{180}{$V_g$}
\end{tikzpicture}
\,.
$$
We may think of $V_g$ as a crossing as above, where we use an oriented strand labelled by $g$ to denote three $\rho$-strands cabled by $g$, and the direction gives the location of the $\star$.
Here, we use the convention that $V_e$ is given by
$$
V_e = 
\frac{1}{[4]}\,
\begin{tikzpicture}[baseline=-.1cm]
	\fill[shaded] (0,.4) arc (-180:0:.2cm);
	\fill[shaded] (-.4,-.4) arc (180:0:.2cm);
	\draw[thick, theta] (0,.4) arc (-180:0:.2cm);
	\draw[thick, theta] (-.4,-.4) arc (180:0:.2cm);
	\draw[thick, rho] (.4,-.4) -- (-.4,.4);
\end{tikzpicture}\,.
$$
Note that we also have
$$
V_g^* = 
[3]
\begin{tikzpicture}[baseline = -.1cm, yscale=-1]
	\draw[thick, theta] (.1,1.1) -- (.1,1.4);
	\draw[thick, theta] (.3,1.1) -- (.3,1.4);
	\draw[thick, theta] (-.3,-1.1) -- (-.3,-1.4);
	\draw[thick, theta] (-.1,-1.1) -- (-.1,-1.4);
	\draw[thick, rho] (-.2,-.3) -- (0,.3);
	\draw[thick, rho] (0,-.3) -- (.2,.3);
	\draw[thick, rho] (-.4,-.3) -- (-.4,1.4);
	\draw[thick, rho] (.4,.3) -- (.4,-1.4);
	\nbox{unshaded}{(-.2,-.7)}{.4}{0}{0}{$g$}
	\nbox{unshaded}{(.2,.7)}{.4}{0}{0}{$g$}
	\node at (-.75,-.7) {$\star$};	
	\node at (.75,.7) {$\star$};	
\end{tikzpicture}
=
[3]
\begin{tikzpicture}[baseline = -.1cm, yscale=-1]
	\draw[thick, rho] (0,1.1) -- (0,1.4);
	\draw[thick, rho] (.2,1.1) -- (.2,1.4);
	\draw[thick, rho] (.4,1.1) -- (.4,1.4);
	\draw[thick, rho] (-.4,-1.1) -- (-.4,-1.4);
	\draw[thick, rho] (-.2,-1.1) -- (-.2,-1.4);
	\draw[thick, rho] (0,-1.1) -- (0,-1.4);
	\draw[thick, rho] (-.2,-.3) -- (0,.3);
	\draw[thick, rho] (0,-.3) -- (.2,.3);
	\draw[thick, rho] (-.4,-.3) -- (-.4,1.4);
	\draw[thick, rho] (.4,.3) -- (.4,-1.4);
	\nbox{unshaded}{(-.2,-.7)}{.4}{0}{0}{$g$}
	\nbox{unshaded}{(.2,.7)}{.4}{0}{0}{$g$}
	\node at (-.75,-.7) {$\star$};	
	\node at (.75,.7) {$\star$};	
\end{tikzpicture}
=
\begin{tikzpicture}[baseline = -.1cm, yscale=-1]
	\draw[thick, rho] (.4,-.8) -- (-.4,.8);
	\draw[mid<] (-.4,-.8) -- (-.15,-.3);
	\node at (.1,.7) {\scriptsize{$g$}};
	\draw[mid>] (.15,.3) -- (.4,.8);
	\node at (-.5,-.6) {\scriptsize{$g$}};
	\ncircle{unshaded}{(0,0)}{.4}{180}{$V_g^*$}
\end{tikzpicture}\,.
$$

Using $V_g$ and $V_g^*$, we define the map $\Phi_g$ on $x\in \cR_{n}$ by encircling $x$ by a strand whose orientation reverses as it crosses the $\rho$-strands connected to $x$.
The orientation is clockwise in the distinguished region of $x$.
We replace crossings with either $V_g$ or $V_{g}^*$ depending on the crossings.
This means that if we travel on the $g$-strand from an unshaded region to a shaded region, we replace the crossing with $V_g$, and if we cross from shaded to unshaded we replace the crossing with $V_g^*$.
\end{defn}

\begin{rem}
It is easy to see that for $x\in \cR_n$ and $g\in G$, $\Phi_g(x)\in \cR_n$.
When $g=e$, $\Phi_e$ is the identity.
When $g\neq e$, $V_g, V_g^*\in \cR_4$, so $\Phi_g(x)\in\cR_n$.
\end{rem}

\begin{ex}
When $x\in \cR_{2}$, we have
$
\Phi_g(x) = 
\begin{tikzpicture}[baseline = -.1cm]
	\pgfmathsetmacro{\innerRadius}{.9}
	\pgfmathsetmacro{\middleRadius}{1.1}
	\pgfmathsetmacro{\outerRadius}{1.5}
	\draw[mid>] (225:\innerRadius) arc (225:135:\innerRadius);
	\draw[mid<] (135:\innerRadius) arc (135:45:\innerRadius);
	\draw[mid>] (45:\innerRadius) arc (45:-45:\innerRadius);
	\draw[mid<] (315:\innerRadius) arc (315:225:\innerRadius);
	\draw[thick, rho] (0,0) -- (45:\outerRadius);
	\draw[thick, rho] (0,0) -- (135:\outerRadius);
	\draw[thick, rho] (0,0) -- (225:\outerRadius);
	\draw[thick, rho] (0,0) -- (315:\outerRadius);
	\node at (0:\middleRadius) {\scriptsize{$g$}};
	\node at (90:\middleRadius) {\scriptsize{$g$}};
	\node at (180:\middleRadius) {\scriptsize{$g$}};
	\node at (270:\middleRadius) {\scriptsize{$g$}};
	\ncircle{unshaded}{(45:\innerRadius)}{.27}{0}{{\scriptsize{$V_{g}^*$}}}
	\ncircle{unshaded}{(135:\innerRadius)}{.27}{180}{{\scriptsize{$V_g$}}}
	\ncircle{unshaded}{(225:\innerRadius)}{.27}{180}{{\scriptsize{$V_{g}^*$}}}
	\ncircle{unshaded}{(315:\innerRadius)}{.27}{0}{{\scriptsize{$V_g$}}}
	\ncircle{unshaded}{(0,0)}{.27}{180}{$x$}
\end{tikzpicture}
$.
\end{ex}

\begin{lem}\label{lem:OneClickRotationCompatible}
There is a constant $\theta_g\in U(1)$ such that $\cF(V_g^*) = \theta_g V_{g^{-1}}$ and $\cF^{-1}(V_g)=\theta_g^{-1} V_{g^{-1}}^*$.
Moreover, $\theta_g=\theta_{g^{-1}}$ for all $g\in G$, and $\theta_e=1$.
\end{lem}
\begin{proof}
There is exactly one map up to a scalar from $g^{-1}\rho$ to $\rho g$, so there is a constant $\theta_g\neq 0$ such that $\cF(V_g^*)=\theta_g V_{g^{-1}}$.
Since the norm squared of $\cF(V_g^*)$ equals the norm squared of $V_{g^{-1}}$, $\theta_g\in U(1)$.
Now taking adjoints, we have $\cF^{-1}(V_{g})=\theta_g^{-1}V_{g^{-1}}^*$.

We now apply $\cF$ to the equation $\cF^{-1}(V_{g})=\theta_g^{-1}V_{g^{-1}}^*$ to get the equation $\cF(V_{g^{-1}}^*)=\theta_g V_g$. 
This means $\theta_g = \theta_{g^{-1}}$ by the definition of $\theta_{g^{-1}}$.

Finally, a simple diagrammatic calculation shows $\theta_e=1$.
\end{proof}

\begin{cor}
The map $\Phi$ is also given as follows. 
First, encircle $x$ by a strand whose orientation reverses as it crosses the $\rho$-strands connected to $x$.
The orientation is clockwise in the distinguished region of $x$.
We replace crossings with either $V_{g^{-1}}$ or $V_{g^{-1}}^*$ depending on the crossings.

For example, when $x\in \cR_{2}$ we have
$
\Phi_g(x) = 
\begin{tikzpicture}[baseline = -.1cm]
	\pgfmathsetmacro{\innerRadius}{1}
	\pgfmathsetmacro{\middleRadius}{1.2}
	\pgfmathsetmacro{\outerRadius}{1.6}
	\draw[mid>] (225:\innerRadius) arc (225:135:\innerRadius);
	\draw[mid<] (135:\innerRadius) arc (135:45:\innerRadius);
	\draw[mid>] (45:\innerRadius) arc (45:-45:\innerRadius);
	\draw[mid<] (315:\innerRadius) arc (315:225:\innerRadius);
	\draw[thick, rho] (0,0) -- (45:\outerRadius);
	\draw[thick, rho] (0,0) -- (135:\outerRadius);
	\draw[thick, rho] (0,0) -- (225:\outerRadius);
	\draw[thick, rho] (0,0) -- (315:\outerRadius);
	\node at (0:\middleRadius) {\scriptsize{$g$}};
	\node at (90:\middleRadius) {\scriptsize{$g$}};
	\node at (180:\middleRadius) {\scriptsize{$g$}};
	\node at (270:\middleRadius) {\scriptsize{$g$}};
	\ncircle{unshaded}{(45:\innerRadius)}{.4}{90}{{\scriptsize{$V_{g^{-1}}$}}}
	\ncircle{unshaded}{(135:\innerRadius)}{.4}{90}{{\scriptsize{$V_{g^{-1}}^*$}}}
	\ncircle{unshaded}{(225:\innerRadius)}{.4}{270}{{\scriptsize{$V_{g^{-1}}$}}}
	\ncircle{unshaded}{(315:\innerRadius)}{.4}{270}{{\scriptsize{$V_{g^{-1}}^*$}}}
	\ncircle{unshaded}{(0,0)}{.27}{180}{$x$}
\end{tikzpicture}
$.
\end{cor}

\begin{cor}\label{cor:1ClickProblem}
We have $\Phi_g \circ \cF_{\cR_\bullet} = \cF_{\cR_\bullet} \circ \Phi_{g^{-1}}$.
\end{cor}

\begin{rem}\label{rem:AutomorphismSkein}
Since each $g\in G$ has dimension 1, we have the usual skein relations for $g$-strands:
$$
\begin{tikzpicture}[baseline = -.1cm]	
	\nbox{unshaded}{(0,0)}{.4}{0}{0}{}
	\draw[mid>] (-.2,.4) arc (-180:0:.2);
	\draw[mid<] (-.2,-.4) arc (180:0:.2);	
\end{tikzpicture}
=
\begin{tikzpicture}[baseline = -.1cm]	
	\nbox{unshaded}{(0,0)}{.4}{0}{0}{}
	\draw[mid<] (-.2,-.4) -- (-.2,.4);
	\draw[mid>] (.2,-.4) -- (.2,.4);
\end{tikzpicture}
\text{ and }
\begin{tikzpicture}[baseline = -.1cm]	
	\nbox{unshaded}{(0,0)}{.4}{0}{0}{}
	\draw[mid<] (-.2,.4) arc (-180:0:.2);
	\draw[mid>] (-.2,-.4) arc (180:0:.2);	
\end{tikzpicture}
=
\begin{tikzpicture}[baseline = -.1cm]	
	\nbox{unshaded}{(0,0)}{.4}{0}{0}{}
	\draw[mid>] (-.2,-.4) -- (-.2,.4);
	\draw[mid<] (.2,-.4) -- (.2,.4);
\end{tikzpicture}\,.
$$
\end{rem}

\begin{prop}\label{prop:ReidemeisterII}
The $g$-strand and the $\rho$-strand together with $V_g,V_g^*,V_{g^{-1}},V_{g^{-1}}^*$, satisfy the Reidemeister II relations
\begin{enumerate}[(1)]
\item
$
\begin{tikzpicture}[baseline = -.1cm]
	\draw[mid>] (-.4,.2) -- ( .4, .2);	
	\draw[thick, rho] (-.4,-.2) -- (.4, -.2);
	\draw[thick] (-.4,-.4)--(-.4,.4)--(.4,.4)--(.4,-.4)--(-.4,-.4);
\end{tikzpicture}
=
\begin{tikzpicture}[baseline = -.1cm]
	\draw[thick, rho] (-1.4,-.4)-- (-.6,0) .. controls ++(45:.6cm) and ++(135:.6cm) .. (.6,0)--(1.4,-.4);	
	\draw[mid>] (-1.4,0) -- (-.9, 0);
	\draw[mid<] (-.6,0) -- (.6, 0);
	\draw[mid>] (1,0) -- (1.4, 0);
	\draw[thick, unshaded] (.6,0) circle (.4cm);
	\node at (.6,0) {\scriptsize{$V_{g}^*$}};
	\node at (.6,.55) {$\star$};
	\draw[thick, unshaded] (-.6,0) circle (.4cm);
	\node at (-.6,0) {\scriptsize{$V_g$}};
	\node at (-.6,.55) {$\star$};
\end{tikzpicture}
=
\begin{tikzpicture}[baseline = -.1cm]
	\draw[thick, rho] (-1.4,-.4)-- (-.6,0) .. controls ++(45:.6cm) and ++(135:.6cm) .. (.6,0)--(1.4,-.4);	
	\draw[mid>] (-1.4,0) -- (-.9, 0);
	\draw[mid<] (-.6,0) -- (.6, 0);
	\draw[mid>] (1,0) -- (1.4, 0);
	\draw[thick, unshaded] (.6,0) circle (.4cm);
	\node at (.6,0) {\scriptsize{$V_{g}^*$}};
	\node at (.6,-.55) {$\star$};
	\draw[thick, unshaded] (-.6,0) circle (.4cm);
	\node at (-.6,0) {\scriptsize{$V_g$}};
	\node at (-.6,-.55) {$\star$};
\end{tikzpicture}
$, and
\item
$
\begin{tikzpicture}[baseline = -.1cm]
	\draw[mid<] (-.4,.2) -- ( .4, .2);	
	\draw[thick, rho] (-.4,-.2) -- (.4, -.2);
	\draw[thick] (-.4,-.4)--(-.4,.4)--(.4,.4)--(.4,-.4)--(-.4,-.4);
\end{tikzpicture}
=
\begin{tikzpicture}[baseline = -.1cm]
	\draw[thick, rho] (-1.4,-.4)-- (-.6,0) .. controls ++(45:.6cm) and ++(135:.6cm) .. (.6,0)--(1.4,-.4);	
	\draw[mid<] (-1.4,0) -- (-.9, 0);
	\draw[mid>] (-.6,0) -- (.6, 0);
	\draw[mid<] (1,0) -- (1.4, 0);
	\draw[thick, unshaded] (.6,0) circle (.4cm);
	\node at (.6,0) {\scriptsize{$V_{g^{-1}}^*$}};
	\node at (.6,.55) {$\star$};
	\draw[thick, unshaded] (-.6,0) circle (.4cm);
	\node at (-.6,0) {\scriptsize{$V_{g^{-1}}$}};
	\node at (-.6,.55) {$\star$};
\end{tikzpicture}
=
\begin{tikzpicture}[baseline = -.1cm]
	\draw[thick, rho] (-1.4,-.4)-- (-.6,0) .. controls ++(45:.6cm) and ++(135:.6cm) .. (.6,0)--(1.4,-.4);	
	\draw[mid<] (-1.4,0) -- (-.9, 0);
	\draw[mid>] (-.6,0) -- (.6, 0);
	\draw[mid<] (1,0) -- (1.4, 0);
	\draw[thick, unshaded] (.6,0) circle (.4cm);
	\node at (.6,0) {\scriptsize{$V_{g^{-1}}^*$}};
	\node at (.6,-.55) {$\star$};
	\draw[thick, unshaded] (-.6,0) circle (.4cm);
	\node at (-.6,0) {\scriptsize{$V_{g^{-1}}$}};
	\node at (-.6,-.55) {$\star$};
\end{tikzpicture}\,.
$
\end{enumerate}
\end{prop}
\begin{proof}
The case $g=e$ is trivial. 
When $g\neq e$, we prove (1), and (2) follows by replacing $g$ with $g^{-1}$.
The second equality follows from the fact that each of $V_g,V_g^*,V_{g^{-1}}, V_{g^{-1}}^*$ is fixed under $\cF_{\cR_\bullet}^4$.
To prove the first equality in (1), we see
$$
\begin{tikzpicture}[baseline = -.1cm, rotate=90]
	\draw[thick, rho] (-1.4,-.4)-- (-.6,0) .. controls ++(45:.6cm) and ++(135:.6cm) .. (.6,0)--(1.4,-.4);	
	\draw[mid>] (-1.4,0) -- (-.9, 0);
	\draw[mid<] (-.6,0) -- (.6, 0);
	\draw[mid>] (1,0) -- (1.4, 0);
	\draw[thick, unshaded] (.6,0) circle (.4cm);
	\node at (.6,0) {\scriptsize{$V_{g}^*$}};
	\node at (.6,.55) {$\star$};
	\draw[thick, unshaded] (-.6,0) circle (.4cm);
	\node at (-.6,0) {\scriptsize{$V_g$}};
	\node at (-.6,.55) {$\star$};
\end{tikzpicture}
=
[3]^2
\begin{tikzpicture}[baseline = -.1cm]
	\draw[thick, rho] (-.2,1) -- (-.2,1.2);
	\draw[thick, rho] (0,1) -- (0,1.2);
	\draw[thick, rho] (-.2,-1) -- (-.2,-1.2);
	\draw[thick, rho] (0,-1) -- (0,-1.2);
	\draw[thick, rho] (-.2,-.6) -- (0,-.2) -- (0,.2) -- (-.2,.6);
	\draw[thick, rho] (0,-.6) -- (.2,-.2) -- (.2,.2) -- (0,.6);
	\draw[thick, rho] (-.4,-1.2) -- (-.4,1.2);
	\draw[thick, rho] (.4,-1.2) -- (.4,1.2);
	\nbox{unshaded}{(-.2,.8)}{.2}{.1}{.1}{$g$}
	\nbox{unshaded}{(.2,0)}{.2}{.1}{.1}{$g$}
	\nbox{unshaded}{(-.2,-.8)}{.2}{.1}{.1}{$g$}
	\node at (-.65,-.8) {$\star$};	
	\node at (.65,0) {$\star$};	
	\node at (-.65,.8) {$\star$};
\end{tikzpicture}
=
[3]^2
\begin{tikzpicture}[baseline = -.1cm]
	\draw[thick, rho] (-.4,-.4) arc (-90:0:.15cm) -- (-.25,.25) arc (0:90:.15cm);
	\draw[thick, rho] (-.4,-.6) arc (-90:0:.4cm);
	\draw[thick, rho] (-.4,-.8) arc (-90:0:.6cm);
	\draw[thick, rho] (-.4,.6) arc (90:0:.4cm);
	\draw[thick, rho] (-.4,.8) arc (90:0:.6cm);
	\draw[thick, rho] (-.8,-.4) arc (90:180:.8cm);
	\draw[thick, rho] (-.8,-.6) arc (90:180:.6cm);
	\draw[thick, rho] (-.8,-.8) arc (90:180:.4cm);
	\draw[thick, rho] (-.8,.4) arc (270:180:.8cm);
	\draw[thick, rho] (-.8,.6) arc (270:180:.6cm);
	\draw[thick, rho] (-.8,.8) arc (270:180:.4cm);
	\draw[thick, rho] (.4,-1.2) -- (.4,1.2);
	\nbox{unshaded}{(-.6,.6)}{.3}{-.1}{-.1}{\rotatebox{90}{$g$}}
	\nbox{unshaded}{(.2,0)}{.2}{.1}{.1}{$g$}
	\nbox{unshaded}{(-.6,-.6)}{.3}{-.1}{-.1}{\rotatebox{-90}{$g$}}
	\node at (-.6,-.15) {$\star$};	
	\node at (.65,0) {$\star$};	
	\node at (-.6,.15) {$\star$};
\end{tikzpicture}
=
[3]^2
\begin{tikzpicture}[baseline = -.1cm]
	\draw[thick, rho] (.2,.2) arc (180:0:.2cm) -- (.6,-.2) arc (0:-180:.2cm);
	\draw[thick, rho] (0,.2) arc (180:0:.4cm);
	\draw[thick, rho] (-.2,.2) arc (180:0:.6cm);
	\draw[thick, rho] (0,-.2) arc (-180:0:.4cm);
	\draw[thick, rho] (-.2,-.2) arc (-180:0:.6cm);
	\draw[thick, rho] (-1.2,-1.2) -- (-1.2,1.2);
	\draw[thick, rho] (-1,-1.2) -- (-1,1.2);
	\draw[thick, rho] (-.8,-1.2) -- (-.8,1.2);
	\draw[thick, rho] (1.2,-1.2) -- (1.2,1.2);
	\nbox{unshaded}{(-1,0)}{.2}{.1}{.1}{$g$}
	\nbox{unshaded}{(0,0)}{.2}{.1}{.1}{$g$}
	\nbox{unshaded}{(1,0)}{.2}{.1}{.1}{$g$}
	\node at (-1.45,0) {$\star$};	
	\node at (.45,0) {$\star$};	
	\node at (1.45,0) {$\star$};
\end{tikzpicture}
=
\begin{tikzpicture}[baseline = -.1cm]
	\draw[thick, rho] (-.2,-1.2) -- (-.2,1.2);
	\draw[thick, rho] (0,-1.2) -- (0,1.2);
	\draw[thick, rho] (.2,-1.2) -- (.2,1.2);
	\draw[thick, rho] (.5,-1.2) -- (.5,1.2);
	\nbox{unshaded}{(0,0)}{.2}{.1}{.1}{$g$}
	\node at (-.45,0) {$\star$};	
\end{tikzpicture}
$$
using the skein relation from Remark \ref{rem:AutomorphismSkein} for $g$-cabled strands.
\end{proof}

\begin{prop}\label{prop:GradedMultiplication}
The map $\Phi_g$ is compatible with the graded multiplication operator given for $x\in \cR_m$ and $y\in \cR_n$ by
$$
x\wedge y = 
\begin{tikzpicture}[baseline = -.1cm]	
	\draw[thick, rho] (.4,.3) -- (.4,.7);
	\draw[thick, rho] (-.4,.3) -- (-.4,.7);
	\node at (-.2,.5) {\scriptsize{$m$}};
	\node at (.6,.5) {\scriptsize{$n$}};
	\nbox{unshaded}{(-.4,0)}{.3}{0}{0}{$x$}
	\nbox{unshaded}{(.4,0)}{.3}{0}{0}{$y$}
\end{tikzpicture}\,.
$$ 
\end{prop}
\begin{proof}
Just use the skein relation in Remark \ref{rem:AutomorphismSkein}:
\begin{align*}
\Phi_g(x)\wedge \Phi_g(y) 
&=
\begin{tikzpicture}[baseline = .5cm]	
	\draw[thick, rho] (-.7,.3) -- (-.7,1.5);
	\draw[thick, rho] (.7,.3) -- (.7,1.5);
	\node at (0,.8) {\scriptsize{$\cdots$}};
	\draw[mid<] (-.8,1) arc (90:270:.6) -- (.8,-.2) arc (-90:90:.6cm);
	\draw[mid>] (-.2,1) -- (-.55,1);
	\draw[mid<] (.2,1) -- (.55,1);
	\ncircle{unshaded}{(-.7,1)}{.25}{135}{{\scriptsize{$V_g$}}}
	\ncircle{unshaded}{(.7,1)}{.25}{45}{{\scriptsize{$V_{g}^*$}}}
	\nbox{unshaded}{(0,.3)}{.2}{.7}{.7}{$x$}
\end{tikzpicture}\,\,
\begin{tikzpicture}[baseline = .5cm]	
	\draw[thick, rho] (-.7,.3) -- (-.7,1.5);
	\draw[thick, rho] (.7,.3) -- (.7,1.5);
	\node at (0,.8) {\scriptsize{$\cdots$}};
	\draw[mid<] (-.8,1) arc (90:270:.6) -- (.8,-.2) arc (-90:90:.6cm);
	\draw[mid>] (-.2,1) -- (-.55,1);
	\draw[mid<] (.2,1) -- (.55,1);
	\ncircle{unshaded}{(-.7,1)}{.25}{135}{{\scriptsize{$V_g$}}}
	\ncircle{unshaded}{(.7,1)}{.25}{45}{{\scriptsize{$V_{g}^*$}}}
	\nbox{unshaded}{(0,.3)}{.2}{.7}{.7}{$y$}
\end{tikzpicture}
=
\begin{tikzpicture}[baseline = .5cm]	
	\draw[thick, rho] (-.7,.3) -- (-.7,1.5);
	\draw[thick, rho] (.7,.3) -- (.7,1.5);
	\draw[thick, rho] (1.5,.3) -- (1.5,1.5);
	\draw[thick, rho] (2.9,.3) -- (2.9,1.5);
	\node at (0,.8) {\scriptsize{$\cdots$}};
	\node at (2.2,.8) {\scriptsize{$\cdots$}};
	\draw[mid<] (-.8,1) arc (90:270:.6) -- (3,-.2) arc (-90:90:.6cm);
	\draw[mid>] (-.2,1) -- (-.55,1);
	\draw[mid<] (.2,1) -- (.55,1);
	\draw[mid>] (2,1) -- (1.65,1);
	\draw[mid<] (2.4,1) -- (2.75,1);
	\draw[mid>] (.8,1) -- (1.5,1);
	\ncircle{unshaded}{(-.7,1)}{.25}{135}{{\scriptsize{$V_g$}}}
	\ncircle{unshaded}{(.7,1)}{.25}{60}{{\scriptsize{$V_{g}^*$}}}
	\ncircle{unshaded}{(1.5,1)}{.25}{120}{{\scriptsize{$V_g$}}}
	\ncircle{unshaded}{(2.9,1)}{.25}{45}{{\scriptsize{$V_{g}^*$}}}
	\nbox{unshaded}{(0,.3)}{.2}{.7}{.7}{$x$}
	\nbox{unshaded}{(2.2,.3)}{.2}{.7}{.7}{$y$}
\end{tikzpicture}
=
\Phi_g(x\wedge y).
\qedhere
\end{align*}
\end{proof}

\begin{cor}\label{cor:ActionOnEvenHalf}
The `almost' action of $\Phi_g$ on $\cR_\bullet$ induces the action of $g^2$ on $\frac{1}{2}\cP_+$.
\end{cor}
\begin{proof}
For $h,g\in G$, we will show that $\Phi_g(p_h\in \cS_{2,+}) =p_{g^2h}$ and $\Phi_g(h\in \cS_{3,+}) \cong g^2h$.
As both proofs are similar, we will only show $\Phi_g(p_h)=p_{g^2h}$.

First, by Propositions \ref{prop:ReidemeisterII} and \ref{prop:GradedMultiplication} together with the fact that $h$ is an orthogonal projection, we see $\Phi_g(p_h)^2 = \Phi_g(p_h)=\Phi_g(p_h)^*$, i.e., $\Phi_g(p_h)$ is an orthogonal projection.
Next, taking the trace, we see $\Phi_g(p_h)\neq 0$ by sphericality and again using the Reidemeister II relation from Proposition \ref{prop:ReidemeisterII}:
$$
\begin{tikzpicture}[baseline = -.1cm]
	\draw[thick, rho] (-.4,.9) -- (-.4,-.9) .. controls ++(270:.6cm) and ++(270:.6cm) .. (2.4,-.9) -- (2.4,.9) .. controls ++(90:.6cm) and ++(90:.6cm) .. (-.4,.9);
	\draw[thick, rho] (.4,.9) -- (.4,-.9) .. controls ++(270:.3cm) and ++(270:.3cm) .. (1.6,-.9) -- (1.6,.9) .. controls ++(90:.3cm) and ++(90:.3cm) .. (.4,.9);
	\draw[mid>] (-.4,-.6) arc (270:90:.6cm);
	\draw[mid>] (.4,.6) arc (90:-90:.6cm);
	\draw[mid>] (.4,.6) -- (-.4,.6);
	\draw[mid>] (-.4,-.6) -- (.4,-.6);
	\nbox{unshaded}{(0,0)}{.2}{.3}{.3}{$p_h$}
	\ncircle{unshaded}{(-.4,.6)}{.25}{150}{{\scriptsize{$V_g$}}}
	\ncircle{unshaded}{(.4,.6)}{.25}{30}{{\scriptsize{$V_{g}^*$}}}
	\ncircle{unshaded}{(-.4,-.6)}{.25}{210}{{\scriptsize{$V_g$}}}
	\ncircle{unshaded}{(.4,-.6)}{.25}{-30}{{\scriptsize{$V_{g}^*$}}}
\end{tikzpicture}
=
\begin{tikzpicture}[baseline = -.1cm, xscale=-1]
	\draw[thick, rho] (-.4,.9) -- (-.4,-.9) .. controls ++(270:.6cm) and ++(270:.6cm) .. (2.4,-.9) -- (2.4,.9) .. controls ++(90:.6cm) and ++(90:.6cm) .. (-.4,.9);
	\draw[thick, rho] (.4,.9) -- (.4,-.9) .. controls ++(270:.3cm) and ++(270:.3cm) .. (1.6,-.9) -- (1.6,.9) .. controls ++(90:.3cm) and ++(90:.3cm) .. (.4,.9);
	\draw[mid>] (-.4,-.6) arc (270:90:.6cm);
	\draw[mid>] (.4,.6) arc (90:-90:.6cm);
	\draw[mid>] (.4,.6) -- (-.4,.6);
	\draw[mid>] (-.4,-.6) -- (.4,-.6);
	\nbox{unshaded}{(2,0)}{.2}{.3}{.3}{$p_h$}
	\ncircle{unshaded}{(-.4,.6)}{.25}{150}{{\scriptsize{$V_g$}}}
	\ncircle{unshaded}{(.4,.6)}{.25}{30}{{\scriptsize{$V_{g}^*$}}}
	\ncircle{unshaded}{(-.4,-.6)}{.25}{210}{{\scriptsize{$V_g$}}}
	\ncircle{unshaded}{(.4,-.6)}{.25}{-30}{{\scriptsize{$V_{g}^*$}}}
\end{tikzpicture}=\Tr(p_h)=[3].
$$
Finally, it's obvious that $\Phi_g(p_h)\cong g\otimes (h\rho)\otimes g^{-1}\cong g^2h\rho$, so $\Phi_g(p_h)=p_{g^2h}$ since $\cR_2$ is abelian.
\end{proof}

\begin{prop}\label{prop:TrueForSquaredg}
For every $g\in G$, we have 
$$
\cF_{\cR_\bullet}(p_{g^2})=p_{g^{-2}}
- \frac{1}{[3]-1} 
\left(\,
\begin{tikzpicture}[baseline = -.1cm]
	\nbox{unshaded}{(0,0)}{.4}{0}{0}{}
	\draw[thick, rho] (-.2,.4) arc (-180:0:.2);
	\draw[thick, rho] (-.2,-.4) arc (180:0:.2);	
\end{tikzpicture}
-
\begin{tikzpicture}[baseline = -.1cm]
	\draw[thick, rho] (-.2,-.4) -- ( -.2, .4);	
	\draw[thick, rho] (.2,-.4) -- (.2, .4);
	\nbox{}{(0,0)}{.4}{0}{0}{}
\end{tikzpicture}
\,\right).
$$
(Compare the above equation with Corollary \ref{cor:EquivalentConjecture}.)
\end{prop}
\begin{proof}
By Relation \ref{rel:IH}, we know 
$$
\cF_{\cR_\bullet}(p_1) = p_{1} 
- \frac{1}{[3]-1} 
\left(\,
\begin{tikzpicture}[baseline = -.1cm]
	\nbox{unshaded}{(0,0)}{.4}{0}{0}{}
	\draw[thick, rho] (-.2,.4) arc (-180:0:.2);
	\draw[thick, rho] (-.2,-.4) arc (180:0:.2);	
\end{tikzpicture}
-
\begin{tikzpicture}[baseline = -.1cm]
	\draw[thick, rho] (-.2,-.4) -- ( -.2, .4);	
	\draw[thick, rho] (.2,-.4) -- (.2, .4);
	\nbox{}{(0,0)}{.4}{0}{0}{}
\end{tikzpicture}
\,\right).
$$
Apply $\Phi_g$ to the equation to see
\begin{align*}
\cF_{\cR_\bullet}(p_{g^2})
&= 
\cF_{\cR_\bullet}(\Phi_g(p_1))
\\&=
\Phi_{g^{-1}}(\cF_{\cR_\bullet}(p_1))
\\&=
\Phi_{g^{-1}}\left(p_{1} 
- \frac{1}{[3]-1} 
\left(\,
\begin{tikzpicture}[baseline = -.1cm]
	\nbox{unshaded}{(0,0)}{.4}{0}{0}{}
	\draw[thick, rho] (-.2,.4) arc (-180:0:.2);
	\draw[thick, rho] (-.2,-.4) arc (180:0:.2);	
\end{tikzpicture}
-
\begin{tikzpicture}[baseline = -.1cm]
	\draw[thick, rho] (-.2,-.4) -- ( -.2, .4);	
	\draw[thick, rho] (.2,-.4) -- (.2, .4);
	\nbox{}{(0,0)}{.4}{0}{0}{}
\end{tikzpicture}
\,\right)
\right)
\\&=
\Phi_{g^{-1}}(p_{1})  
- \frac{1}{[3]-1} 
\left(\,
\begin{tikzpicture}[baseline = -.1cm]
	\nbox{unshaded}{(0,0)}{.4}{0}{0}{}
	\draw[thick, rho] (-.2,.4) arc (-180:0:.2);
	\draw[thick, rho] (-.2,-.4) arc (180:0:.2);	
\end{tikzpicture}
-
\begin{tikzpicture}[baseline = -.1cm]
	\draw[thick, rho] (-.2,-.4) -- ( -.2, .4);	
	\draw[thick, rho] (.2,-.4) -- (.2, .4);
	\nbox{}{(0,0)}{.4}{0}{0}{}
\end{tikzpicture}
\,\right)
\\&=
p_{g^{-2}}
- \frac{1}{[3]-1} 
\left(\,
\begin{tikzpicture}[baseline = -.1cm]
	\nbox{unshaded}{(0,0)}{.4}{0}{0}{}
	\draw[thick, rho] (-.2,.4) arc (-180:0:.2);
	\draw[thick, rho] (-.2,-.4) arc (180:0:.2);	
\end{tikzpicture}
-
\begin{tikzpicture}[baseline = -.1cm]
	\draw[thick, rho] (-.2,-.4) -- ( -.2, .4);	
	\draw[thick, rho] (.2,-.4) -- (.2, .4);
	\nbox{}{(0,0)}{.4}{0}{0}{}
\end{tikzpicture}
\,\right),
\end{align*}
where we used $\Phi_{g^{-1}} \circ \cF_{\cR_\bullet} = \cF_{\cR_\bullet} \circ \Phi_{g}$ by Corollary \ref{cor:1ClickProblem}.
\end{proof}

We can now prove our main theorem, which says that Conjecture \ref{conj:Main} is true for $|G|$ odd.

\begin{proof}[Proof of Theorem \ref{thm:Main}]
Note that the action induced by $\Phi$ of $G$ on $G\rho=\set{h\rho}{h\in G}$ is freely transitive exactly when $|G|$ is odd.
This is because the action is given by $\Phi_g(h\rho)=g^2h\rho$ and $\Phi_g(h)=g^2h$ by Corollary \ref{cor:ActionOnEvenHalf}, and the map $g\mapsto g^2$ is an automorphism of $G$ when $G$ is odd.

Thus for every $g\in G$, the equation in Corollary \ref{cor:EquivalentConjecture} holds by Proposition \ref{prop:TrueForSquaredg}, which concludes the proof.
\end{proof}

%%%%%%%%%%%%%%%%%%%%%%%%%%%%%%%%%%%%%%%%%%%%%%%%%%%%%%%%%%%%%%
\subsection{Lifting involutions to the center}

\begin{defn}
Let $\cQ_\bullet$ be the the planar subalgebra of $\cR_\bullet$ generated by 
$$
\set{
\begin{tikzpicture}[baseline = -.1cm]
	\draw[thick, rho] (-.1,.5) -- ( -.1, -.5);	
	\draw[thick, rho] (.1,.5) -- ( .1, -.5);	
	\nbox{unshaded}{(0,0)}{.25}{0}{0}{$p_g$}
\end{tikzpicture}
}{g\neq 1}
\text{ and }
\begin{tikzpicture}[baseline = -.1cm]
	\draw[thick, rho] (-.15,.6) --  ( -.15, -.6);	
	\draw[thick, rho] (.15,.6) --  ( .15, -.6);	
	\nbox{unshaded}{(0,0)}{.3}{0}{0}{$p_1$}
\end{tikzpicture}
:=
\begin{tikzpicture}[baseline = -.1cm]
	\draw[thick, rho] (-.2,.6) --  ( -.2, -.6);	
	\draw[thick, rho] (.2,.6) --  ( .2, -.6);	
	\nbox{unshaded}{(0,0)}{.4}{0}{0}{}
	\filldraw[rho] (0,.2) circle (.05cm);
	\filldraw[rho] (0,-.2) circle (.05cm);
	\draw[thick, rho] (-.2,.4) arc (-180:0:.2);
	\draw[thick, rho] (-.2,-.4) arc (180:0:.2);	
	\draw[thick, rho] (0,-.2) -- (0, .2);
\end{tikzpicture}\,,
$$
under the $\rho$-strand planar operad, whose tangles do not contain trivalent vertices.
\end{defn}

We expect the following conjecture to be true, but it seems to be highly non-trivial at this time.
We will prove it for the case $|G|$ is odd in Theorem \ref{thm:GeneratedByTwoBoxes}.

\begin{conj}\label{conj:GeneratedByTwoBoxes}
$\cQ_\bullet=\cR_\bullet$.
\end{conj}

We now show that each involution in $G$, i.e., a $g\in G$ with $g^2=1$, lifts to the center of the projection category of $\cQ_\bullet$, i.e., $\cZ(\Pro(\cQ_\bullet))$ \cite{1208.5505}.
By Corollary \ref{cor:ActionOnEvenHalf}, each $\Phi_g$ restricts to an automorphism of $\cQ_\bullet$.

\begin{lem}\label{lem:IdentityAutomorphism}
If $g^2=1$, then $\Phi_g = \id_{\cQ_\bullet}$.
\end{lem}
\begin{proof}
In the proof of Corollary \ref{cor:ActionOnEvenHalf}, we showed that $\Phi_g(p_h)=p_{g^2h}$ for all $h\in H$. 
Since $g^2=1$, $\Phi_g$ fixes every $p_h$, which generate $\cQ_\bullet$ as a planar algebra.
\end{proof}

\begin{prop}
For all $\varphi\in \cQ_n$, we have 
$$
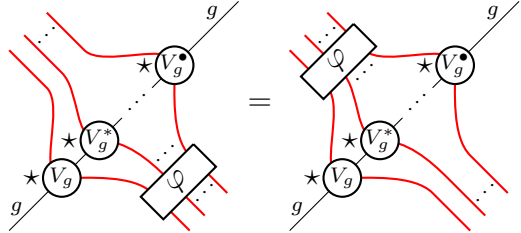
\begin{tikzpicture}[baseline = -.1cm]
	\coordinate (a) at (.35,-1.55);
	\draw (.2,.2) -- (1.35,1.35);
	\draw (-1.7,-1.7) -- (-.2,-.2);
	\node at (-1.6,-1.4) {\scriptsize{$g$}};
	\node at (1,1.2) {\scriptsize{$g$}};
	\node at (.02,.02) {\rotatebox{45}{{\scriptsize{$\cdots$}}}};
	\node at ($(a) + (.57,.37)$) {\rotatebox{45}{{\scriptsize{$\cdots$}}}};
	\node at ($(a) + (-1.53,2.47)$) {\rotatebox{45}{{\scriptsize{$\cdots$}}}};
	\node at ($(a) + (.05,.9)$) {\rotatebox{45}{{\scriptsize{$\cdots$}}}};
	\draw[thick, rho] ($(a) + (-1.55,2.75)$) -- ($(a) + (-1.05,2.25)$) .. controls ++(-45:.35cm) and ++(90:.35cm) .. (.5,.5) .. controls ++(270:.35cm) and ++(135:.35cm) .. ($(a) + (.35,.85)$)  -- ($(a) + (.85,.35)$);
	\draw[thick, rho] ($(a) + (-1.85,2.45)$) -- ($(a) + (-1.35,1.95)$) .. controls ++(-45:.35cm) and ++(180:.35cm) .. (-.5,-.5) .. controls ++(0:.35cm) and ++(135:.35cm) .. ($(a) + (.05,.55)$) -- ($(a) + (.55,.05)$);
	\draw[thick, rho] ($(a) + (-2.05,2.25)$) -- ($(a) + (-1.55,1.75)$) .. controls ++(-45:.35cm) and ++(180:.35cm) .. (-1,-1) .. controls ++(0:.35cm) and ++(135:.35cm) .. ($(a) + (-.15,.35)$) -- ($(a) + (.35,-.15)$);
	\ncircle{unshaded}{(.5,.5)}{.25}{180}{{\scriptsize{$V_{g}^\bullet$}}}
	\ncircle{unshaded}{(-.5,-.5)}{.25}{180}{{\scriptsize{$V_{g}^*$}}}
	\ncircle{unshaded}{(-1,-1)}{.25}{180}{{\scriptsize{$V_{g}$}}}
	\draw[thick, unshaded] ($(a) + (.4,1)$) -- ($(a) + (-.3,.3)$)-- (a) -- ($(a) + (.7,.7)$) -- ($(a) + (.4,1)$);
	\node at ($(a) + (.2,.5)$)  {\rotatebox{45}{$\varphi$}};
\end{tikzpicture}
=
\begin{tikzpicture}[baseline = -.1cm]
	\coordinate (a) at (.35,-1.55);
	\coordinate (b) at (-1.25,.05);
	\draw (.2,.2) -- (1.35,1.35);
	\draw (-1.7,-1.7) -- (-.2,-.2);
	\node at (-1.6,-1.4) {\scriptsize{$g$}};
	\node at (1,1.2) {\scriptsize{$g$}};
	\node at (.02,.02) {\rotatebox{45}{{\scriptsize{$\cdots$}}}};
	\node at ($(a) + (.57,.37)$) {\rotatebox{45}{{\scriptsize{$\cdots$}}}};
	\node at ($(a) + (-1.53,2.47)$) {\rotatebox{45}{{\scriptsize{$\cdots$}}}};
	\node at ($(b) + (.57,.37)$) {\rotatebox{45}{{\scriptsize{$\cdots$}}}};
	\draw[thick, rho] ($(a) + (-1.55,2.75)$) -- ($(a) + (-1.05,2.25)$) .. controls ++(-45:.35cm) and ++(90:.35cm) .. (.5,.5) .. controls ++(270:.35cm) and ++(135:.35cm) .. ($(a) + (.35,.85)$)  -- ($(a) + (.85,.35)$);
	\draw[thick, rho] ($(a) + (-1.85,2.45)$) -- ($(a) + (-1.35,1.95)$) .. controls ++(-45:.35cm) and ++(180:.35cm) .. (-.5,-.5) .. controls ++(0:.35cm) and ++(135:.35cm) .. ($(a) + (.05,.55)$) -- ($(a) + (.55,.05)$);
	\draw[thick, rho] ($(a) + (-2.05,2.25)$) -- ($(a) + (-1.55,1.75)$) .. controls ++(-45:.35cm) and ++(180:.35cm) .. (-1,-1) .. controls ++(0:.35cm) and ++(135:.35cm) .. ($(a) + (-.15,.35)$) -- ($(a) + (.35,-.15)$);
	\ncircle{unshaded}{(.5,.5)}{.25}{180}{{\scriptsize{$V_{g}^\bullet$}}}
	\ncircle{unshaded}{(-.5,-.5)}{.25}{180}{{\scriptsize{$V_{g}^*$}}}
	\ncircle{unshaded}{(-1,-1)}{.25}{180}{{\scriptsize{$V_{g}$}}}
	\draw[thick, unshaded] ($(b) + (.4,1)$) -- ($(b) + (-.3,.3)$)-- (b) -- ($(b) + (.7,.7)$) -- ($(b) + (.4,1)$);
	\node at ($(b) + (.2,.5)$)  {\rotatebox{45}{$\varphi$}};
\end{tikzpicture}
$$
where the $\bullet$ is either blank or $*$ depending on the parity of $n$.
\end{prop}
\begin{proof}
Take the norm squared of the difference and use Lemma \ref{lem:IdentityAutomorphism}.
\end{proof}

\begin{cor}
Suppose $g\in G$ is an involution.
The map $e_g: X\to \Hom(g\otimes X, X\otimes g)$ for simple $X$ by $e_g(\rho)=V_g$ and
$$
e_g(h) =
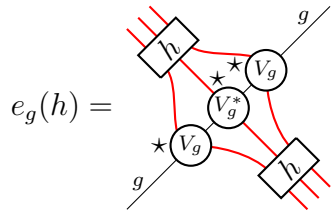
\begin{tikzpicture}[baseline = -.1cm]
	\draw (-1.35,-1.35) -- (1.35,1.35);
	\node at (-1.2,-1) {\scriptsize{$g$}};
	\node at (1,1.2) {\scriptsize{$g$}};
	\draw[thick, rho] (-1.05,1.35) -- (-.55,.85) .. controls ++(-45:.35cm) and ++(90:.35cm) .. (.5,.5) .. controls ++(270:.35cm) and ++(135:.35cm) .. (.85,-.55) -- (1.35,-1.05);
	\draw[thick, rho] (-1.2,1.2) -- (1.2,-1.2);
	\draw[thick, rho] (-1.35,1.05) -- (-.85,.55) .. controls ++(-45:.35cm) and ++(180:.35cm) .. (-.5,-.5) .. controls ++(0:.35cm) and ++(135:.35cm) .. (.55,-.85) -- (1.05,-1.35);
	\ncircle{unshaded}{(.5,.5)}{.27}{170}{{\scriptsize{$V_{g}$}}}
	\ncircle{unshaded}{(0,0)}{.27}{110}{{\scriptsize{$V_g^*$}}}
	\ncircle{unshaded}{(-.5,-.5)}{.27}{180}{{\scriptsize{$V_{g}$}}}
	\draw[thick, unshaded] (-.9,.4) -- (-.4,.9) -- (-.7,1.2) -- (-1.2,.7) -- (-.9,.4);
	\node at (-.8,.8) {\rotatebox{45}{$h$}};
	\draw[thick, unshaded] (.9,-.4) -- (.4,-.9) -- (.7,-1.2) -- (1.2,-.7) -- (.9,-.4);
	\node at (.8,-.8) {\rotatebox{45}{$h$}};
\end{tikzpicture}
$$
naturally extends to a half-braiding.
Hence $(g,e_g)$ defines an element in the center $\cZ(\Pro(\cQ_\bullet))$.
\end{cor}

\begin{cor}
Suppose $g\in G$ is an involution.
The twist factor of $(g,e_g)$ is given by $\theta_g$.
\end{cor}
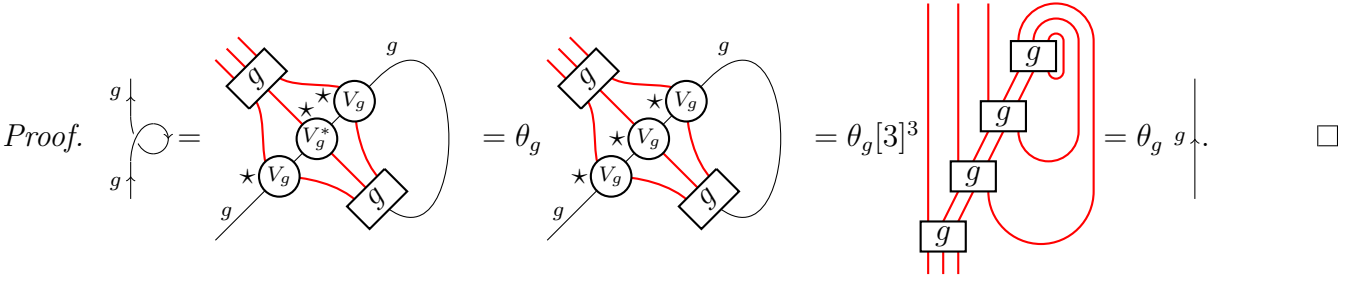
\begin{proof}
$
\begin{tikzpicture}[baseline = -.1cm]
	\draw[mid>] (0,-.8) -- (0,-.2);
	\draw[mid<] (0,.8) -- (0,.2);
	\node at (-.2,-.6) {\scriptsize{$g$}};
	\node at (-.2,.6) {\scriptsize{$g$}};
	\coordinate (a) at (0,0);
	\fill[unshaded] ($ (a) - (.1,.3) $) rectangle ($ (a) + (.3,.3) $);
	\draw[mid>] ($ (a) + (.3,.2) $) arc (90:-90:.2cm);
	\draw ($ (a) + (0,.3) $)  .. controls ++(270:.2cm) and ++(180:.2cm) .. ($ (a) + (.3,-.2) $);
	\draw[super thick, white] ($ (a) + (0,-.3) $)  .. controls ++(90:.2cm) and ++(180:.2cm) .. ($ (a) + (.3,.2) $);
	\draw ($ (a) + (0,-.3) $)  .. controls ++(90:.2cm) and ++(180:.2cm) .. ($ (a) + (.3,.2) $);
\end{tikzpicture}
=
\begin{tikzpicture}[baseline = -.1cm]
	\draw (-1.35,-1.35) -- (.8,.8) .. controls ++(45:1.8cm) and ++(-45:1.8cm) .. (.8,-.8);
	\node at (-1.2,-1) {\scriptsize{$g$}};
	\node at (1,1.2) {\scriptsize{$g$}};
	\draw[thick, rho] (-1.05,1.35) -- (-.55,.85) .. controls ++(-45:.35cm) and ++(90:.35cm) .. (.5,.5) .. controls ++(270:.35cm) and ++(135:.35cm) .. (.85,-.55);
	\draw[thick, rho] (-1.2,1.2) -- (.8,-.8);
	\draw[thick, rho] (-1.35,1.05) -- (-.85,.55) .. controls ++(-45:.35cm) and ++(180:.35cm) .. (-.5,-.5) .. controls ++(0:.35cm) and ++(135:.35cm) .. (.55,-.85);
	\ncircle{unshaded}{(.5,.5)}{.27}{170}{{\scriptsize{$V_{g}$}}}
	\ncircle{unshaded}{(0,0)}{.27}{110}{{\scriptsize{$V_g^*$}}}
	\ncircle{unshaded}{(-.5,-.5)}{.27}{180}{{\scriptsize{$V_{g}$}}}
	\draw[thick, unshaded] (-.9,.4) -- (-.4,.9) -- (-.7,1.2) -- (-1.2,.7) -- (-.9,.4);
	\node at (-.8,.8) {\rotatebox{45}{$g$}};
	\draw[thick, unshaded] (.9,-.4) -- (.4,-.9) -- (.7,-1.2) -- (1.2,-.7) -- (.9,-.4);
	\node at (.8,-.8) {\rotatebox{45}{$g$}};
\end{tikzpicture}
=
\theta_g
\begin{tikzpicture}[baseline = -.1cm]
	\draw (-1.35,-1.35) -- (.8,.8) .. controls ++(45:1.8cm) and ++(-45:1.8cm) .. (.8,-.8);
	\node at (-1.2,-1) {\scriptsize{$g$}};
	\node at (1,1.2) {\scriptsize{$g$}};
	\draw[thick, rho] (-1.05,1.35) -- (-.55,.85) .. controls ++(-45:.35cm) and ++(90:.35cm) .. (.5,.5) .. controls ++(270:.35cm) and ++(135:.35cm) .. (.85,-.55);
	\draw[thick, rho] (-1.2,1.2) -- (.8,-.8);
	\draw[thick, rho] (-1.35,1.05) -- (-.85,.55) .. controls ++(-45:.35cm) and ++(180:.35cm) .. (-.5,-.5) .. controls ++(0:.35cm) and ++(135:.35cm) .. (.55,-.85);
	\ncircle{unshaded}{(.5,.5)}{.27}{180}{{\scriptsize{$V_{g}$}}}
	\ncircle{unshaded}{(0,0)}{.27}{180}{{\scriptsize{$V_g$}}}
	\ncircle{unshaded}{(-.5,-.5)}{.27}{180}{{\scriptsize{$V_{g}$}}}
	\draw[thick, unshaded] (-.9,.4) -- (-.4,.9) -- (-.7,1.2) -- (-1.2,.7) -- (-.9,.4);
	\node at (-.8,.8) {\rotatebox{45}{$g$}};
	\draw[thick, unshaded] (.9,-.4) -- (.4,-.9) -- (.7,-1.2) -- (1.2,-.7) -- (.9,-.4);
	\node at (.8,-.8) {\rotatebox{45}{$g$}};
\end{tikzpicture}
=\theta_g[3]^3
\begin{tikzpicture}[baseline = 0cm]
	\draw[thick, rho] (-.8,-1.4) -- (-.8,-1.7);
	\draw[thick, rho] (-.6,-1.4) -- (-.6,-1.7);
	\draw[thick, rho] (-.4,-1.4) -- (-.4,-1.7);
	\draw[thick, rho] (.8,1.4) arc (180:0:.1cm) -- (1, 1) arc (0:-180:.1cm);
	\draw[thick, rho] (.6,1.4) arc (180:0:.3cm) -- (1.2,.2) arc (0:-180:.4cm);
	\draw[thick, rho] (.4,1.4) arc (180:0:.5cm) -- (1.4,-.6) arc (0:-180:.7cm);	

	\draw[thick, rho] (-.8,-1.2) -- (-.8,1.9);
	\draw[thick, rho] (-.4,-.4) -- (-.4,1.9);
	\draw[thick, rho] (0,.4) -- (0,1.9);
	\draw[thick, rho] (-.7,-1.2) -- (.5,1.2);
	\draw[thick, rho] (-.5,-1.2) -- (.7,1.2);
	\nbox{unshaded}{(-.2,-.4)}{.2}{.1}{.1}{$g$}
	\nbox{unshaded}{(.2,.4)}{.2}{.1}{.1}{$g$}
	\nbox{unshaded}{(-.6,-1.2)}{.2}{.1}{.1}{$g$}
	\nbox{unshaded}{(.6,1.2)}{.2}{.1}{.1}{$g$}
\end{tikzpicture}=
\theta_g
\begin{tikzpicture}[baseline = -.1cm]
	\draw[mid>] (0,-.8) -- (0,.8);
	\node at (-.2,0) {\scriptsize{$g$}};
\end{tikzpicture}\,.
$
\end{proof}

%%%%%%%%%%%%%%%%%%%%%%%%%%%%%%%%%%%%%%%%%%%%%%%%%%%%%%%%%%%%%%%%%%
%%%%%%%%%%%%%%%%%%%%%%%%%%%%%%%%%%%%%%%%%%%%%%%%%%%%%%%%%%%%%%%%%%
%%%%%%%%%%%%%%%%%%%%%%%%%%%%%%%%%%%%%%%%%%%%%%%%%%%%%%%%%%%%%%%%%%
\section{The $G$-action on $\cS_\bullet$ and bases for $\cS_{3,+}$}\label{sec:ActionOnS}

We saw each $\Phi_g$ almost gives an automorphism of the unshaded factor planar algebra $\cR_\bullet$, except for the problem with the 1-click rotation from Corollary \ref{cor:1ClickProblem}.
However, $\Phi_g$ is compatible with the 2-click rotation, which suggests that the $\Phi_g$'s can be used to construct automorphisms of the shaded subfactor planar algebra $\cS_\bullet$.
There is a slight technicality here -- when we identify $\cS_{n,\pm}$ with $\cR_n$, the map $\Phi_g$ goes from $\cS_{n,\pm}\to \cS_{n,\mp}$.
Hence to get a planar algebra map, we need to also use the symmetric self-duality $\Delta_\pm:\cS_\pm\to \cS_\mp$ which reverses the shading. 

\begin{defn}
For $g\in G$, define $\Psi_g$ on $\cS_{n,\pm}=\cR_n$ by $\Psi_g=\Delta_\mp\circ \Phi_{g^{\pm 1}}$.
\end{defn}

\begin{ex}
When $x\in \cS_{2,+}$ and $y\in \cS_{3,-}$, we have
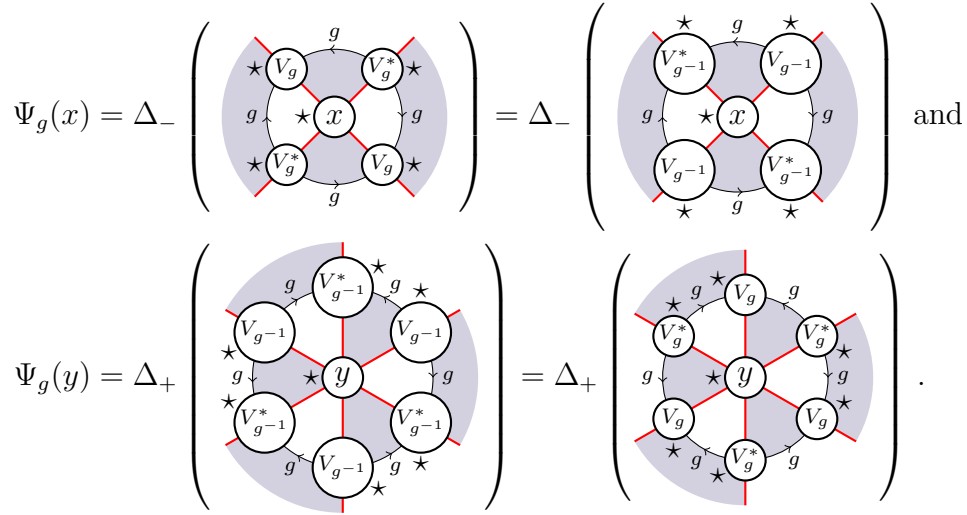
\begin{align*}
\Psi_g(x) &= 
\Delta_-\left(
\begin{tikzpicture}[baseline = -.1cm]
	\pgfmathsetmacro{\innerRadius}{.9}
	\pgfmathsetmacro{\middleRadius}{1.1}
	\pgfmathsetmacro{\outerRadius}{1.5}
	\fill[shaded] (0,0) -- (45:\innerRadius) arc (45:135:\innerRadius) -- (0,0);
	\fill[shaded] (0,0) -- (225:\innerRadius) arc (225:315:\innerRadius) -- (0,0);
	\fill[shaded] (-45:\innerRadius) arc (-45:45:\innerRadius) -- (45:\outerRadius) arc (45:-45:\outerRadius) -- (-45:\innerRadius);
	\fill[shaded] (135:\innerRadius) arc (135:225:\innerRadius) -- (225:\outerRadius) arc (225:135:\outerRadius) -- (135:\innerRadius);
	\draw[mid>] (225:\innerRadius) arc (225:135:\innerRadius);
	\draw[mid<] (135:\innerRadius) arc (135:45:\innerRadius);
	\draw[mid>] (45:\innerRadius) arc (45:-45:\innerRadius);
	\draw[mid<] (315:\innerRadius) arc (315:225:\innerRadius);
	\draw[thick, rho] (0,0) -- (45:\outerRadius);
	\draw[thick, rho] (0,0) -- (135:\outerRadius);
	\draw[thick, rho] (0,0) -- (225:\outerRadius);
	\draw[thick, rho] (0,0) -- (315:\outerRadius);
	\node at (0:\middleRadius) {\scriptsize{$g$}};
	\node at (90:\middleRadius) {\scriptsize{$g$}};
	\node at (180:\middleRadius) {\scriptsize{$g$}};
	\node at (270:\middleRadius) {\scriptsize{$g$}};
	\ncircle{unshaded}{(45:\innerRadius)}{.27}{0}{{\scriptsize{$V_{g}^*$}}}
	\ncircle{unshaded}{(135:\innerRadius)}{.27}{180}{{\scriptsize{$V_g$}}}
	\ncircle{unshaded}{(225:\innerRadius)}{.27}{180}{{\scriptsize{$V_{g}^*$}}}
	\ncircle{unshaded}{(315:\innerRadius)}{.27}{0}{{\scriptsize{$V_g$}}}
	\ncircle{unshaded}{(0,0)}{.27}{180}{$x$}
\end{tikzpicture}
\right)
=
\Delta_-\left(
\begin{tikzpicture}[baseline = -.1cm]
	\pgfmathsetmacro{\innerRadius}{1}
	\pgfmathsetmacro{\middleRadius}{1.2}
	\pgfmathsetmacro{\outerRadius}{1.6}
	\fill[shaded] (0,0) -- (45:\innerRadius) arc (45:135:\innerRadius) -- (0,0);
	\fill[shaded] (0,0) -- (225:\innerRadius) arc (225:315:\innerRadius) -- (0,0);
	\fill[shaded] (-45:\innerRadius) arc (-45:45:\innerRadius) -- (45:\outerRadius) arc (45:-45:\outerRadius) -- (-45:\innerRadius);
	\fill[shaded] (135:\innerRadius) arc (135:225:\innerRadius) -- (225:\outerRadius) arc (225:135:\outerRadius) -- (135:\innerRadius);
	\draw[mid>] (225:\innerRadius) arc (225:135:\innerRadius);
	\draw[mid<] (135:\innerRadius) arc (135:45:\innerRadius);
	\draw[mid>] (45:\innerRadius) arc (45:-45:\innerRadius);
	\draw[mid<] (315:\innerRadius) arc (315:225:\innerRadius);
	\draw[thick, rho] (0,0) -- (45:\outerRadius);
	\draw[thick, rho] (0,0) -- (135:\outerRadius);
	\draw[thick, rho] (0,0) -- (225:\outerRadius);
	\draw[thick, rho] (0,0) -- (315:\outerRadius);
	\node at (0:\middleRadius) {\scriptsize{$g$}};
	\node at (90:\middleRadius) {\scriptsize{$g$}};
	\node at (180:\middleRadius) {\scriptsize{$g$}};
	\node at (270:\middleRadius) {\scriptsize{$g$}};
	\ncircle{unshaded}{(45:\innerRadius)}{.4}{90}{{\scriptsize{$V_{g^{-1}}$}}}
	\ncircle{unshaded}{(135:\innerRadius)}{.4}{90}{{\scriptsize{$V_{g^{-1}}^*$}}}
	\ncircle{unshaded}{(225:\innerRadius)}{.4}{270}{{\scriptsize{$V_{g^{-1}}$}}}
	\ncircle{unshaded}{(315:\innerRadius)}{.4}{270}{{\scriptsize{$V_{g^{-1}}^*$}}}
	\ncircle{unshaded}{(0,0)}{.27}{180}{$x$}
\end{tikzpicture}
\right)
\text{ and }
\\
\Psi_g(y) &=
\Delta_+\left(
\begin{tikzpicture}[baseline = -.1cm]
	\pgfmathsetmacro{\innerRadius}{1.2}
	\pgfmathsetmacro{\middleRadius}{1.4}
	\pgfmathsetmacro{\outerRadius}{1.8}
	\fill[shaded] (-30:\innerRadius)  arc (-30:30:\innerRadius) -- (30:\outerRadius) arc (30:-30:\outerRadius) -- (-30:\innerRadius);
	\fill[shaded] (90:\innerRadius)  arc (90:150:\innerRadius) -- (150:\outerRadius) arc (150:90:\outerRadius) -- (90:\innerRadius);
	\fill[shaded] (210:\innerRadius)  arc (210:270:\innerRadius) -- (270:\outerRadius) arc (270:210:\outerRadius) -- (210:\innerRadius);
	\fill[shaded] (0,0) -- (210:\innerRadius) arc (210:150:\innerRadius) -- (0,0);
	\fill[shaded] (0,0) -- (30:\innerRadius) arc (30:90:\innerRadius) -- (0,0);
	\fill[shaded] (0,0) -- (270:\innerRadius) arc (270:330:\innerRadius) -- (0,0);
	\draw[mid<] (210:\innerRadius) arc (210:150:\innerRadius);
	\draw[mid>] (150:\innerRadius) arc (150:90:\innerRadius);
	\draw[mid<] (90:\innerRadius) arc (90:30:\innerRadius);
	\draw[mid>] (30:\innerRadius) arc (30:-30:\innerRadius);
	\draw[mid<] (330:\innerRadius) arc (330:270:\innerRadius);
	\draw[mid>] (270:\innerRadius) arc (270:210:\innerRadius);
	\draw[thick, rho] (0,0) -- (30:\outerRadius);
	\draw[thick, rho] (0,0) -- (90:\outerRadius);
	\draw[thick, rho] (0,0) -- (150:\outerRadius);
	\draw[thick, rho] (0,0) -- (210:\outerRadius);
	\draw[thick, rho] (0,0) -- (270:\outerRadius);
	\draw[thick, rho] (0,0) -- (330:\outerRadius);
	\node at (0:\middleRadius) {\scriptsize{$g$}};
	\node at (60:\middleRadius) {\scriptsize{$g$}};
	\node at (120:\middleRadius) {\scriptsize{$g$}};
	\node at (180:\middleRadius) {\scriptsize{$g$}};
	\node at (240:\middleRadius) {\scriptsize{$g$}};
	\node at (300:\middleRadius) {\scriptsize{$g$}};
	\ncircle{unshaded}{(150:\innerRadius)}{.4}{210}{{\scriptsize{$V_{g^{-1}}$}}}
	\ncircle{unshaded}{(90:\innerRadius)}{.4}{30}{{\scriptsize{$V_{g^{-1}}^*$}}}
	\ncircle{unshaded}{(30:\innerRadius)}{.4}{90}{{\scriptsize{$V_{g^{-1}}$}}}
	\ncircle{unshaded}{(330:\innerRadius)}{.4}{270}{{\scriptsize{$V_{g^{-1}}^*$}}}
	\ncircle{unshaded}{(270:\innerRadius)}{.4}{330}{{\scriptsize{$V_{g^{-1}}$}}}
	\ncircle{unshaded}{(210:\innerRadius)}{.4}{150}{{\scriptsize{$V_{g^{-1}}^*$}}}
	\ncircle{unshaded}{(0,0)}{.27}{180}{$y$}
\end{tikzpicture}
\right)
= 
\Delta_+\left(
\begin{tikzpicture}[baseline = -.1cm]
	\pgfmathsetmacro{\innerRadius}{1.1}
	\pgfmathsetmacro{\middleRadius}{1.3}
	\pgfmathsetmacro{\outerRadius}{1.7}
	\fill[shaded] (-30:\innerRadius)  arc (-30:30:\innerRadius) -- (30:\outerRadius) arc (30:-30:\outerRadius) -- (-30:\innerRadius);
	\fill[shaded] (90:\innerRadius)  arc (90:150:\innerRadius) -- (150:\outerRadius) arc (150:90:\outerRadius) -- (90:\innerRadius);
	\fill[shaded] (210:\innerRadius)  arc (210:270:\innerRadius) -- (270:\outerRadius) arc (270:210:\outerRadius) -- (210:\innerRadius);
	\fill[shaded] (0,0) -- (210:\innerRadius) arc (210:150:\innerRadius) -- (0,0);
	\fill[shaded] (0,0) -- (30:\innerRadius) arc (30:90:\innerRadius) -- (0,0);
	\fill[shaded] (0,0) -- (270:\innerRadius) arc (270:330:\innerRadius) -- (0,0);
	\draw[mid<] (210:\innerRadius) arc (210:150:\innerRadius);
	\draw[mid>] (150:\innerRadius) arc (150:90:\innerRadius);
	\draw[mid<] (90:\innerRadius) arc (90:30:\innerRadius);
	\draw[mid>] (30:\innerRadius) arc (30:-30:\innerRadius);
	\draw[mid<] (330:\innerRadius) arc (330:270:\innerRadius);
	\draw[mid>] (270:\innerRadius) arc (270:210:\innerRadius);
	\draw[thick, rho] (0,0) -- (30:\outerRadius);
	\draw[thick, rho] (0,0) -- (90:\outerRadius);
	\draw[thick, rho] (0,0) -- (150:\outerRadius);
	\draw[thick, rho] (0,0) -- (210:\outerRadius);
	\draw[thick, rho] (0,0) -- (270:\outerRadius);
	\draw[thick, rho] (0,0) -- (330:\outerRadius);
	\node at (0:\middleRadius) {\scriptsize{$g$}};
	\node at (60:\middleRadius) {\scriptsize{$g$}};
	\node at (120:\middleRadius) {\scriptsize{$g$}};
	\node at (180:\middleRadius) {\scriptsize{$g$}};
	\node at (240:\middleRadius) {\scriptsize{$g$}};
	\node at (300:\middleRadius) {\scriptsize{$g$}};
	\ncircle{unshaded}{(150:\innerRadius)}{.27}{90}{{\scriptsize{$V_{g}^*$}}}
	\ncircle{unshaded}{(90:\innerRadius)}{.27}{150}{{\scriptsize{$V_{g}$}}}
	\ncircle{unshaded}{(30:\innerRadius)}{.27}{330}{{\scriptsize{$V_{g}^*$}}}
	\ncircle{unshaded}{(330:\innerRadius)}{.27}{30}{{\scriptsize{$V_{g}$}}}
	\ncircle{unshaded}{(270:\innerRadius)}{.27}{210}{{\scriptsize{$V_{g}^*$}}}
	\ncircle{unshaded}{(210:\innerRadius)}{.27}{270}{{\scriptsize{$V_{g}$}}}
	\ncircle{unshaded}{(0,0)}{.27}{180}{$y$}
\end{tikzpicture}
\right)
\,.
\end{align*}
\end{ex}

\begin{lem}\label{lem:1ClickFixed}
On $\cS_-$, $\Psi_g=\cF_{\cS_\bullet}^{-1}\circ \Psi_g\circ \cF_{\cS_\bullet}$
\end{lem}
\begin{proof}
Identifying $\cS_{n,\pm}=\cR_n$, we have $\cF_{\cS_\bullet}=\cF_{\cR_\bullet}$. 
Since $\Delta$ is a symmetric self-duality, $\Delta_+ \circ \cF_{\cS_\bullet}=\cF_{\cS_\bullet}\circ \Delta_-$.
By Corollary \ref{cor:1ClickProblem}, $\cF_{\cR_\bullet} \circ \Phi_{g^{-1}}=\Phi_g \circ \cF_{\cR_\bullet}$.
Combining these, we have
\begin{align*}
\Psi_g 
&= \Delta_+ \circ \Phi_{g^{-1}} 
= (\Delta_+ \circ \cF_{\cR_\bullet}^{-1})\circ (\cF_{\cR_\bullet} \circ \Phi_{g^{-1}})
=\cF_{\cR_\bullet}^{-1}\circ (\Delta_- \circ \Phi_g) \circ \cF_{\cR_\bullet}
=\cF_{\cS_\bullet}^{-1} \circ \Psi_g \circ \cF_{\cS_\bullet}.
\qedhere
\end{align*}
\end{proof}

\begin{prop}
The map $\Psi_g$ is a shaded planar algebra automorphism.
\end{prop}
\begin{proof}
We need to show that $\Psi_g$ commutes with a set of generating tangles for the planar operad. 
We can use the graded multiplication operator and the annular tangles, which are clearly generated by adding cups and caps and the 1-click rotation operator.

The fact that $\Psi_g$ is compatible with the 1-click rotation is exactly Lemma \ref{lem:1ClickFixed}.
We see that $\Psi_g$ is compatible with cups and caps by the Reidemeister II relation from Proposition \ref{prop:ReidemeisterII}.
Finally, $\Psi_g$ is compatible with the graded multiplication operator by Proposition \ref{prop:GradedMultiplication}.
\end{proof}

%%%%%%%%%%%%%%%%%%%%%%%%%%%%%%%%%%%%%%%%%%%%%%%%%%%%%%%%%%%%%%
\subsection{The $G$-action on $\cS_\bullet$}

We prove a few lemmas to calculate constants, after which we will see $\Psi$ gives an action of $G$ on $\cS_\bullet$.

\begin{cor}\label{cor:NormOfTrivalentVertex}
$
\begin{tikzpicture}[baseline = -.1cm]
	\draw[thick, theta] (-.5,-.6) -- (-.5,.6);
	\draw[thick, theta] (.5,-.6) -- (.5,.6);
	\draw[thick, theta] (-.3,-.6) -- (-.3,0);
	\draw[thick, theta] (.3,-.6) -- (.3,0);
	\nbox{unshaded}{(-.4,0)}{.3}{0}{0}{$g$}
	\nbox{unshaded}{(.4,0)}{.3}{0}{0}{$h$}
	\draw[thick, theta] (-.3,.3) arc (180:0:.3cm);
\end{tikzpicture}
=
\begin{tikzpicture}[baseline = .6cm]
	\draw[thick, theta] (-.5,-.6) -- (-.5,0);
	\draw[thick, theta] (.5,-.6) -- (.5,0);
	\draw[thick, theta] (-.3,-.6) -- (-.3,0);
	\draw[thick, theta] (.3,-.6) -- (.3,0);
	\draw[thick, theta] (-.5,.3) -- (-.1,.9);
	\draw[thick, theta] (.5,.3) -- (.1,.9);
	\draw[thick, theta] (-.1,1.5) -- (-.1,1.8);
	\draw[thick, theta] (.1,1.5) -- (.1,1.8);
	\nbox{unshaded}{(-.4,0)}{.3}{0}{0}{$g$}
	\nbox{unshaded}{(.4,0)}{.3}{0}{0}{$h$}
	\nbox{unshaded}{(0,1.2)}{.3}{0}{0}{$gh$}
	\draw[thick, theta] (-.3,.3) arc (180:0:.3cm);
\end{tikzpicture}
$
and
$
\begin{tikzpicture}[baseline = -.1cm]
	\draw[thick, theta] (-.5,-.6) -- (-.5,.6);
	\draw[thick, theta] (.5,-.6) -- (.5,.6);
	\nbox{unshaded}{(-.4,0)}{.3}{0}{0}{$g$}
	\nbox{unshaded}{(.4,0)}{.3}{0}{0}{$h$}
	\draw[thick, theta] (-.3,.3) arc (180:0:.3cm);
	\draw[thick, theta] (-.3,-.3) arc (-180:0:.3cm);
\end{tikzpicture}
=
\begin{tikzpicture}[baseline = -.1cm]
	\draw[thick, theta] (-.5,.3) -- (-.1,.9);
	\draw[thick, theta] (.5,.3) -- (.1,.9);
	\draw[thick, theta] (-.1,1.5) -- (-.1,1.8);
	\draw[thick, theta] (.1,1.5) -- (.1,1.8);
	\draw[thick, theta] (-.5,-.3) -- (-.1,-.9);
	\draw[thick, theta] (.5,-.3) -- (.1,-.9);
	\draw[thick, theta] (-.1,-1.5) -- (-.1,-1.8);
	\draw[thick, theta] (.1,-1.5) -- (.1,-1.8);
	\nbox{unshaded}{(-.4,0)}{.3}{0}{0}{$g$}
	\nbox{unshaded}{(.4,0)}{.3}{0}{0}{$h$}
	\nbox{unshaded}{(0,1.2)}{.3}{0}{0}{$gh$}
	\nbox{unshaded}{(0,-1.2)}{.3}{0}{0}{$gh$}
	\draw[thick, theta] (-.3,.3) arc (180:0:.3cm);
	\draw[thick, theta] (-.3,-.3) arc (-180:0:.3cm);
\end{tikzpicture}
=
\displaystyle
\frac{1}{[4]}
\begin{tikzpicture}[baseline = -.1cm]
	\draw[thick, theta] (-.1,-.6) -- (-.1,.6);
	\draw[thick, theta] (.1,-.6) -- (.1,.6);
	\nbox{unshaded}{(0,0)}{.3}{0}{0}{$gh$}
\end{tikzpicture}
$\,.
\end{cor}
\begin{proof}
The first equation follows by taking the norm squared of the difference and applying Lemma \ref{lem:GCoproduct}.
The second equation then follows immediately.
\end{proof}

\begin{defn}
For $h,g\in G$, let the $(g,h,hg)$-trivalent vertex be given by
$$
Y_{g,h}=
\begin{tikzpicture}[baseline = -.1cm]
	\coordinate (aa) at (0,0);
	\filldraw (aa) circle (.05cm);
	\draw[mid>] (aa) -- (0,.6);
	\draw[mid<] (aa) -- (-.4,-.4);
	\draw[mid<] (aa) -- (.4,-.4);
	\node at (-.5,-.2) {\scriptsize{$g$}};
	\node at (.5,-.2) {\scriptsize{$h$}};
	\node at (-.3,.4) {\scriptsize{$gh$}};
\end{tikzpicture}
=
[4]^{1/2}
\begin{tikzpicture}[baseline = .6cm]
	\draw[thick, theta] (-.5,-.6) -- (-.5,0);
	\draw[thick, theta] (.5,-.6) -- (.5,0);
	\draw[thick, theta] (-.3,-.6) -- (-.3,0);
	\draw[thick, theta] (.3,-.6) -- (.3,0);
	\draw[thick, theta] (-.5,.3) -- (-.1,.9);
	\draw[thick, theta] (.5,.3) -- (.1,.9);
	\draw[thick, theta] (-.1,1.5) -- (-.1,1.8);
	\draw[thick, theta] (.1,1.5) -- (.1,1.8);
	\nbox{unshaded}{(-.4,0)}{.3}{0}{0}{$g$}
	\nbox{unshaded}{(.4,0)}{.3}{0}{0}{$h$}
	\nbox{unshaded}{(0,1.2)}{.3}{0}{0}{$gh$}
	\draw[thick, theta] (-.3,.3) arc (180:0:.3cm);
\end{tikzpicture},
$$
so that $Y_{g,h}^*Y_{g,h} = gh$ by Corollary \ref{cor:NormOfTrivalentVertex}.
Since $Y_{g,h}Y_{g,h}^*$ also has trace 1, we immediately have that
\begin{equation}
\begin{tikzpicture}[baseline = -.1cm]
	\draw[mid<] (-.2,.8) -- (-.2,-.8);
	\draw[mid<] (.2,.8) -- (.2,-.8);
	\node at (-.4,0) {\scriptsize{$g$}};
	\node at (0,0) {\scriptsize{$h$}};
\end{tikzpicture}
=
\begin{tikzpicture}[baseline = -.1cm]
	\draw[thick, theta] (-.5,-.8) -- (-.5,.8);
	\draw[thick, theta] (.5,-.8) -- (.5,.8);
	\draw[thick, theta] (-.3,-.8) -- (-.3,.8);
	\draw[thick, theta] (.3,-.8) -- (.3,.8);
	\nbox{unshaded}{(-.4,0)}{.3}{0}{0}{$g$}
	\nbox{unshaded}{(.4,0)}{.3}{0}{0}{$h$}
\end{tikzpicture}
=
[4]
\begin{tikzpicture}[baseline = -.1cm]
	\draw[thick, theta] (-.5,-1.8) -- (-.5,-1.2);
	\draw[thick, theta] (.5,-1.8) -- (.5,-1.2);
	\draw[thick, theta] (-.3,-1.8) -- (-.3,-1.2);
	\draw[thick, theta] (.3,-1.8) -- (.3,-1.2);
	\draw[thick, theta] (-.5,-.9) -- (-.1,-.3);
	\draw[thick, theta] (.5,-.9) -- (.1,.-.3);
	\draw[thick, theta] (-.5,1.8) -- (-.5,1.2);
	\draw[thick, theta] (.5,1.8) -- (.5,1.2);
	\draw[thick, theta] (-.3,1.8) -- (-.3,1.2);
	\draw[thick, theta] (.3,1.8) -- (.3,1.2);
	\draw[thick, theta] (-.5,.9) -- (-.1,.3);
	\draw[thick, theta] (.5,.9) -- (.1,.3);
	\nbox{unshaded}{(-.4,-1.2)}{.3}{0}{0}{$g$}
	\nbox{unshaded}{(.4,-1.2)}{.3}{0}{0}{$h$}
	\nbox{unshaded}{(-.4,1.2)}{.3}{0}{0}{$g$}
	\nbox{unshaded}{(.4,1.2)}{.3}{0}{0}{$h$}
	\nbox{unshaded}{(0,0)}{.3}{0}{0}{$gh$}
	\draw[thick, theta] (-.3,-.9) arc (180:0:.3cm);
	\draw[thick, theta] (-.3,.9) arc (-180:0:.3cm);
\end{tikzpicture}
=
\begin{tikzpicture}[baseline = -.1cm]
	\coordinate (aa) at (0,.4);
	\coordinate (bb) at (0,-.4);
	\filldraw (aa) circle (.05cm);
	\filldraw (bb) circle (.05cm);
	\draw[mid<] (aa) -- (bb);
	\draw[mid<] (bb) -- (-.4,-.8);
	\draw[mid<] (bb) -- (.4,-.8);
	\draw[mid>] (aa) -- (-.4,.8);
	\draw[mid>] (aa) -- (.4,.8);
	\node at (-.5,-.6) {\scriptsize{$g$}};
	\node at (.5,-.6) {\scriptsize{$h$}};
	\node at (-.5,.6) {\scriptsize{$g$}};
	\node at (.5,.6) {\scriptsize{$h$}};
	\node at (-.3,0) {\scriptsize{$gh$}};
\end{tikzpicture}
\,.
\label{rel:Zip}
\end{equation}
\end{defn}

\begin{rem}
Note that the $(g,h,gh)$-trivalent vertex $Y_{g,h}$ is not in $\cR_\bullet$.
At this point, we do not know whether $Y_{g,h}Y_{g,h}^*$ is in $\cR_\bullet$.
\end{rem}

\begin{lem}
The $G$-trivalent vertices are associative, i.e.,
$
\begin{tikzpicture}[baseline = -.1cm]
	\draw[mid>] (-.6,-.8)--(-.3,-.2);
	\draw[mid>] (0,-.8)--(-.3,-.2);
	\draw[mid>] (-.3,-.2)--(0,.4);
	\draw[mid>] (.6,-.8)--(0,.4);
	\draw[mid>] (0,.4)--(0,.8);
	\node at (-.7,-.6) {\scriptsize{$g$}};
	\node at (.1,-.6) {\scriptsize{$h$}};
	\node at (.7,-.6) {\scriptsize{$k$}};
	\node at (-.5,0) {\scriptsize{$gh$}};
	\filldraw (-.3,-.2) circle (.05cm);
	\filldraw (0,.4) circle (.05cm);
\end{tikzpicture}
=
\begin{tikzpicture}[baseline = -.1cm]
	\draw[mid>] (.6,-.8)--(.3,-.2);
	\draw[mid>] (0,-.8)--(.3,-.2);
	\draw[mid>] (.3,-.2)--(0,.4);
	\draw[mid>] (-.6,-.8)--(0,.4);
	\draw[mid>] (0,.4)--(0,.8);
	\node at (-.7,-.6) {\scriptsize{$g$}};
	\node at (-.1,-.6) {\scriptsize{$h$}};
	\node at (.7,-.6) {\scriptsize{$k$}};
	\node at (.5,0) {\scriptsize{$hk$}};
	\filldraw (.3,-.2) circle (.05cm);
	\filldraw (0,.4) circle (.05cm);
\end{tikzpicture}
$\,.
\end{lem}
\begin{proof}
Using Corollary \ref{cor:NormOfTrivalentVertex}, we have
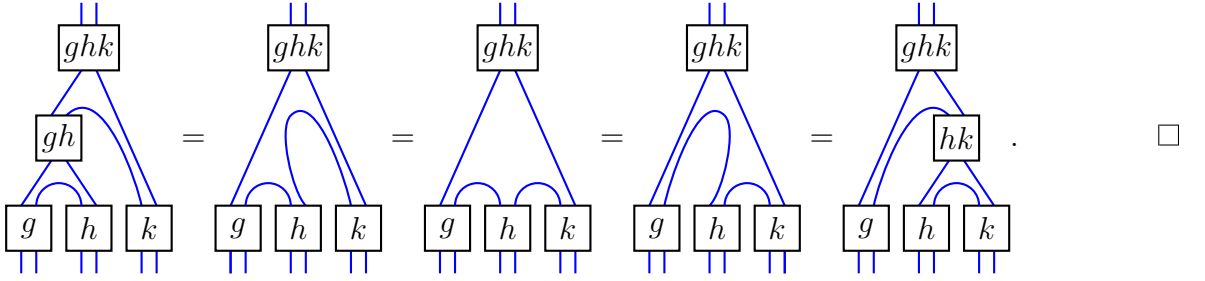
\begin{align*}
\begin{tikzpicture}[baseline = 1.1cm]
	\draw[thick, theta] (-.5,-.6) -- (-.5,0);
	\draw[thick, theta] (.5,-.6) -- (.5,0);
	\draw[thick, theta] (-.3,-.6) -- (-.3,0);
	\draw[thick, theta] (.3,-.6) -- (.3,0);
	\draw[thick, theta] (-.5,.3) -- (-.1,.9);
	\draw[thick, theta] (.5,.3) -- (.1,.9);
	\draw[thick, theta] (-.1,1.5) -- (.3,2.1) -- (.3,3);
	\draw[thick, theta] (.1,1.5)  .. controls ++(45:.7cm) and ++(100:.25cm) ..  (1.1,.3) -- (1.1,-.6);
	\draw[thick, theta] (1.3,-.6) -- (1.3,.3) -- (.5,2.1) -- (.5,3);
	\nbox{unshaded}{(-.4,0)}{.3}{0}{0}{$g$}
	\nbox{unshaded}{(.4,0)}{.3}{0}{0}{$h$}
	\nbox{unshaded}{(1.2,0)}{.3}{0}{0}{$k$}
	\nbox{unshaded}{(0,1.2)}{.3}{0}{0}{$gh$}
	\nbox{unshaded}{(.4,2.4)}{.3}{.1}{.1}{$ghk$}
	\draw[thick, theta] (-.3,.3) arc (180:0:.3cm);
\end{tikzpicture}
&=
\begin{tikzpicture}[baseline = 1.1cm]
	\draw[thick, theta] (-.5,-.6) -- (-.5,0);
	\draw[thick, theta] (.5,-.6) -- (.5,0);
	\draw[thick, theta] (-.3,-.6) -- (-.3,0);
	\draw[thick, theta] (.3,-.6) -- (.3,0);
	\draw[thick, theta] (.5,.3) .. controls ++(135:.25cm) and ++(225:.25cm) .. (.3,1.5)  .. controls ++(45:.5cm) and ++(100:.25cm) ..  (1.1,.3) -- (1.1,-.6);
	\draw[thick, theta] (-.5,-.6) -- (-.5,.3) -- (.3,2.1) -- (.3,3);
	\draw[thick, theta] (1.3,-.6) -- (1.3,.3) -- (.5,2.1) -- (.5,3);
	\nbox{unshaded}{(-.4,0)}{.3}{0}{0}{$g$}
	\nbox{unshaded}{(.4,0)}{.3}{0}{0}{$h$}
	\nbox{unshaded}{(1.2,0)}{.3}{0}{0}{$k$}
	\nbox{unshaded}{(.4,2.4)}{.3}{.1}{.1}{$ghk$}
	\draw[thick, theta] (-.3,.3) arc (180:0:.3cm);
\end{tikzpicture}
=
\begin{tikzpicture}[baseline = 1.1cm]
	\draw[thick, theta] (.5,-.6) -- (.5,0);
	\draw[thick, theta] (-.3,-.6) -- (-.3,0);
	\draw[thick, theta] (.3,-.6) -- (.3,0);
	\draw[thick, theta] (1.1,-.6) -- (1.1,0);
	\draw[thick, theta] (-.5,-.6) -- (-.5,.3) -- (.3,2.1) -- (.3,3);
	\draw[thick, theta] (1.3,-.6) -- (1.3,.3) -- (.5,2.1) -- (.5,3);
	\nbox{unshaded}{(-.4,0)}{.3}{0}{0}{$g$}
	\nbox{unshaded}{(.4,0)}{.3}{0}{0}{$h$}
	\nbox{unshaded}{(1.2,0)}{.3}{0}{0}{$k$}
	\nbox{unshaded}{(.4,2.4)}{.3}{.1}{.1}{$ghk$}
	\draw[thick, theta] (-.3,.3) arc (180:0:.3cm);
	\draw[thick, theta] (.5,.3) arc (180:0:.3cm);
\end{tikzpicture}
=
\begin{tikzpicture}[baseline = 1.1cm, xscale=-1]
	\draw[thick, theta] (-.5,-.6) -- (-.5,0);
	\draw[thick, theta] (.5,-.6) -- (.5,0);
	\draw[thick, theta] (-.3,-.6) -- (-.3,0);
	\draw[thick, theta] (.3,-.6) -- (.3,0);
	\draw[thick, theta] (.5,.3) .. controls ++(135:.25cm) and ++(225:.25cm) .. (.3,1.5)  .. controls ++(45:.5cm) and ++(100:.25cm) ..  (1.1,.3) -- (1.1,-.6);
	\draw[thick, theta] (-.5,-.6) -- (-.5,.3) -- (.3,2.1) -- (.3,3);
	\draw[thick, theta] (1.3,-.6) -- (1.3,.3) -- (.5,2.1) -- (.5,3);
	\nbox{unshaded}{(-.4,0)}{.3}{0}{0}{$k$}
	\nbox{unshaded}{(.4,0)}{.3}{0}{0}{$h$}
	\nbox{unshaded}{(1.2,0)}{.3}{0}{0}{$g$}
	\nbox{unshaded}{(.4,2.4)}{.3}{.1}{.1}{$ghk$}
	\draw[thick, theta] (-.3,.3) arc (180:0:.3cm);
\end{tikzpicture}
=
\begin{tikzpicture}[baseline = 1.1cm, xscale=-1]
	\draw[thick, theta] (-.5,-.6) -- (-.5,0);
	\draw[thick, theta] (.5,-.6) -- (.5,0);
	\draw[thick, theta] (-.3,-.6) -- (-.3,0);
	\draw[thick, theta] (.3,-.6) -- (.3,0);
	\draw[thick, theta] (-.5,.3) -- (-.1,.9);
	\draw[thick, theta] (.5,.3) -- (.1,.9);
	\draw[thick, theta] (-.1,1.5) -- (.3,2.1) -- (.3,3);
	\draw[thick, theta] (.1,1.5)  .. controls ++(45:.7cm) and ++(100:.25cm) ..  (1.1,.3) -- (1.1,-.6);
	\draw[thick, theta] (1.3,-.6) -- (1.3,.3) -- (.5,2.1) -- (.5,3);
	\nbox{unshaded}{(-.4,0)}{.3}{0}{0}{$k$}
	\nbox{unshaded}{(.4,0)}{.3}{0}{0}{$h$}
	\nbox{unshaded}{(1.2,0)}{.3}{0}{0}{$g$}
	\nbox{unshaded}{(0,1.2)}{.3}{0}{0}{$hk$}
	\nbox{unshaded}{(.4,2.4)}{.3}{.1}{.1}{$ghk$}
	\draw[thick, theta] (-.3,.3) arc (180:0:.3cm);
\end{tikzpicture}.
\qedhere
\end{align*}
\end{proof}

\begin{rem}
Suppose we did not assume that $\cR_\bullet$ was the reduced subfactor planar algebra of $\cP_\bullet$ at $\rho=\jw{2}$, and instead we started with a factor planar algebra with principal graph $\Lambda$ from Definition \ref{defn:ReducedPlanarAlgebra}.
In this case, we might have that the $G$-trivalent vertices are not associative.
Rather, there may be a non-trivial 3-cocycle giving a non-trivial associator for $G$.

In fact, there are such examples for $\bbZ/3\bbZ$ giving `twisted' Haagerup categories due to Ostrik \cite[Proposition 7.7, and the following paragraph]{1501.06869}.
\end{rem}

Recall from Lemma \ref{lem:OneClickRotationCompatible} that there is a distinguished 1-cochain $\theta\in C^1(G,U(1))$.

\begin{cor}\label{cor:Cocycle}
For all $g,h\in G$, there is a scalar $\mu_{g,h}\in U(1)$ such that
$$
\begin{tikzpicture}[baseline = -.1cm]
	\coordinate (aa) at (.8,.6);
	\coordinate (bb) at (-.8,-.6);
	\filldraw (aa) circle (.05cm);
	\filldraw (bb) circle (.05cm);
	\draw[thick, rho] (-.6,1) -- (.6,-1);
	\draw[mid<] (0,.5) -- (aa);
	\draw[mid<] (.4,-.3) -- (aa);
	\draw[mid<] (aa) -- (1.4,1);
	\draw[mid<] (bb) -- (-1.4,-1);
	\draw[mid<] (0,-.5) -- (bb);
	\draw[mid<] (-.4,.3) -- (bb);
	\node at (-1.2,-.6) {\scriptsize{$gh$}};
	\node at (1.2,.6) {\scriptsize{$gh$}};
	\node at (-.8,0) {\scriptsize{$g$}};
	\node at (.8,0) {\scriptsize{$h$}};
	\node at (-.3,-.8) {\scriptsize{$h$}};
	\node at (.3,.8) {\scriptsize{$g$}};
	\ncircle{unshaded}{(-.3,.5)}{.35}{180}{$V_g$}
	\ncircle{unshaded}{(.3,-.5)}{.35}{150}{$V_h$}
\end{tikzpicture}
=
\mu_{g,h}
\begin{tikzpicture}[baseline = -.1cm]
	\draw[thick, rho] (.4,-.8) -- (-.4,.8);
	\draw[mid>] (-.4,-.8) -- (-.15,-.3);
	\node at (.1,.7) {\scriptsize{$gh$}};
	\draw[mid<] (.15,.3) -- (.4,.8);
	\node at (-.6,-.6) {\scriptsize{$gh$}};
	\ncircle{unshaded}{(0,0)}{.4}{180}{$V_{gh}$}
\end{tikzpicture}\,.
$$
Moreover, $\mu\in Z^2(G,U(1))$, and $\mu_{g,h}\mu_{g^{-1},h^{-1}}=[(d\theta)(g,h)]^{-1}$.
\end{cor}
\begin{proof}
Take the norm squared of each diagram, unzip the trivalent vertices, and use the Reidemeister II relation from Proposition \ref{prop:ReidemeisterII} to get that both closed diagrams equal $[3]$.
Since there is only one map up to scaling from $gh\otimes \rho$ to $\rho\otimes (gh)^{-1}$, both sides must be equal up to a phase, denoted $\mu_{g,h}$.

A straightforward calculation again by unzipping and using the Reidemeister II relation from Proposition \ref{prop:ReidemeisterII} shows that for $g,h,k\in G$, $\mu_{g,hk}\mu_{h,k}=\mu_{gh,k}\mu_{g,h}$, i.e., $\mu\in Z^2(G,U(1))$. 

For the final claim, we first look at
$$
\cF\left(\left(
\begin{tikzpicture}[baseline = -.1cm]
	\coordinate (aa) at (.8,.6);
	\coordinate (bb) at (-.8,-.6);
	\filldraw (aa) circle (.05cm);
	\filldraw (bb) circle (.05cm);
	\draw[thick, rho] (-.6,1) -- (.6,-1);
	\draw[mid<] (0,.5) -- (aa);
	\draw[mid<] (.4,-.3) -- (aa);
	\draw[mid<] (aa) -- (1.4,1);
	\draw[mid<] (bb) -- (-1.4,-1);
	\draw[mid<] (0,-.5) -- (bb);
	\draw[mid<] (-.4,.3) -- (bb);
	\node at (-1.2,-.6) {\scriptsize{$gh$}};
	\node at (1.2,.6) {\scriptsize{$gh$}};
	\node at (-.8,0) {\scriptsize{$g$}};
	\node at (.8,0) {\scriptsize{$h$}};
	\node at (-.3,-.8) {\scriptsize{$h$}};
	\node at (.3,.8) {\scriptsize{$g$}};
	\ncircle{unshaded}{(-.3,.5)}{.35}{180}{$V_g$}
	\ncircle{unshaded}{(.3,-.5)}{.35}{150}{$V_h$}
\end{tikzpicture}
\right)^*\right)
=
\begin{tikzpicture}[baseline = -.1cm]
	\coordinate (aa) at (.8,.6);
	\coordinate (bb) at (-.8,-.6);
	\filldraw (aa) circle (.05cm);
	\filldraw (bb) circle (.05cm);
	\draw[thick, rho] (-.6,1) -- (.6,-1);
	\draw[mid>] (0,.5) -- (aa);
	\draw[mid>] (.4,-.3) -- (aa);
	\draw[mid>] (aa) -- (1.4,1);
	\draw[mid>] (bb) -- (-1.4,-1);
	\draw[mid>] (0,-.5) -- (bb);
	\draw[mid>] (-.4,.3) -- (bb);
	\node at (-1.2,-.6) {\scriptsize{$gh$}};
	\node at (1.2,.6) {\scriptsize{$gh$}};
	\node at (-.8,0) {\scriptsize{$g$}};
	\node at (.8,0) {\scriptsize{$h$}};
	\node at (-.3,-.8) {\scriptsize{$h$}};
	\node at (.3,.8) {\scriptsize{$g$}};
	\ncircle{unshaded}{(-.3,.5)}{.37}{90}{{\scriptsize{$V_{g}^*$}}}
	\ncircle{unshaded}{(.3,-.5)}{.37}{90}{{\scriptsize{$V_{h}^*$}}}
\end{tikzpicture}
=
\theta_g
\theta_h
\begin{tikzpicture}[baseline = -.1cm]
	\coordinate (aa) at (.8,.6);
	\coordinate (bb) at (-.8,-.6);
	\filldraw (aa) circle (.05cm);
	\filldraw (bb) circle (.05cm);
	\draw[thick, rho] (-.6,1) -- (.6,-1);
	\draw[mid>] (0,.5) -- (aa);
	\draw[mid>] (.4,-.3) -- (aa);
	\draw[mid>] (aa) -- (1.4,1);
	\draw[mid>] (bb) -- (-1.4,-1);
	\draw[mid>] (0,-.5) -- (bb);
	\draw[mid>] (-.4,.3) -- (bb);
	\node at (-1.2,-.6) {\scriptsize{$gh$}};
	\node at (1.2,.6) {\scriptsize{$gh$}};
	\node at (-.8,0) {\scriptsize{$g$}};
	\node at (.8,0) {\scriptsize{$h$}};
	\node at (-.3,-.8) {\scriptsize{$h$}};
	\node at (.3,.8) {\scriptsize{$g$}};
	\ncircle{unshaded}{(-.3,.5)}{.35}{180}{\scriptsize{$V_{g^{-1}}$}}
	\ncircle{unshaded}{(.3,-.5)}{.35}{150}{\scriptsize{$V_{h^{-1}}$}}
\end{tikzpicture}.
$$
This must be equal to 
$$
\mu_{g,h}^{-1}
\cF\left(\left(
\begin{tikzpicture}[baseline = -.1cm]
	\draw[thick, rho] (.4,-.8) -- (-.4,.8);
	\draw[mid>] (-.4,-.8) -- (-.15,-.3);
	\node at (.1,.7) {\scriptsize{$gh$}};
	\draw[mid<] (.15,.3) -- (.4,.8);
	\node at (-.6,-.6) {\scriptsize{$gh$}};
	\ncircle{unshaded}{(0,0)}{.4}{180}{$V_{gh}$}
\end{tikzpicture}\,\,
\right)^*\right)
=
\mu_{g,h}^{-1}
\begin{tikzpicture}[baseline = -.1cm]
	\draw[thick, rho] (.4,-.8) -- (-.4,.8);
	\draw[mid<] (-.4,-.8) -- (-.15,-.3);
	\node at (.1,.7) {\scriptsize{$gh$}};
	\draw[mid>] (.15,.3) -- (.4,.8);
	\node at (-.6,-.6) {\scriptsize{$gh$}};
	\ncircle{unshaded}{(0,0)}{.4}{180}{\scriptsize{$V_{gh}^*$}}
\end{tikzpicture}\
=
\mu_{g,h}^{-1}\theta_{gh}
\begin{tikzpicture}[baseline = -.1cm]
	\draw[thick, rho] (.4,-.8) -- (-.4,.8);
	\draw[mid<] (-.4,-.8) -- (-.2,-.4);
	\node at (.1,.7) {\scriptsize{$gh$}};
	\draw[mid>] (.2,.4) -- (.4,.8);
	\node at (-.6,-.6) {\scriptsize{$gh$}};
	\ncircle{unshaded}{(0,0)}{.5}{180}{\scriptsize{$V_{(gh)^{-1}}$}}
\end{tikzpicture}\,.
$$
This means that we must have $\mu_{g,h}\mu_{g^{-1},h^{-1}} = (\theta_g\theta_{gh}^{-1}\theta_h)^{-1}=[(d\theta)(g,h)]^{-1}$.
\end{proof}

\begin{rem}
The significance of the final formula in Corollary \ref{cor:Cocycle} is that there is a strong relation between the structure constants of $\cS_\bullet$.
\end{rem}

\begin{cor}
On $\cR_n$, we have $\Phi_g\circ \Phi_h = \Phi_{gh}$, and similarly for the $\Psi_g$'s on $\cS_{n,\pm}$.
\end{cor}
\begin{proof}
Given $x\in \cS_{n,+}$, start with the diagram for $\Phi_g\circ \Phi_h(x)$, use Relation \eqref{rel:Zip}, and apply Corollary \ref{cor:Cocycle}.
We will get alternating contributions of $\mu_{g,h}$ and $\overline{\mu_{g,h}}=\mu_{g,h}^{-1}$, which cancel, leaving us with the diagram for $\Phi_{gh}(x)$.
The proof for $y\in \cS_{n,-}$ is similar.
\end{proof}

%%%%%%%%%%%%%%%%%%%%%%%%%%%%%%%%%%%%%%%%%%%%%%%%%%%%%%%%%%%%%%
\subsection{Bases for $\cS_{3,+}$}

We now define some distinguished elements for $\cS_{3,+}$.
When $|G|$ is odd, we show these elements form a basis for $\cS_{3,+}$.
We then prove Theorem \ref{thm:YB} and Conjecture \ref{conj:GeneratedByTwoBoxes} in the case $|G|$ is odd.

We use the notation 
$
p_\emptyset = \rhoE = \displaystyle\frac{1}{[3]}
\begin{tikzpicture}[baseline = -.1cm]
	\nbox{unshaded}{(0,0)}{.4}{0}{0}{}
	\draw[thick, rho] (-.2,.4) arc (-180:0:.2);
	\draw[thick, rho] (-.2,-.4) arc (180:0:.2);	
\end{tikzpicture}
$\,.

\begin{defn}
For $i,j\in G\cup \{\emptyset\}$ and $g,h,k\in G$, define
$$
\alpha_{i,j}=
\frac{1}{\Tr(p_i)^{1/2}\Tr(p_j)^{1/2}}
\begin{tikzpicture}[baseline=-.1cm]
	\pgfmathsetmacro{\height}{.8};
	\coordinate (a) at ($ 3/4*(0,-\height) $);
	\coordinate (b) at ($ 3/4*(0,\height) $);

	\fill[shaded] ($(a)+3/4*(0,-\height)+(-.15,0)$) -- ($(b)+3/4*(0,\height)+(-.15,0)$) -- ($(b)+3/4*(0,\height) + (.15,0)$) -- ($(b) + (.15,-.3)$) arc (-180:0:.15cm) -- ($(b)+3/4*(0,\height) + (.45,0)$) -- ($(b)+3/4*(0,\height) + (.75,0)$) -- ($(a)+3/4*(0,-\height) + (.75,0)$) -- ($(a)+3/4*(0,-\height) + (.45,0)$)  -- ($(a)+(.45,.3)$) arc (0:180:.15cm) -- ($(a)+3/4*(0,-\height) + (.15,0)$);

	\draw[thick, rho] ($(a)+(-.15,0)+3/4*(0,-\height)$) -- ($(b)+(-.15,0)+3/4*(0,\height)$);
	\draw[thick, rho] ($(a)+(.15,0)+3/4*(0,-\height)$)--($(a)+(.15,.3)$) arc (180:0:.15cm) -- ($(a)+(.45,0)+3/4*(0,-\height)$);
	\draw[thick, rho] ($(b)+(.15,0)+3/4*(0,\height)$)--($(b)+(.15,-.3)$) arc (-180:0:.15cm) -- ($(b)+(.45,0)+3/4*(0,\height)$);
	\nbox{unshaded}{(b)}{.3}{0}{0}{$p_j$}
	\nbox{unshaded}{(a)}{.3}{0}{0}{$p_i$}
\end{tikzpicture}
\text{ and }
\beta_{h,k,\ell}=
\begin{tikzpicture}[baseline=-.1cm]
	\pgfmathsetmacro{\width}{.8};
	\pgfmathsetmacro{\height}{.8};
	\coordinate (a) at ($ 3/4*(0,-\height) $);
	\coordinate (b) at ($ 3/4*(0,\height) $);
	\coordinate (c) at (\width,0);

	\fill[shaded] ($(a)+3/4*(0,-\height)$) -- ($(b)+3/4*(0,\height)$) -- ($(b)+3/4*(0,\height) + 1/2*(\width,0)$) -- (b) -- (c) -- (a) -- ($(a)+3/4*(0,-\height) + 1/2*(\width,0)$);
	\fill[shaded] ($(a)+3/4*(0,-\height)+(\width,0)$) -- ($(b)+3/4*(0,\height)+(\width,0)$) -- ($(b)+3/4*(0,\height)+3/2*(\width,0)$) -- ($(a)+3/4*(0,-\height)+3/2*(\width,0)$);

	\draw[thick, rho] (b) -- (c) -- (a);
	\draw[thick, rho] ($(a)+3/4*(0,-\height)$) -- ($(b)+3/4*(0,\height)$);
	\draw[thick, rho] ($(a)+3/4*(0,-\height)+(\width,0)$) -- ($(b)+3/4*(0,\height)+(\width,0)$);
	\draw[thick, rho] (a) -- ($(a)+3/4*(0,-\height) + 1/2*(\width,0)$);
	\draw[thick, rho] (b) -- ($(b)+3/4*(0,\height) + 1/2*(\width,0)$);
	\ncircle{unshaded}{(b)}{.3}{180}{$p_\ell$}
	\ncircle{unshaded}{(c)}{.3}{115}{$p_k$}
	\ncircle{unshaded}{(a)}{.3}{180}{$p_h$}
\end{tikzpicture}\,.
$$
\end{defn}

The following facts are straightforward.

\begin{facts}\label{facts:S3Basis}
\mbox{}
\begin{enumerate}[(1)]
\item
The elements $\{\alpha_{i,j}\}$ form a system of matrix units for the copy of $M_{|G|+1}(\bbC)$ corresponding to $\cI_{3,+}=\cS_{2,+}e_2\cS_{2,+}$.
\item
The inner product $\langle \beta_{g,h,k},\beta_{g',h',k'}\rangle$ is zero unless $g=g'$, $h=h'$, and $k=k'$.
\item
For all $g\in G$, we have $\Psi_g(\alpha_{i,j})=\alpha_{g^2i,g^2j}$, where we define $g^2\emptyset=\emptyset$.
\item
For all $g\in G$, we have $\Psi_g(\beta_{h,k,\ell})=\beta_{g^2h,g^2k,g^2\ell}$.
\end{enumerate}
\end{facts}

\begin{lem}
If $P_{\cI_{3,+}}$ is the orthogonal projection onto $\cI_{3,+}$ in $\cS_{3,+}$, then $P_{\cI_{3,+}}(\beta_{h,k,\ell}) = c_{h,k,\ell} \alpha_{h,\ell}$ where
$$
c_{h,k,\ell} =
\frac{1}{[3]}\,
\begin{tikzpicture}[baseline=-.1cm]
	\pgfmathsetmacro{\width}{.8};
	\pgfmathsetmacro{\height}{.8};
	\pgfmathsetmacro{\boxSize}{.3};
	\coordinate (a) at (0,-\height);
	\coordinate (b) at (0,\height);
	\coordinate (c) at ($ (\width,0) + 3/4*(\boxSize,0)$);

	\draw[thick, rho] (b) -- (c) -- (a) -- (b);
	\draw[thick, rho] (b) arc (0:180:{1/2*\width}) -- ($(a)-(\width,0)$) arc (-180:0:{1/2*\width});
	\draw[thick, rho] ($ (b) + (\boxSize,0) $) arc (90:-90:{\height});
	\ncircle{unshaded}{(b)}{\boxSize}{180}{$p_\ell$}
	\ncircle{unshaded}{(c)}{\boxSize}{115}{$p_k$}
	\ncircle{unshaded}{(a)}{\boxSize}{180}{$p_h$}
\end{tikzpicture}\,.
$$
\end{lem}
\begin{proof}
Using (1) from Facts \ref{facts:S3Basis}, we see that the projection of $\beta_{h,k,\ell}$ onto $\cI_{3,+}$ is given by 
\begin{align*}
P_{I_{3,+}}(\beta_{h,k,\ell})
&=
\sum_{i,j}
\frac{\langle \beta_{h,k,\ell}, \alpha_{i,j}\rangle}{\|\alpha_{i,j}\|_2^2}\alpha_{i,j}
=
\sum_{i,j}
\frac{\Tr(\beta_{h,k,\ell}\alpha_{i,j}^*)}{\Tr(\alpha_{i,j}\alpha_{i,j}^*)}\alpha_{i,j}
=
\frac{\Tr(\beta_{h,k,\ell}\alpha_{h,\ell}^*)}{\Tr(\alpha_{h,\ell}\alpha_{h,\ell}^*)}\alpha_{h,\ell}
=
c_{h,k,\ell}\alpha_{h,\ell}.
\qedhere
\end{align*} 
\end{proof}

\begin{lem}\label{lem:Nonzero}
If $k$ has a square root in $G$, then 
$$
c_{h,k,\ell}
=
\frac{1}{[3]}\,
\begin{tikzpicture}[baseline=-.1cm]
	\pgfmathsetmacro{\width}{.8};
	\pgfmathsetmacro{\height}{.8};
	\pgfmathsetmacro{\boxSize}{.3};
	\coordinate (a) at (0,-\height);
	\coordinate (b) at (0,\height);
	\coordinate (c) at ($ (\width,0) + 3/4*(\boxSize,0)$);

	\draw[thick, rho] (b) -- (c) -- (a) -- (b);
	\draw[thick, rho] (b) arc (0:180:{1/2*\width}) -- ($(a)-(\width,0)$) arc (-180:0:{1/2*\width});
	\draw[thick, rho] ($ (b) + (\boxSize,0) $) arc (90:-90:{\height});
	\ncircle{unshaded}{(b)}{\boxSize}{180}{$p_\ell$}
	\ncircle{unshaded}{(c)}{\boxSize}{115}{$p_k$}
	\ncircle{unshaded}{(a)}{\boxSize}{180}{$p_h$}
\end{tikzpicture}
=
\frac{1}{[3]}\,
\begin{tikzpicture}[baseline=-.1cm]
	\pgfmathsetmacro{\width}{.8};
	\pgfmathsetmacro{\height}{.8};
	\pgfmathsetmacro{\boxSize}{.5};
	\coordinate (a) at (0,-\height);
	\coordinate (b) at (0,\height);
	\coordinate (c) at ($ (\width,0) + 3/4*(\boxSize,0)$);

	\draw[thick, rho] (b) -- (c) -- (a) -- (b);
	\draw[thick, rho] (b) arc (0:180:{1/2*\width}) -- ($(a)-(\width,0)$) arc (-180:0:{1/2*\width});
	\draw[thick, rho] ($ (b) + (\boxSize,0) $) arc (90:-90:{\height});
	\ncircle{unshaded}{(b)}{\boxSize}{180}{$p_{k^{-1}\ell}$}
	\ncircle{unshaded}{(c)}{.4}{115}{$p_1$}
	\ncircle{unshaded}{(a)}{\boxSize}{180}{$p_{k^{-1}h}$}
	
\end{tikzpicture}
=
\begin{cases}
\displaystyle \frac{[3]}{([3]-1)^2}& \text{if $h,k,\ell$ are distinct}
\\
\displaystyle \frac{1}{([3]-1)^2} & \text{if any two are equal}
\\
\displaystyle \left( \frac{[3]-2}{[2]} \right)^2 &\text{if $h=k=\ell$.}
\end{cases}
$$
By symmetry, a similar statement holds if $h$ or $\ell$ has a square root in $G$.

Moreover, if any of $h,k,\ell$ have a square root in $G$, then $\gamma_{h,k,\ell}:=\beta_{h,k,\ell}-c_{h,k,\ell}\alpha_{h,\ell}\neq 0$.
\end{lem}
\begin{proof}
Using (4) of Facts \ref{facts:S3Basis} and sphericality, we have
$$
\langle \beta_{h,k,\ell}, \alpha_{h,\ell}\rangle 
= 
\|\beta_{h,k,\ell}\|^2_2
=
\|\beta_{g^2(g^{-2}h),g^2,g^2(g^{-2}\ell)l}\|^2_2
=
\|\Psi_g(\beta_{k^{-1}h,1,k^{-1}\ell})\|^2_2
=
\begin{tikzpicture}[baseline=-.1cm]
	\pgfmathsetmacro{\width}{.8};
	\pgfmathsetmacro{\height}{.8};
	\pgfmathsetmacro{\boxSize}{.5};
	\coordinate (a) at (0,-\height);
	\coordinate (b) at (0,\height);
	\coordinate (c) at ($ (\width,0) + 3/4*(\boxSize,0)$);

	\draw[thick, rho] (b) -- (c) -- (a) -- (b);
	\draw[thick, rho] (b) arc (0:180:{1/2*\width}) -- ($(a)-(\width,0)$) arc (-180:0:{1/2*\width});
	\draw[thick, rho] ($ (b) + (\boxSize,0) $) arc (90:-90:{\height});
	\ncircle{unshaded}{(b)}{\boxSize}{180}{$p_{k^{-1}\ell}$}
	\ncircle{unshaded}{(c)}{.4}{115}{$p_1$}
	\ncircle{unshaded}{(a)}{\boxSize}{180}{$p_{k^{-1}h}$}
	
\end{tikzpicture}\,.
$$
When $h\neq k\neq \ell$, expanding $p_1$ using Equation \eqref{eq:V} and simplifying, we get
$$
\|\beta_{h,k,\ell}\|^2_2 = \left(\frac{[2]}{[3]-1}\right)\left(\frac{[3]^2}{[4]} - \delta_{h,\ell}\frac{[3]}{[2]}\right)
=
\begin{cases}
\displaystyle \frac{[2][3]^2}{([3]-1)[4]}& \text{if $h\neq \ell$}
\\
\displaystyle \frac{[2]^2[3]}{([3]-1)[2][4]} & \text{if $h=\ell$}
\end{cases}
=
\begin{cases}
\displaystyle \frac{[3]^2}{([3]-1)^2}& \text{if $h\neq \ell$}
\\
\displaystyle \frac{[3]}{([3]-1)^2} & \text{if $h=\ell$.}
\end{cases}
$$
A similar calculation handles the cases $h=k\neq \ell$ and $h\neq k=\ell$.

Finally, for the case $h=k=\ell$, the following relation derived using Equation \eqref{eq:V} is helpful:
$$
\begin{tikzpicture}[baseline = -.1cm]
	\draw[thick, unshaded] (-.6,-.4)--(-.6,.4)--(.6,.4)--(.6,-.4)--(-.6,-.4);
	\coordinate (a) at (0,.1) ;
	\coordinate (b) at (-.2,-.2) ;
	\coordinate (c) at (.2,-.2) ;
	\filldraw[rho] (a) circle (.05cm);
	\filldraw[rho] (b) circle (.05cm);
	\filldraw[rho] (c) circle (.05cm);
	\draw[thick, rho] (-.4,-.4) -- (b) -- (c) -- (.4,-.4);
	\draw[thick, rho] (b) -- (a) -- (c);	
	\draw[thick, rho] (a) -- (0, .4);
\end{tikzpicture}
=
\left( [2] - \frac{3}{[2]}\right)\,
\begin{tikzpicture}[baseline = -.1cm]
	\draw[thick, unshaded] (-.4,-.4)--(-.4,.4)--(.4,.4)--(.4,-.4)--(-.4,-.4);
	\coordinate (a) at (0,0) ;
	\filldraw[rho] (a) circle (.05cm);
	\draw[thick, rho] (-.2,-.4) -- (a) --  ( .2, -.4);	
	\draw[thick, rho] (a) -- (0, .4);
\end{tikzpicture}
\,.
$$
Again using $\Psi_g$, we see $c_{h,h,h}$ is a multiple of the inner product of two triangles:
$$
c_{h,h,h}= 
\frac{1}{[3]}
 \left\langle \,
\begin{tikzpicture}[baseline = -.1cm]
	\draw[thick, unshaded] (-.6,-.4)--(-.6,.4)--(.6,.4)--(.6,-.4)--(-.6,-.4);
	\coordinate (a) at (0,.1) ;
	\coordinate (b) at (-.2,-.2) ;
	\coordinate (c) at (.2,-.2) ;
	\filldraw[rho] (a) circle (.05cm);
	\filldraw[rho] (b) circle (.05cm);
	\filldraw[rho] (c) circle (.05cm);
	\draw[thick, rho] (-.4,-.4) -- (b) -- (c) -- (.4,-.4);
	\draw[thick, rho] (b) -- (a) -- (c);	
	\draw[thick, rho] (a) -- (0, .4);
\end{tikzpicture}
\,, 
\begin{tikzpicture}[baseline = -.1cm]
	\draw[thick, unshaded] (-.6,-.4)--(-.6,.4)--(.6,.4)--(.6,-.4)--(-.6,-.4);
	\coordinate (a) at (0,.1) ;
	\coordinate (b) at (-.2,-.2) ;
	\coordinate (c) at (.2,-.2) ;
	\filldraw[rho] (a) circle (.05cm);
	\filldraw[rho] (b) circle (.05cm);
	\filldraw[rho] (c) circle (.05cm);
	\draw[thick, rho] (-.4,-.4) -- (b) -- (c) -- (.4,-.4);
	\draw[thick, rho] (b) -- (a) -- (c);	
	\draw[thick, rho] (a) -- (0, .4);
\end{tikzpicture}
\,\right\rangle
=
\left( [2] - \frac{3}{[2]}\right)^2
=
\left( \frac{[3]-2}{[2]} \right)^2.
$$

To prove $\gamma_{h,k,\ell}\neq 0$, a straightforward calculation shows $\|\alpha_{h,\ell}\|_2^2 = [3]^{-1}$, which implies
\begin{align*}
\|c_{h,k,\ell}\alpha_{h,\ell}\|_2^2 
&= 
\frac{|c_{h,k,\ell}|^2}{[3]}
\neq
\|\beta_{h,k,\ell}\|_2^2.
\qedhere
\end{align*}
\end{proof}

\begin{prop}
When $|G|$ is odd, $\{\gamma_{h,k,\ell}=\beta_{h,k,\ell} - c_{h,\ell}\alpha_{h,\ell}\}$ is a basis of $\cS_{3,+}\ominus \cI_{3,+}$.
\end{prop}
\begin{proof}
By counting dimensions, it suffices to show linear independence.
Suppose we have a linear combination 
$$
0=\sum_{h,k,\ell \in G} \lambda_{h,k,\ell} \gamma_{h,k,\ell}.
$$
Compress by a particular $p_h$ and $p_\ell$ on the bottom and top respectively to get
$$
0=\sum_{k \in G} \lambda_{h,k,\ell} \gamma_{h,k,\ell}.
$$
Attaching by $p_g$ on the right hand side, we see for every $g\in G$,
$$
0=
\sum_{k\in G} \lambda_{h,k,\ell} 
\begin{tikzpicture}[baseline=-.1cm]
	\draw[thick, red] (-.2,-.7) -- (-.2,.7);
	\draw[thick, red] (0,-.7) -- (0,.7);
	\draw[thick, red] (.2,.4) arc (180:0:.3cm);
	\draw[thick, red] (.2,-.4) arc (-180:0:.3cm);
	\draw[thick, red] (1.2,-.7) -- (1.2,.7);
	\nbox{unshaded}{(0,0)}{.4}{.1}{.1}{$\gamma_{h,k,\ell}$}
	\nbox{unshaded}{(1,0)}{.4}{0}{0}{\rotatebox{-90}{$p_{g}$}}
	\node at (1,.55) {$\star$};
\end{tikzpicture}
=
\lambda_{h,g,\ell} \beta_{h,g,\ell} - 
\sum_{k\in G} \lambda_{h,k,\ell}c_{h,\ell} \beta_{h,g,\ell} 
=
\left(\lambda_{h,g,\ell}-\sum_{k\in G} \lambda_{h,k,\ell}c_{h,\ell}\right)\beta_{h,g,\ell}.
$$
Hence $\lambda_{h,g,\ell} = \sum_{k\in G} \lambda_{h,k,\ell} c_{h,\ell}$ is independent of $g\in G$.
Denote this common value by $\lambda_{h,\ell}$.
We now see that $(1-|G| c_{h,\ell})\lambda_{h,\ell}=0$, which implies $\lambda_{h,\ell}=0$.
(A straightforward calculation using Remark \ref{rem:GSize} and Lemma \ref{lem:Nonzero} shows $c_{h,\ell}\neq |G|^{-1}$.)
\end{proof}

\begin{cor}\label{cor:BasisOfS3}
When $|G|$ is odd, $\cB_{3,+}=\{\alpha_{i,j}\}_{i,j\in G\cup\{\emptyset\}}\cup\{\gamma_{h,k,\ell}\}_{h,k,\ell\in G}$ is a basis of $\cS_{3,+}$.
\end{cor}

Recall that $\cQ_\bullet$ is the planar subalgebra of $\cR_\bullet$ generated by $\cR_{2}$. 
We now prove Conjecture \ref{conj:GeneratedByTwoBoxes} for the case when $|G|$ is odd.

\begin{thm}\label{thm:GeneratedByTwoBoxes}
When $|G|$ is odd, $\cQ_\bullet=\cR_\bullet$.
\end{thm}
\begin{proof}
First, identifying $\cS_{3,+}$ with $\cR_3$, we note that every element in $\cB_{3,+}$ is in $\cQ_{3}$.
Using Wenzl's generalized relation \cite[Section 2.3]{1308.5656}, we see that there is a basis of $\cR_{4}$ of elements in $\cQ_{4}$, since every element in the basic construction ideal $\cI_{4}=\cR_{3}e_3\cR_{3}$ is in $\cQ_{4}$.
Since $\cR_\bullet$ has depth 4, all further box spaces are equal to the corresponding basic construction ideals, and thus $\cQ_\bullet=\cR_\bullet$.

We give an alternate argument.
Since $\cB_{3,+}\subset \cQ_{3}$, $\cQ_{3}=\cR_{3}$, so the principal graphs agree to depth 3.
By counting dimensions, there is only one way for this graph to terminate, so the principal graphs of $\cQ_\bullet$ and $\cR_\bullet$ are equal, and $\cQ_\bullet=\cR_\bullet$.
\end{proof}

\begin{thm*}[Theorem \ref{thm:YB}]
For $|G|$ odd, $\cS_\bullet$ is a Yang-Baxter planar algebra with $|G|-1$ generators.
\end{thm*}
\begin{proof}
By Theorem \ref{thm:GeneratedByTwoBoxes}, $\cS_\bullet$ is generated by 2-boxes. 
For $i,j\in G\cup\{\emptyset\}$ and $h,k,\ell\in G$, we may similarly define elements by considering
$$
\eta_{i,j}=
\frac{1}{\Tr(p_i)^{1/2}\Tr(p_j)^{1/2}}
\begin{tikzpicture}[baseline=-.1cm, xscale=-1]
	\pgfmathsetmacro{\height}{.8};
	\coordinate (a) at ($ 3/4*(0,-\height) $);
	\coordinate (b) at ($ 3/4*(0,\height) $);

	\fill[shaded] ($(a)+3/4*(0,-\height)+(.45,0)$) -- ($(a)+(.45,.3)$) arc (0:180:.15cm) -- ($(a)+3/4*(0,-\height)+(.15,0)$);
	\fill[shaded] ($(b)+3/4*(0,\height)+(.45,0)$) -- ($(b)+(.45,-.3)$) arc (0:-180:.15cm) -- ($(b)+3/4*(0,\height)+(.15,0)$);
	\fill[shaded] ($(a)+3/4*(0,-\height)+(-.15,0)$) -- ($(a)+3/4*(0,-\height)+(-.45,0)$) -- ($(b)+3/4*(0,\height)+(-.45,0)$) -- ($(b)+3/4*(0,\height)+(-.15,0)$) ;
	
	\draw[thick, rho] ($(a)+(-.15,0)+3/4*(0,-\height)$) -- ($(b)+(-.15,0)+3/4*(0,\height)$);
	\draw[thick, rho] ($(a)+(.15,0)+3/4*(0,-\height)$)--($(a)+(.15,.3)$) arc (180:0:.15cm) -- ($(a)+(.45,0)+3/4*(0,-\height)$);
	\draw[thick, rho] ($(b)+(.15,0)+3/4*(0,\height)$)--($(b)+(.15,-.3)$) arc (-180:0:.15cm) -- ($(b)+(.45,0)+3/4*(0,\height)$);
	\nbox{unshaded}{(b)}{.3}{0}{0}{$p_j$}
	\nbox{unshaded}{(a)}{.3}{0}{0}{$p_i$}
\end{tikzpicture}
\text{ and }
\xi_{h,k,\ell} = 
\begin{tikzpicture}[xscale=-1, baseline=-.1cm]
	\pgfmathsetmacro{\width}{.8};
	\pgfmathsetmacro{\height}{.8};
	\coordinate (a) at ($ 3/4*(0,-\height) $);
	\coordinate (b) at ($ 3/4*(0,\height) $);
	\coordinate (c) at (\width,0);

	\fill[shaded] ($(a)+3/4*(0,-\height)+ (\width,0)$) -- ($(b)+3/4*(0,\height)+ (\width,0)$) -- ($(b)+3/4*(0,\height) + 1/2*(\width,0)$) -- (b) -- (c) -- (a) -- ($(a)+3/4*(0,-\height) + 1/2*(\width,0)$);
	\fill[shaded] ($(a)+3/4*(0,-\height)$) -- ($(b)+3/4*(0,\height)$) -- ($(b)+3/4*(0,\height)+1/2*(-\width,0)$) -- ($(a)+3/4*(0,-\height)+1/2*(-\width,0)$);

	\draw [thick, rho](b) -- (c) -- (a);
	\draw[thick, rho] ($(a)+3/4*(0,-\height)$) -- ($(b)+3/4*(0,\height)$);
	\draw[thick, rho] ($(a)+3/4*(0,-\height)+(\width,0)$) -- ($(b)+3/4*(0,\height)+(\width,0)$);
	\draw[thick, rho] (a) -- ($(a)+3/4*(0,-\height) + 1/2*(\width,0)$);
	\draw[thick, rho] (b) -- ($(b)+3/4*(0,\height) + 1/2*(\width,0)$);
	\ncircle{unshaded}{(b)}{.3}{73}{$p_\ell$}
	\ncircle{unshaded}{(c)}{.3}{0}{$p_k$}
	\ncircle{unshaded}{(a)}{.3}{65}{$p_h$}
\end{tikzpicture}\,.
$$
Similar calculations show $\cB_{3,+}'=\{\eta_{i,j}\}_{i,j\in G\cup\{\emptyset\}}\cup\{\zeta_{h,k,\ell}=\xi_{h,k,\ell} - c_{h,\ell}\eta_{h,\ell}\}_{h,k,\ell\in G}$ is a basis for $\cS_{3,+}\ominus \cI_{3,+}$.

Now for all $h,k,\ell\in G$, $\xi_{h,k,\ell}\in\spann(\cB_{3,+})$ and $\beta_{h,k,\ell}\in \spann(\cB_{3,+}')$, which gives the necessary relations to show $\cS_\bullet$ is a Yang-Baxter planar algebra.
\end{proof}

\begin{rem}
At this point, it seems highly non-trivial to compute the structure constants $\langle \beta_{h,k,\ell}, \xi_{h',k',\ell'}\rangle$.
\end{rem}

%%%%%%%%%%%%%%%%%%%%%%%%%%%%%%%%%%%%%%%%%%%%%%%%%
\bibliographystyle{amsalpha}
{\footnotesize{
\bibliography{../../bibliography}
}}
\end{document}